\documentclass{amsart}
\pdfoutput=1

\usepackage{amsmath, mathtools}
\usepackage{amsfonts}
\usepackage{amssymb}
\usepackage{amsthm}
\usepackage{tikz-cd}
\usepackage{tikz}
\usepackage{amsxtra}
\usepackage{wrapfig}
\usepackage{color}
\usepackage{url}
\usepackage[hidelinks]{hyperref}
\usepackage[percent]{overpic}
\usepackage{mathrsfs}
\usepackage{stmaryrd}

\usepackage{setspace}

\newtheorem{thm}{Theorem}[subsection]
\newtheorem{prop}[thm]{Proposition}

\newtheorem{cor}[thm]{Corollary}

\newtheorem{lem}[thm]{Lemma}
\theoremstyle{definition}
\newtheorem{defn}[thm]{Definition}

\newtheorem*{thma}{Theorem A}
\newtheorem*{thmb}{Theorem B}
\newtheorem*{thmc}{Theorem C}
\newtheorem*{thmd}{Theorem D}
\newtheorem*{cora}{Corollary A}
\newtheorem*{corb}{Corollary B}
\newtheorem*{corc}{Corollary C}

\newtheorem{conj}[thm]{Conjecture}
\theoremstyle{remark}
\newtheorem{ex}[thm]{Example}
\newtheorem{rmk}[thm]{Remark}

\usepackage{textcomp}
\DeclareSymbolFont{ugrf@m}{U}{eur}{m}{n}
\DeclareMathSymbol{\upmu}{\mathord}{ugrf@m}{"16}

\newcommand{\cat}[1]{{\mathbf{#1}}}
\newcommand{\p}{}
\newcommand{\spec}{\operatorname{Spec}}
\newcommand{\onto}{\twoheadrightarrow}
\newcommand{\into}{\hookrightarrow}
\newcommand{\from}{\leftarrow}
\newcommand{\lot}{\otimes^{\mathbb{L}}}
\newcommand{\per}{{\ensuremath{\cat{per}}}\kern 1pt}

\DeclareMathOperator{\id}{id}
\let\im\relax\DeclareMathOperator{\im}{im}
\let\ker\relax\DeclareMathOperator{\ker}{ker}

\let\hom\relax\newcommand{\hom}{\mathrm{Hom}}
\newcommand{\enn}{\mathrm{End}}
\DeclareMathOperator{\tor}{Tor}
\DeclareMathOperator{\ext}{Ext}

\newcommand{\Z}{\mathbb{Z}}
\newcommand{\N}{\mathbb{N}}

\renewcommand{\P}{\mathbb{P}}
\newcommand{\R}{{\mathrm{\normalfont\mathbb{R}}}}

\numberwithin{equation}{section}

\newcommand{\con}{\mathrm{con}}

\newcommand{\dq}{\ensuremath{A/^{\mathbb{L}}\kern -2pt AeA} }
\newcommand{\dqb}{\ensuremath{B/^{\mathbb{L}}\kern -2pt BeB} }
\newcommand{\thick}{\ensuremath{\cat{thick} \kern 0.5pt}}

\newcommand{\dgh}{\underline{\hom}}

\newcommand{\rmap}{\R\mathrm{Map}}

\newcommand{\del}{\text{\raisebox{.15ex}{$\mathscr{D}$}}\mathrm{el}}
\newcommand{\sdel}{\underline{\del}}
\newcommand{\defm}{\text{\raisebox{.15ex}{$\mathscr{D}$}}\mathrm{ef}}
\newcommand{\sdefm}{\underline{\defm}}
\newcommand{\mcs}{\mathrm{MC}}
\newcommand{\mc}{\mathscr{M}\kern -0.7pt C}
\newcommand{\smc}{\underline{\mc}}
\newcommand{\ggr}{\mathscr{G}\kern -1pt g}

\newcommand{\pdf}{{\ensuremath{\Omega(\Delta^\bullet)}}}
\newcommand{\sset}{\cat{sSet}}
\newcommand{\prodef}{\widehat{\sdefm}}
\newcommand{\frmdef}{\sdefm^{\mathrm{fr}}}
\newcommand{\frmdefset}{\defm^{\mathrm{fr}}}
\newcommand{\profrmdef}{\prodef^{\text{\raisebox{-1.1ex}{$\mathrm{fr}$}}}}
\newcommand{\prodefset}{\widehat{\defm}}
\newcommand{\profrmdefset}{\prodefset^{\text{\raisebox{-1.1ex}{$\mathrm{fr}$}}}}

\makeatletter
\newcommand{\holim@}[2]{%
	\vtop{\m@th\ialign{##\cr
			\hfil$#1\operator@font holim$\hfil\cr
			\noalign{\nointerlineskip\kern1.5\ex@}#2\cr
			\noalign{\nointerlineskip\kern-\ex@}\cr}}%
}
\newcommand{\holim}{%
	\mathop{\mathpalette\holim@{\leftarrowfill@\textstyle}}\nmlimits@
}
\newcommand{\hocolim}{%
	\mathop{\mathpalette\holim@{\rightarrowfill@\textstyle}}\nmlimits@
}
\makeatother

\usepackage[utf8]{inputenc}
\usepackage[T1]{fontenc}
\usepackage{textcomp}

\allowdisplaybreaks

\calclayout

\newcommand{\dgc}{\cat{dgc}_k}
\newcommand{\cndgc}{\cat{con.dgc}_k}

\newcommand{\fdcndgc}{\cat{fd\text{-}con.dgc}_k}

\newcommand{\vect}{\cat{dgvect}_k}
\newcommand{\fdvect}{\cat{fd\text{-}dg vect}_k}

\newcommand{\dgart}{\cat{dgArt}_k^{\leq 0}}
\newcommand{\dga}{\cat{dga}_{k}^{\leq 0}}
\newcommand{\proart}{{\cat{pro}(\cat{dgArt}_k^{\leq 0})}}

\newcommand{\ubdgart}{\cat{dgArt}_k}
\newcommand{\ubdga}{\cat{dga}_{k}}
\newcommand{\ubproart}{{\cat{pro}(\cat{dgArt}_k)}}

\newcommand{\augdga}{\cat{aug.dga}_{k}}

\newcommand{\scat}{\cat{sSetCat}}
\newcommand{\dgcat}{\cat{dgCat}}
\newcommand{\grp}{\cat{Grp}}
\newcommand{\grpd}{\cat{Grpd}}

\usepackage[left, final]{showlabels}

\makeatletter
\@namedef{subjclassname@2020}{\textup{2020} Mathematics Subject Classification}
\makeatother

\begin{document}
	
		\title{The derived deformation theory of a point}
		
\author{Matt Booth}

\address{Universiteit Antwerpen,
	Departement Wiskunde-Informatica,
	Campus Middelheim,
	Middelheimlaan 1,
	2020 Antwerpen,
	Belgium}

\email{matt.booth@uantwerpen.be}

\urladdr{mattbooth.info}


	
	\begin{abstract}
		We provide a prorepresenting object for the noncommutative derived deformation problem of deforming a module $X$ over a differential graded algebra. Roughly, we show that the corresponding deformation functor is homotopy prorepresented by the dual bar construction on the derived endomorphism algebra of $X$. We specialise to the case when $X$ is one-dimensional over the base field, and introduce the notion of framed deformations, which rigidify the problem slightly and allow us to obtain derived analogues of the results of Ed Segal's thesis. Our main technical tool is Koszul duality, following Pridham and Lurie's interpretation of derived deformation theory. Along the way we prove that a large class of dgas are quasi-isomorphic to their Koszul double dual, which we interpret as a derived completion functor; this improves a theorem of Lu--Palmieri--Wu--Zhang. We also adapt our results to the setting of multi-pointed deformation theory, and furthermore give an analysis of universal prodeformations. As an application, we give a deformation-theoretic interpretation to Braun--Chuang--Lazarev's derived quotient.
\end{abstract}

	\maketitle

\section{Introduction}
Let $k$ be a field of characteristic zero. We study the noncommutative derived deformation theory of modules over dg-$k$-algebras, with a focus on explicitly finding (pro)representing objects. Classically, commutative deformation theory studies certain functors $\cat{cArt}_k \to \cat{Set}$, where $\cat{cArt}_k$ denotes the category of commutative Artinian local $k$-algebras. The value of a deformation functor on such an algebra $\Gamma$ is supposed to behave like the set of deformations of some geometric object $X$ over $\Gamma$: i.e.\ those objects $\mathcal{X}$, flat over $\spec(\Gamma)$, whose base change along the point $\Gamma \to k$ is $X$. In this setting, deformation functors are rarely representable, but instead they are prorepresentable, meaning representable by an object of the procategory $\cat{pro}(\cat{cArt}_k)$. One typically embeds this procategory in the category of commutative complete local $k$-algebras, and often classical prorepresentability statements are presented in this form.

\p Classical noncommutative deformation theory is similar: one studies set-valued functors on the category $\cat{Art}_k$ of noncommutative Artinian local $k$-algebras. To pass to the derived world, one makes two changes: firstly, one accepts connective noncommutative Artinian local dg-$k$-algebras as input to deformation functors. One thinks of such a dga $\Gamma$ as a noncommutative derived affine scheme\footnote{The connective hypothesis is to avoid stacky pathologies: loosely, one thinks of the connective direction of some derived-geometric object as recording derived phenomena, and one thinks of the nonconnective direction as recording stacky phenomena. More explicitly, our results fail badly if one allows nonconnective dgas as input; see \ref{connectivereasons} for an example.} given by some nilpotent thickening of $k$. Secondly, in the derived setting one wants prorepresentable functors to be both left exact and quasi-isomorphism invariant. As a result, we are forced to use not just sets, but simplicial sets; one can obtain set-valued functors by truncating to $\pi_0$. Consequently one can think of deformation functors as higher stacks. If $A$ is a $k$-algebra and $X$ is an $A$-module, we summarise the differences between the classical noncommutative and derived noncommutative deformation functors in the following table:

\begin{table}[h]
	\begin{tabular}{c|c|c}
		\textbf{Functor} & \textbf{Test objects} $\Gamma$ & \textbf{Output} \\ \hline
		classical $\defm^\mathrm{cl}_A(X)$ & noncommutative Artinian local $k$-algebras   & sets    \\ \hline
		derived $\sdefm_A(X)$   & connective noncommutative Artinian local $k$-dgas & simplicial sets
	\end{tabular}
\end{table}
Derived deformation functors are only really naturally defined up to weak equivalence of simplicial sets. Hence, one can only expect to obtain homotopy prorepresentability statements, as in e.g.\ \cite{jardinerep}. In other words, we expect to obtain statements of the form $\sdefm_A(X)(\Gamma)\simeq \rmap_{\proart}(P,\Gamma)$ for some pro-Artinian dga $P$, where $\rmap$ denotes the derived mapping space in a model category \cite[5.4.9]{hovey}. As for how we actually expect to get these statements, it should be through the machinery of Koszul duality. 

\p It is by now well known that commutative derived deformation theory, at least in characteristic zero, is a manifestation of the Koszul duality (a.k.a.\ bar-cobar duality) between the commutative and Lie operads. This theorem has a rich history: a famous letter of Deligne introduced the philosophy that differential graded Lie algebras (dglas for short) should control commutative deformation problems \cite{goldmanmillson}. Hinich viewed this through the lens of Koszul duality, via which coalgebras become important objects. The correspondence between commutative deformation problems and dglas was made precise by later work of Pridham \cite{unifying} and Lurie \cite[2.0.2]{luriedagx}. 

In a similar fashion, noncommutative derived deformation theory is effectively the Koszul self-duality of the associative operad \cite[3.0.4]{luriedagx}, and Koszul duality will be one of our major technical tools. Because $\proart$ does not homotopy embed in a category of complete local algebras (unless one imposes extra finiteness conditions, as in \ref{dermapsthm}), it is now necessary for us to work with pro-Artinian dgas as opposed to their limits.

\p One goal of this paper is to obtain a derived result of a representability theorem of Segal \cite{segaldefpt} for deformations of one-dimensional modules. Typically, one is interested in deforming objects of derived or homotopy categories, which has been studied in detail by Efimov, Lunts, and Orlov \cite{ELO,ELO2,ELO3}. They give some prorepresentability theorems, although they do not identify the full simplicial set of deformations, and moreover they require some extra smoothness assumptions. Our first main theorem is the following:

\begin{thma}[\ref{proreps}]
		Let $A$ be a dga and let $X$ be an $A$-module. Let $E\coloneqq \R\enn_A(X)$ be the derived endomorphism dga of $X$. Then there is a functorial weak equivalence $$\sdefm_A(X)(\Gamma)\simeq\R\mathrm{Map}_{\ubproart}(B_{\mathrm{nu}}^\sharp E, \Gamma).$$
	\end{thma}

\p Here $B_\mathrm{nu}E$ refers to the `nonunital bar construction' on $E$ (\ref{nubar}), and $B^\sharp_\mathrm{nu}E$ is the continuous nonunital Koszul dual of $E$: it is a pro-Artinian dga whose limit is $B^*_\mathrm{nu}E$, the linear dual of $B_\mathrm{nu}E$ (see \ref{csharp} for the definition of the $\sharp$ functor). We regard Theorem A as a `nonconnective' version of a theorem of Lurie \cite[5.2.8]{luriedagx}, see \ref{lurie528} for the details. Implicit in Theorem A is the existence of a model structure on pro-Artinian algebras; setting up this model structure, and various related model structures, is the subject of $\S3$. The model structure itself, at least for connective pro-Artinian algebras, is due to Pridham \cite{unifying}. We can deduce from our prorepresentability theorem a representability theorem for classical deformations:

\begin{cora}[\ref{repclass}]
	Let $A$ be a $k$-algebra and let $X$ be an $A$-module.	Let $E\coloneqq \R\enn_A(X)$ be the derived endomorphism dga of $X$. Then there is a functorial isomorphism $$\defm^\mathrm{cl}_A(X)(\Gamma)\cong\hom_{\cat{aug.alg}_k}(H^0(B^*_{\mathrm{nu}}E), \Gamma).$$
	\end{cora}

\p Essentially, one obtains Corollary A by taking $\pi_0$ of Theorem A and applying an inclusion-truncation adjunction.

\p Our representability theorems are not yet the end of the story: they involve the somewhat unsatisfying nonunital bar construction. When $X$ is an $A$-module with $\dim_k(X)=1$ (we call such $X$ one-dimensional), its derived endomorphism algebra $E$ is augmented, and one expects the continuous Koszul dual $B^\sharp E$ to prorepresent some naturally defined deformation problem. We rigidify slightly by introducing the notion of framed deformations; a framed deformation of $S$ is essentially a deformation of $S$ that respects a fixed choice of isomorphism $S \cong k$ (\ref{framingdefsset}, \ref{framingdefset}). After defining a functor of framed deformations $\frmdef_A(S)$, we obtain the following theorem:

\begin{thmb}[\ref{prorepfrm}]
		Let $A$ be a connective dga and let $S$ be a one-dimensional $A$-module. Let $E\coloneqq \R\enn_A(S)$ be the derived endomorphism dga of $S$. Then there is a functorial weak equivalence$$\frmdef_A(S)(\Gamma)\simeq \R\mathrm{Map}_{\ubproart}(B^\sharp E, \Gamma).$$
	\end{thmb}

\p Here, $BE$ denotes the usual bar construction on $E$, so that the limit of $B^\sharp E$ is the usual Koszul dual $E^!\coloneqq B^*E$. We regard Theorem B as an extension of the deformation-theoretic results of Segal's thesis \cite{segaldefpt} to the derived setting: as before, taking $\pi_0$ of our prorepresentability theorem gives us a representability theorem for classical deformations, and in fact we obtain the following theorem of Segal \cite[2.13]{segaldefpt}. For a vector space $V$, we denote its tensor algebra by $T(V)$ and its completed tensor algebra by $\hat{T}(V)$.

\begin{corb}[\ref{mysegal}]
		Let $A$ be a $k$-algebra and let $S$ be a one-dimensional $A$-module. Assume that $\ext^1_A(S,S)$ is finite-dimensional. Let $T$ be the $k$-algebra$$T\coloneqq\frac{\hat{T}(\ext^1_A(S,S)^*)}{m^*(\ext^2_A(S,S)^*)}$$where $m:T(\ext^1_A(S,S)) \to \ext^2_A(S,S)$ is the homotopy Maurer--Cartan function (\ref{hmcdef}). Then there is a functorial isomorphism $$\defm^\mathrm{cl}_A(S)(\Gamma)\cong\frac{\hom_{\cat{aug.alg}_k}(T, \Gamma)}{ (\text{inner automorphisms of $\Gamma$})}.$$
	\end{corb}

\p Note that we are no longer using framed deformations. In fact, the functor of classical framed deformations is precisely $\hom_{\cat{aug.alg}_k}(H^0(E^!), -)$ (\ref{prorepfrmset}), and the quotient by inner automorphisms is precisely the thing we need to do to forget the framing (\ref{gensegal}). The $\ext$ condition is there to ensure that we have $H^0(E^!)\cong T$. Heuristically, the difference between the framed and unframed deformation functors is precisely the difference between the unital and nonunital bar construction, and on the set-valued level this presents itself as an extra action of the group of units of $\Gamma^\times$ by inner automorphisms.

\p We apply our prorepresentability theorems to give a deformation-theoretic interpretation to Braun--Chuang--Lazarev's derived quotient \cite{bcl}, which is a way to quotient an algebra by an idempotent in a homotopically well-behaved manner. Let $A$ be a $k$-algebra with an idempotent $e\in A$. Suppose that the quotient $A/AeA$ is an Artinian local $k$-algebra, and let $S$ be the quotient of $A/AeA$ by its radical. Then $S$ is a one-dimensional $A$-module. We show that in this situation, under some mild finiteness conditions the derived quotient $\dq$ homotopy \textbf{represents} the functor of derived deformations of $S$:

\begin{thmc}[\ref{almostdq}]
		Let $A$ be a $k$-algebra with an idempotent $e\in A$. Suppose that $A/AeA$ is an Artinian local $k$-algebra and let $S$ be the quotient of $A/AeA$ by its radical. Suppose that $\dq$ is cohomologically locally finite. Then there is a functorial weak equivalence 
	$$\frmdef_A(S)(\Gamma)\simeq \R\mathrm{Map}_{\augdga^{\leq 0}}(\dq, \Gamma).$$
	\end{thmc}

\p Surprising here is that neither Koszul duals nor pro-Artinian algebras seem to appear anywhere. The proof is based on a careful analysis of the Koszul double dual functor, which takes up the entirely of $\S4$. The most important ingredient is the following theorem - valid in any characteristic - which may be of independent interest:

\begin{thmd}[\ref{kdfin}]
		Let $A$ be a connective augmented cohomologically locally finite dga with $H^0(A)$ Artinian local. Then there is a natural quasi-isomorphism $A^{!!} \simeq A$.
	\end{thmd}

\p Theorem D is a generalisation of \cite[Theorem A]{lpwz}, which, while true, originally appeared with an incorrect proof: see \ref{lpwzremark} for details. Note that Theorem D really is nontrivial, as there even exist finite-dimensional graded algebras which are not quasi-isomorphic to their own Koszul double duals (\ref{lpwzwrong}). Our proof essentially identifies the Koszul double dual as the derived completion \cite{gmdc, dgcompletetorsion, efimov, psycentralizer, psydc, shauldc}, and then argues that dgas satisfying the hypotheses are derived complete. 

\p As usual, Theorem C specialises to the classical case and yields the following theorem, which appears in e.g.\ \cite{contsdefs}:

\begin{corc}[\ref{dwrep}]
		Let $A$ be a $k$-algebra with an idempotent $e\in A$. Suppose that $A/AeA$ is an Artinian local $k$-algebra and let $S$ be the quotient of $A/AeA$ by its radical. Suppose that $\dq$ is cohomologically locally finite. Then there is a functorial isomorphism $$\defm^\mathrm{cl}_A(S)(\Gamma)\cong\frac{\hom_{\cat{Art}_k}(A/AeA, \Gamma)}{ (\text{inner automorphisms of }\Gamma)}.$$
	\end{corc}

\p In $\S7$, we extend our deformation functors via homotopy limits to accept pro-Artinian dgas as input, which allows us to associate universal prodeformations to modules. The use of simplicial sets is integral here. We make some progress towards concretely identifying universal prodeformations; there are unanswered questions here which may be a fruitful subject for further research.

\p In $\S8$, we extend our results to the world of multi-pointed deformation theory \cite{laudalpt, eriksen}. Multi-pointed noncommutative deformations have recently found many applications within algebraic geometry, in particular the geometry of threefolds \cite{todatwists,DWncdf,kawamatapointed} and we anticipate that our multi-pointed results will be of use in geometric settings.

\p Some of our results have already been obtained under various extra smoothness hypotheses by various authors. We have already mentioned Efimov--Lunts--Orlov, who obtain a representability statement for groupoid-valued functors under an $\ext$-vanishing hypothesis \cite{ELO2}. Hua and Keller have also obtained a groupoid-valued multi-pointed representability theorem, under the additional assumption that $S$ is perfect \cite{huakeller}. A weaker version of Theorem D appears in \cite[\S2.5]{kalckyang3}, under the additional hypothesis that the $A$-module $k$ given by the augmentation $A \to k$ is perfect over $A$, and a similar version for dg categories appears in \cite{efimov}. The advantage of our approach is that we can identify the full simplicial set of deformations, and moreover do not need to assume any sort of smoothness hypothesis.

\p Much of this paper has close links to part X of Lurie's thesis \cite{luriedagx}, which also appears as \cite[IV]{luriesag}. In \textit{op. cit.} an axiomatic treatment of deformation theories is carried out via a very general notion of formal moduli problem and a representability statement for such problems is given \cite[1.3.12]{luriedagx}. This abstract result is then used to give a correspondence between $E_n$-formal moduli problems and augmented $E_n$-algebras \cite[4.0.8]{luriedagx}, which for $n=\infty$ specialises to the correspondence between dglas and commutative formal moduli problems \cite[2.0.2]{luriedagx} and for $n=1$ gives a representability statement for noncommutative formal moduli problems \cite[3.0.4]{luriedagx}, although a method is not given to explicitly construct prorepresenting objects. When both $X$ and $A$ are connective, our Theorem A is a consequence of \cite[5.2.8]{luriedagx}; see \ref{lurie528} for the details.

\p Profinite algebras inherit the inverse limit topology, and instead of working with pro-Artinian algebras one can instead work with topological algebras, specifically the pseudocompact ones. This is the context for \cite{glstreview}, who work with set-valued deformation functors, and also discuss dg Lie algebras as representing objects for commutative deformation problems. In this commutative setting, Yekutieli has set up the theory of Maurer--Cartan elements and Deligne groupoids for pronilpotent dglas in \cite{yekpronil}. Arguably, the simplest way to state our results would be as corepresentability theorems for coalgebras; this line of thought goes back to Hinich \cite{hinich}. 

\p After the Artin representability theorem, prorepresentability theorems for formal moduli problems in \textit{commutative} derived and spectral geometry have been obtained by Lurie \cite{luriedagxiv} (see also \cite[V.18]{luriesag}) and improved in the derived setting by Pridham \cite{jonrepdstacks}. Although these representability theorems are powerful, they pose two problems for us. Firstly, they are not explicit and do not provide concrete prorepresenting objects. Secondly, we care about \textit{noncommutative} deformations. We will not provide an axiomatic treatment of noncommutative derived deformation functors in this paper; see e.g.\ \cite{unifying} in the commutative derived case or \cite[3.0.3]{luriedagx} for a definition in the noncommutative case. See also \cite{pridhamNC} for the beginnings of noncommutative theory for simplicial rings.

\p One topic we do not discuss in this paper is the deformation theory of \textit{algebras}, as opposed to modules over them. Square-zero extensions of a connective dga (or more generally, an $E_k$-algebra in a suitable symmetric monoidal $\infty$-category) are controlled by its cotangent complex, as in \cite[\S7.4]{lurieha} (or \cite[V.17]{luriesag} in the geometric setting). For discrete rings, the cotangent complex is essentially the same thing as the Hochschild cohomology. We remark that some of our Koszul duality theorems can be proved via the cotangent complex formalism; see \ref{mythesisrmk}.

\p This paper is an extension and improvement of Part I of the author's PhD thesis \cite{me}. The author would like to express his gratitude to his supervisor, Jon Pridham, who patiently explained much of the modern approach to deformation theory to him. The author would also like to thank David Edwards, Bernhard Keller, Andrey Lazarev, Michael Wemyss and Amnon Yekutieli for useful discussions and suggestions which improved both the presentation and the content of the paper. The author is grateful to the anonymous referees for their careful reading and helpful comments.

		\section{Notation and conventions}
	Throughout this paper, $k$ will denote a field. In sections $3$ and $4$, $k$ is arbitrary, and in sections $5$ onwards, $k$ will be assumed to be characteristic zero. Occasionally we may in addition require $k$ to be algebraically closed; we will indicate where this is necessary. Modules are right modules, unless stated otherwise. Consequently, all notions are `right': e.g.\ `noetherian' means right noetherian. Unadorned tensor products are by default over $k$. We denote isomorphisms (of modules, functors, \ldots) with $\cong$ and weak equivalences with $\simeq$.
	
	\p We use cohomological grading conventions, so that the differential of a complex has degree $1$. 	If we refer to an object as just \textbf{graded}, then by convention we mean that it is $\Z$-graded. If $X$ is a complex, we will denote its cohomology complex by $H(X)$ or just $HX$. If $X$ is a complex, let $X[i]$ denote `$X$ shifted left $i$ times': the complex with $X[i]^j=X^{i+j}$ and the same differential as $X$, but twisted by a sign of $(-1)^i$. This sign flip can be worked out using the \textbf{Koszul sign rule}: when an object of degree $p$ moves past an object of degree $q$, one should introduce a factor of $(-1)^{pq}$. If $x$ is a homogeneous element of a complex of modules, we denote its degree by $|x|$.
	
	\p A \textbf{differential graded algebra} (\textbf{dga} for short) over $k$ is a monoid in the category of chain complexes of vector spaces. A $k$-algebra is equivalently a dga concentrated in degree zero, and a graded $k$-algebra is equivalently a dga with zero differential. We will sometimes refer to $k$-algebras as \textbf{ungraded algebras} to emphasise that they should be considered as dgas concentrated in degree zero. 
	
	\p A \textbf{dg module} (or just a \textbf{module}) over a dga $A$ is a complex of vector spaces $M$ together with an action map $M \otimes A \to M$ satisfying the obvious identities, or equivalently a dga map $A \to \enn_k(M)$. Note that a dg module over an ungraded ring is exactly a complex of modules. If $A$ is an algebra, write $\cat{Mod}\text{-}A$ for its category of right modules; it is a closed monoidal abelian category.
	
	\p Let $X$ be a complex of $k$-vector spaces. The \textbf{total dimension} or just \textbf{dimension} of $X$ is $\sum_{n\in \Z}\mathrm{dim}_k X^n$. Say that $X$ is \textbf{finite-dimensional} or just \textbf{finite} if its total dimension is finite. Say that $X$ is \textbf{locally finite} if each $\mathrm{dim}_kX^n$ is finite. Say that $X$ is \textbf{cohomologically locally finite} if the cohomology dg vector space $HX$ is locally finite. Say that $X$ is \textbf{bounded} if $X^n$ vanishes for all but finitely many $n$, \textbf{bounded above} if $X^n$ vanishes for all $n\gg 0$, and \textbf{bounded below} if $X^n$ vanishes for all $n \ll 0$. We use the same terminology in the case that $X$ admits extra structure.
	
	\p Say that a complex $X$ is \textbf{connective} if one has $X^i=0$ for $i>0$. Up to quasi-isomorphism, by taking the good truncation to nonpositive degrees it is enough to assume that $H^i(X)\cong0$ for $i>0$. Say that $X$ is \textbf{coconnective} if one has $X^i=0$ for $i<0$ (or equivalently $H^i(X)\cong0$ for $i<0$ up to quasi-isomorphism). We use the same terminology in case $X$ admits extra structure (e.g. that of a dga). Note that if $A$ is a connective dga in the weak sense then the good truncation map $\tau_{\leq 0 }A \into A$ is a dga quasi-isomorphism. The reason we use `connective' instead of `concentrated in nonpositive degrees' is that the former notion is not dependent on grading conventions. See e.g.\ \cite{yekutielibook} for further facts about dgas.
	
	\p We freely use terminology and results from the theory of model categories; see for example \cite{quillenHA, hovey, dwyerspalinski, riehlCHT} for references. We will in particular assume that the reader knows the basics of the homotopy theory of simplicial sets, and that a model category admits derived mapping spaces which are (weak equivalence classes of) simplicial sets \cite[5.4.9]{hovey}. See \cite{goerssjardine} for a comprehensive textbook account of simplicial homotopy theory.
	
	\p We will moreover assume that the reader has a good familiarity with the theory of triangulated and derived categories; see \cite{neemanloc} and \cite{weibel, huybrechts} respectively for references. In particular we will make use of the fact that the derived category of a dga is the homotopy category of a model category. By convention we use the projective model structure on dg-modules where every object is fibrant; this is the `q-model structure' of \cite{sixmodels}. Over an ungraded ring, the cofibrant complexes are all degreewise projective, and the converse is true for bounded above complexes \cite[1.7]{sixmodels}. We also assume familiarity with the theory of dg categories and simplicially enriched categories, although we provide a quick overview to set notation.

	\p We will make use of the following categories and model structures. We will not use multiple model structures on the same category.

	\begin{center}
		\begin{tabular}{c|c|c}
				\textbf{Category} & \textbf{Objects} & \textbf{Model structure} \\ \hline
			$\cat{Mod}$-$A$ & dg modules over a dga $A$& projective, e.g.\ \cite{sixmodels}\\
			$\vect $& dg $k$-vector spaces & projective \\
			$\ubdga $& dg algebras over $k$& Hinich \cite{hinichhom}\\
			$\augdga$ &augmented $k$-dgas& Hinich\\
			$\cat{Art}_k$  & ungraded Artinian algebras & trivial \\
			$\cat{pro}(\cat{Art}_k)$  & pro-ungraded Artinian algebras & trivial \\
			$\ubdgart $& Artinian local $k$-dgas& \\
			$\ubproart$ &pro-Artinian dgas & See \ref{ubpamodel} \\
			$\proart$& connective pro-Artinian dgas &Pridham \cite{unifying}\\
			$\dgc $& dg coalgebras over $k$ &\\
			$\cndgc$ & conilpotent dgcs over $k$ &As in \cite{positselski} or \cite{lefevre}\\
			$\cndgc^{\geq 0}$ & coconnective conilpotent dgcs over $k$ & See \ref{ccdgcmod}\\
			$\cat{Set}$& sets & trivial\\
			$\sset$& simplicial sets&Quillen, e.g.\ \cite{goerssjardine}\\
			$\grp$&groups & trivial\\
			$\grpd$& groupoids&Canonical, e.g.\ \cite{stricklandgpds}\\
			$\dgcat $& dg categories & Tabuada \cite{tabuadamodel}\\
			$\scat$& simplicially enriched categories&Bergner \cite{bergnermodelscat}\\ 
			\hline
			$\cat{fd}$-$\mathcal{C}$ & finite-dimensional objects in $\mathcal{C}$& \\
			$\cat{pro}(\mathcal{C})$ & pro-objects in $\mathcal{C}$&\\
			$\cat{ind}(\mathcal{C})$ & ind-objects in $\mathcal{C}$&\\
			$\mathcal{C}^{\leq 0}$ &connective objects in $\mathcal{C}$ &\\
			$\mathcal{C}^{\geq  0}$ &coconnective objects in $\mathcal{C}$ &\\
			
			\end{tabular}
		\end{center}
	
We summarise some of our functors of interest in the following commutative diagram. For readability, we denote functors $X \mapsto X^ \square$ by simply $\square$. The $*$ functor is the $k$-linear dual, and the $!$ functor is the Koszul dual. The $\sharp$ and $\circ$ functors, which are equivalences as well as Quillen equivalences, are defined in $\S\ref{padsc}$. Functors running to the right are right Quillen, and functors running to the left are left Quillen; the bar construction $B$ is a left adjoint as we are using opposite categories. The Koszul dual is not Quillen, although it does preserve weak equivalences.

$$
\begin{tikzcd}
 (\cndgc^{\geq 0})^\text{op} \ar[r,hook]	\ar[d,"\sharp", bend left=15]& \cndgc^\text{op} \ar[rd,"*"]\ar[d,"\sharp", bend left=15]& \augdga^\text{op} \ar[d,"!"] \ar[l,"B", swap] \\
\proart \ar[r,hook]\ar[u,"\circ", bend left=15]& \ubproart \ar[r,"\varprojlim", swap]\ar[u,"\circ", bend left=15]& \augdga  
\end{tikzcd}$$

	\section{Model structures on algebras and coalgebras}\label{modsn}
	In this section, $k$ is any field. We review the standard model structures on algebras and coalgebras, for which the bar-cobar adjunction is a Quillen equivalence. We also consider pro-Artinian algebras, which we think of as formal duals of coalgebras. New here is a model structure on unbounded pro-Artinian algebras which extends Pridham's \cite{unifying} and is distinct from the level model structures \cite{ehpromodels, ehcechsurvey, isaksenstrict}. We remark that if one wants to work with (co)commutative (co)algebras, then $k$ should be characteristic zero, since Pridham's model structure depends on the identification of connective dgas with simplicial algebras.

\subsection{Bar and cobar constructions}
We follow Positselski \cite{positselski}; for other references see Loday--Vallette \cite{lodayvallette} or Lef\`evre-Hasegawa's thesis \cite{lefevre}. 
\begin{defn}
	Let $\ubdga$ be the category of dgas over $k$. Let $\augdga$ be the category of \textbf{augmented} dgas over $k$; these are the dgas $A$ for which the unit map $k \to A$ admits a $k$-linear retraction. A morphism of augmented dgas must respect the augmentation.
	\end{defn}
The \textbf{augmentation ideal} of an augmented dga is $\bar A\coloneqq \ker(A \to k)$. Sending $A$ to $\bar A$ sets up an equivalence between augmented dgas and nonunital dgas. The inverse functor freely appends a unit, and indeed $A$ is isomorphic to $\bar{A}\oplus k$ as an augmented dga.

Just like a dga is a monoid in the monoidal category of dg vector spaces over $k$, a \textbf{differential graded coalgebra} (or \textbf{dgc} for short) is a comonoid in this category. More concretely, a dgc is a dg $k$-vector space $(C,d)$ equipped with a comultiplication $\Delta:C \to C \otimes C$ and a counit $\epsilon:C \to k$, satisfying the appropriate coassociativity and counitality identities, and such that $d$ is a coderivation for $\Delta$.  A \textbf{coaugmentation} on a dgc is a section of $\epsilon$; if $C$ is coaugmented then $\bar C\coloneqq  \ker \epsilon$ is the \textbf{coaugmentation coideal}. It is a dgc under the reduced coproduct $\bar{\Delta}x = \Delta x -x\otimes 1 -1\otimes x$, and $C$ is isomorphic as a nonunital dgc to $\bar C \oplus k$. A morphism of coaugmented coalgebras must respect the coaugmentation. A coaugmented dgc $C$ is \textbf{conilpotent} if every $x \in \bar C$ is annihilated by some suitably high power of $\Delta$.
\begin{defn}
	Let $\dgc$ denote the category of dg coalgebras. Let $\cndgc$ denote the category of conilpotent dg coalgebras.
	\end{defn}
Recall that if $V$ is a dg vector space then $V^*$ denotes its linear dual, which is also a dg vector space. Observe that if $(C,\Delta,\epsilon)$ is a dgc, then $\Delta$ dualises to a map $(C\otimes C)^* \to C^*$, and composing $\Delta^*$ with the natural inclusion $C^* \otimes C^* \to (C \otimes C)^*$ turns $(C^*,\Delta^*,\epsilon^*)$ into a dga. If $C$ is coaugmented, then $C^*$ is augmented. We will almost exclusively be interested in the category $\cndgc$.
\begin{defn}
	If $V$ is a dg vector space, then the tensor algebra ${T}^c(V)\coloneqq k\oplus V \oplus V^{\otimes 2} \oplus\cdots$ is a dg coalgebra when equipped with the \textbf{deconcatenation coproduct} ${T}^c(V) \to {T}^c(V) \otimes {T}^c(V)$ which sends $v_1\cdots v_n$ to $\sum_i v_1\cdots v_i \otimes v_{i+1}\cdots v_n$. The differential is induced from the differential on $V^{\otimes n}$.
\end{defn}
 It is easy to see that ${T}^c(V)$ is coaugmented and moreover conilpotent. Denote by $\bar{T}^c(V)$ the \textbf{reduced tensor coalgebra}: it is the coaugmentation coideal of the tensor coalgebra. The functor ${T}^c$ is the cofree conilpotent coalgebra functor: if $C$ is conilpotent then a morphism $C \to{T}^c(V)$ is determined completely by the composition $l:C \to {T}^c(V) \to V$. In particular, any morphism $f:\bar{T}^c(W) \to \bar{T}^c(V)$ is determined completely by its \textbf{Taylor coefficients} $f_n: W^{\otimes n} \to V$.

\begin{defn}
	Let $A$ be an augmented dga. Put $V\coloneqq \bar A [1]$, the shifted augmentation ideal. Let $d_V$ be the usual differential on the tensor coalgebra $T^cV$. Let $d_B$ be the \textbf{bar differential}: send $a_1\otimes\cdots \otimes a_n$ to the signed sum over $i$ of the $a_1\otimes\cdots\otimes a_ia_{i+1}\otimes\cdots \otimes a_n$ and extend linearly. The signs come from the Koszul sign rule; see \cite[2.2]{lodayvallette} for a concrete formula. One can check that $d_B$ is a degree 1 map from $V^{\otimes n+1} \to V^{\otimes n}$, and that it intertwines with $d_V$. Hence, one obtains a third and fourth quadrant bicomplex $C$ with rows $V^{\otimes{n}}[-n]$. By construction, the direct sum total complex of $C$ is $T^cV$, with a new differential $\partial=d_V+d_B$. The \textbf{bar construction} of $A$ is the complex $BA\coloneqq (T^cV,\partial)$. One can check that the deconcatenation coproduct makes $BA$ into a dgc.
\end{defn}
Note that the degree 0 elements of $A$ become degree $-1$ elements of $BA$.
\begin{rmk}
	If $A$ is an augmented $A_\infty$-algebra, then one can define the \textbf{$A_\infty$ bar construction} $B_\infty A$, which is a dgc, in an analogous manner (see \cite{lefevre} for a concrete formula). If $A$ is a dga then $B_\infty A = BA$.
\end{rmk}
\begin{ex}Let $A$ be the graded algebra $k[\epsilon]/\epsilon^2$, with $\epsilon$ in degree 0. Then $\bar A [1]$ is $k\epsilon$ placed in degree $-1$. Since $\epsilon$ is square zero, the bar differential is identically zero. So $BA$ is the tensor coalgebra $k[\epsilon]$, with $\epsilon$ in degree -1.
\end{ex}

\begin{defn}
	Let $C$ be a coaugmented dgc. One can analogously define a \textbf{cobar differential} $d_\Omega$ on the tensor algebra $T(\bar{C}[-1])$ by sending $c_1\otimes\cdots \otimes c_n$ to the signed sum over $i$ of the ${c_1\otimes\cdots \otimes  \bar{\Delta}c_i\otimes\cdots c_n}$, and the \textbf{cobar construction} on $C$ is the dga $\Omega C\coloneqq (T(\bar{C}[-1]),d_C+d_\Omega)$.
\end{defn}

Bar and cobar are adjoints:
\begin{thm}[{\cite[2.2.6]{lodayvallette}}]
	If $A$ is an augmented dga and $C$ is a conilpotent dgc, then there is a natural isomorphism $$\hom_{\augdga}(\Omega C, A)\cong \hom_{\cndgc}(C, BA).$$
\end{thm}
\begin{lem}
	The bar construction preserves quasi-isomorphisms.
\end{lem}
\begin{proof}
	The idea is to filter $BA$ by setting $F_p BA$ to be the elements of the form $a_1\otimes\cdots\otimes a_n$ with $n\leq p$, and look at the associated spectral sequence. A proof for dgas is in \cite[\S6.10]{positselski} and a proof for $A_\infty$-algebras is in \cite[Chapter 1]{lefevre}.
\end{proof}
\begin{rmk}The cobar construction does not preserve quasi-isomorphisms in general. For a concrete counterexample, put $C=B(k\oplus k)$. Then $C\simeq k$, and $\Omega k$ is acyclic, whereas $\Omega C$ is quasi-isomorphic to $k\oplus k$. A version of this example appears as \cite[2.4.3]{lodayvallette}.
\end{rmk}
\begin{thm}[{\cite[2.3.2]{lodayvallette}}]
	Let $A$ be an augmented dga. Then the counit $\Omega B A \to A$ is a quasi-isomorphism.
\end{thm}

\begin{thm}[Hinich \cite{hinichhom}]\label{hinichmodel}
	The category $\ubdga$ is a model category, with weak equivalences the quasi-isomorphisms and fibrations the degreewise surjections.
\end{thm}
Note that in $\ubdga$, every object is fibrant, and the cofibrant objects are the semifree dgas with `upper triangular' differential (recall that a semifree dga is one that becomes a free graded algebra after forgetting the differential).
\begin{prop}\label{hinichhomaug}
	The category $\augdga$ is a model category, with weak equivalences the quasi-isomorphisms and fibrations the levelwise surjections.
\end{prop}
\begin{proof}
	This is true because $\augdga$ is the slice category $\ubdga / k$ of dgas over $k$, and slice categories of model categories are themselves model categories under the same weak equivalences, fibrations, and cofibrations (see e.g.\ \cite{dwyerspalinski}).
\end{proof}

Note that if $C$ is a conilpotent coalgebra then $\Omega C$ is a cofibrant dga, and in particular we may think of $\Omega B$ as a cofibrant resolution functor on $\augdga$.

The failure of the cobar construction to preserve quasi-isomorphisms means that the `obvious' model structure on $\cndgc$, with weak equivalences quasi-isomorphisms, does not make $\Omega \dashv B$ into a Quillen adjunction. One way to fix this is to put a new model structure on $\cndgc$ whose weak equivalences are created by $\Omega$.
\begin{thm}[{\cite[Theorem 9.3b and Theorem 6.10]{positselski}}]\label{bcquillen}
	The category $\cndgc$ admits a model structure where the weak equivalences $f$ are those maps for which $\Omega f$ is a dga quasi-isomorphism, and the cofibrations are the degreewise monomorphisms. Moreover, if $C$ is a conilpotent dgc then the natural map $C \to B\Omega C$ is a fibrant resolution. The pair $(\Omega, B)$ is a Quillen equivalence between $\cndgc$ and $\augdga$.
\end{thm}
 Every object in $\cat{cndgc}_k$ is cofibrant, and the fibrant objects are the semicofree conilpotent coalgebras, i.e.\ those that are tensor coalgebras after forgetting the differential.
\begin{rmk}
	Every weak equivalence in $\cndgc$ is a quasi-isomorphism, but the converse is not true \cite[2.4.3]{lodayvallette}. One can think of the above model structure on $\cndgc$ as a `Bousfield delocalisation' of the obvious model structure, since it has less weak equivalences.
\end{rmk}
\begin{lem}\label{coconnlem}
	A quasi-isomorphism between coconnective coalgebras is a weak equivalence.
\end{lem}
\begin{proof}This is \cite[1.3.1.5(e)]{lefevre}. Take $f:C \to D$ to be a quasi-isomorphism where $C$ and $D$ are both coconnective. Then the point is that the shifted augmentation ideals $\bar{C}[-1]$ and $\bar{D}[-1]$ are concentrated in strictly positive degrees, which causes $f$ to be a filtered quasi-isomorphism; it is well known that such maps are weak equivalences (e.g.\ \cite[2.4.1]{lodayvallette}).
	\end{proof}

\subsection{Pro-Artinian dgas}\label{padsc}

\begin{defn}
	Say that a dga $A \in \augdga$ is \textbf{Artinian local} or just \textbf{Artinian} if:
	\begin{enumerate}
		\item $A$ has finite total dimension.
		\item the augmentation ideal $\bar A$ is nilpotent.
	\end{enumerate}
\end{defn}
Equivalently, a dga is Artinian local if it is finite-dimensional and has a unique two-sided maximal ideal $\mathfrak{m}$ which is nilpotent and has residue field $k$. Denote the category of Artinian local dgas by $\ubdgart$; it is a full subcategory of $\augdga$ in the obvious way. Let $\dgart$ denote the category of connective Artinian dgas.

Note that if $A$ is a finite-dimensional dga, then the multiplication dualises to a map $A^* \to A^* \otimes A^*$, and this makes $C\coloneqq A^*$ into a dgc. In this case, $A$ is augmented if and only if $C$ is coaugmented, and $A$ is Artinian if and only if $C$ is conilpotent.

\begin{defn}[e.g.\ {\cite[\S6]{kashschap}}]\label{procats}
	Let $\mathcal{C}$ be a category. A \textbf{pro-object} in $\mathcal{C}$ is a formal cofiltered limit, i.e.\ a diagram $J \to \mathcal{C}$ where $J$ is a small cofiltered category. We denote such a pro-object by $\{C_j\}_{j \in J}$. The category of pro-objects $\cat{pro}\mathcal{C}$ has morphisms $$\hom_{\cat{pro}\mathcal{C}}(\{C_i\}_{i \in I}, \{D_j\}_{j \in J}) \coloneqq  \varprojlim_j \varinjlim_i \hom_{\mathcal{C}}(C_i,D_j).$$
\end{defn}
If $\mathcal{C}$ has cofiltered limits, then there is a `realisation' functor $\varprojlim: \cat{pro}\mathcal{C} \to \mathcal{C}$. If $C$ is a constant pro-object, then it is easy to see that one has $\hom_{\cat{pro}\mathcal{C}}(C, \{D_j\}_{j \in J}) \cong  \hom_{\mathcal{C}}(C,\varprojlim_jD_j)$.
\begin{defn}
	Let $\mathcal{C}$ be a category. The \textbf{ind-category} of $\mathcal{C}$ is $\cat{ind}\mathcal{C}\coloneqq \cat{pro}(\mathcal{C}^\text{op})^\text{op}$. Less abstractly, an object of $\cat{ind}\mathcal{C}$ is a formal filtered colimit $J \to \mathcal{C}$, and the morphisms are $$\hom_{\cat{ind}\mathcal{C}}(\{C_i\}_{i \in I}, \{D_j\}_{j \in J}) \coloneqq  \varprojlim_i \varinjlim_j \hom_{\mathcal{C}}(C_i,D_j)$$
\end{defn}
If $\mathcal{C}$ has filtered colimits, then there is a `realisation' functor $\varinjlim: \cat{ind}\mathcal{C} \to \mathcal{C}$. In this situation, if $D \in \mathcal{C}$ is a constant ind-object then one has $\hom_{\cat{ind}\mathcal{C}}(\{C_i\}_{i \in I}, D)\cong \hom_\mathcal{C}( \varinjlim_iC_i,D)$. 
\begin{defn}
	We refer to an object of the procategory $\ubproart$ as a \textbf{pro-Artinian dga} and an object of the procategory $\proart$ as a \textbf{connective pro-Artinian dga}.
\end{defn}
We caution that for the majority of this paper, ``pro-Artinian'' is shorthand for ``pro-(Artinian local)''. Nonlocal profinite dgas will not be our main objects of study, although we remark on how to extend the theory to profinite multi-pointed dgas in $\S8$. We list some standard results on the structure of $\ubproart$:
\begin{prop}[{\cite[6.1.14]{kashschap}}]
	Let $\mathcal {C}$ be any category and let $f:A \to B$ be a morphism in $\cat{pro}(\mathcal{C})$. Then $f$ is isomorphic to a level map: a collection of maps $\{f_\alpha: A_\alpha \to B_\alpha\}_{\alpha \in I}$ between objects of $\mathcal{C}$, where $I$ is cofiltered.
\end{prop}
\begin{prop}[{\cite[Corollary to 3.1]{grothendieckpro}}]\label{strictpro}
	Every object of $\ubproart$ is isomorphic to a \textbf{strict} pro-object, i.e.\ one for which the transition maps are surjections.
\end{prop}
\begin{rmk}
	To apply \cite{grothendieckpro} to some general procategory $\cat{pro}(\mathcal C)$, it suffices that every object of $\mathcal{C}$ is Artinian. In particular, every finite-dimensional dga is Artinian, and so it follows that every profinite dga is isomorphic to a strict profinite dga.
\end{rmk}
\begin{defn}
	Let $A=\{A_\alpha\}_\alpha\in\ubproart$. Let $A^\circ$ denote the dgc $$A^\circ \coloneqq \varinjlim_\alpha A_\alpha^*\in \dgc.$$
\end{defn}The assignment $A \mapsto A^\circ$ is (contravariantly) functorial, since we can represent every map in $\ubproart$ by a level map. Because colimits of conilpotent dgcs are conilpotent, $A^\circ$ is a conilpotent dgc.

\begin{prop}\label{sharpprop}
	The functor $A \mapsto A^\circ$ gives an equivalence $$\ubproart\xrightarrow{\cong} \cndgc^\text{op}$$that restricts to an equivalence $$\proart\xrightarrow{\cong} (\cndgc^{\geq 0})^\text{op}.$$
\end{prop}
\begin{proof}
	Via the linear dual, an Artinian local dga is the same thing as a finite-dimensional coaugmented conilpotent dgc. Taking procategories thus gives an equivalence between $\ubproart^\text{op}$ and $\cat{ind}(\fdcndgc)$. A classical theorem of Sweedler says that an ungraded coalgebra is the filtered colimit of its finite-dimensional subcoalgebras. The same remains true for dgcs \cite[1.6]{getzlergoerss}. In particular, the colimit functor $\varinjlim: \cat{ind}(\fdcndgc) \to \cndgc$ is essentially surjective. By \cite[1.9]{getzlergoerss}, finite-dimensional conilpotent coalgebras are compact: given a finite-dimensional conilpotent dgc $C$, and $D=\{D_\alpha\}_\alpha$ a filtered system of finite-dimensional conilpotent dgcs, then there is a natural isomorphism $\varinjlim_\alpha\hom(C,D_\alpha)\xrightarrow{\cong} \hom(C,\varinjlim_\alpha D_\alpha)$. Hence if $C$ and $D$ are objects of $\cat{ind}(\fdcndgc)$, then one has an isomorphism $\hom(C,D)\cong\hom(\varinjlim C, \varinjlim D)$. Moreover, if $A$ and $B$ are finite-dimensional algebras, then one has an isomorphism $\hom(A,B)\cong \hom(B^*,A^*)$. Putting these together we see that $\varinjlim$ is fully faithful and hence an equivalence. Hence we obtain the first claim, and the second claim is clear.
\end{proof}
\begin{defn}\label{csharp}
	If $C$ is a conilpotent dgc, let $C^\sharp\in \ubproart$ denote the levelwise dual of its filtered system of finite-dimensional sub-dgcs.
\end{defn}
It is easy to see that $C \mapsto C^\sharp$ is the inverse functor to  $A \mapsto A^\circ$.
\begin{lem}\label{limsharplem}\hfill
	\begin{enumerate}\item Let $C \in \cndgc$. Then $C^*$ is functorially isomorphic to $ \varprojlim C^\sharp$ as dgas.
		\item Let $A \in \ubproart$. Then $A^{\circ*} $ is functorially isomorphic to $\varprojlim A$ as dgas.
		\end{enumerate}
	\end{lem}
\begin{proof}
For (1), pass the filtered limit through the linear dual and use that $C\cong \varinjlim C^{\sharp *}$, because $C^{\sharp *}$ is exactly the filtered system of finite-dimensional sub-dgcs of $C$. For (2),  just apply (1) with $C=A^\circ$ and use $A^{\circ\sharp}\cong A$. The isomorphisms are clearly functorial.
	\end{proof}

\begin{thm}\label{ubpamodel}
	The category $\ubproart$ is a model category, with weak equivalences those maps $f$ for which $f^\circ$ is a weak equivalence of coalgebras, and fibrations those maps for which $\varprojlim f$ is a degreewise surjection. The equivalence $\phi:\ubproart\xrightarrow{\cong} \cndgc^\text{op}$ is both left and right Quillen, and is moreover a Quillen equivalence.
\end{thm}
\begin{proof}
	Transferring the model structure on $\cndgc$ along $\phi$ puts a model category structure on $\ubproart$ such that $\phi$ becomes a Quillen equivalence. The functor $\phi$ is both left and right Quillen because a morphism $f$ in $\ubproart$ is a cofibration, resp. a weak equivalence, resp. a fibration, precisely when $\phi(f)$ is. It is clearly a Quillen  equivalence. So we need only check the statements about the weak equivalences and fibrations in $\ubproart$. It is clear that the weak equivalences in $\ubproart$ are exactly the claimed maps. For the fibrations, note that the cofibrations in $\cndgc$, and hence the fibrations in $\cndgc^\text{op}$, are exactly the degreewise injections. So we need to show that $\varprojlim f$ is a degreewise surjection if and only if $f^\circ$ is a degreewise injection. But this follows from \ref{limsharplem}(2).
\end{proof}

Recall that there is a limit functor $\varprojlim: \ubproart \to \ubdga$ which sends a cofiltered system to its limit. Note that this functor is right adjoint to the functor $\ubdga \to \ubproart$ which sends a connective dga $A$ to the cofiltered system $\hat A$ of its Artinian local quotients (because any map $A \to \Gamma$ with $\Gamma$ Artinian must factor through some such quotient). The unit map $ A \to \varprojlim\hat A$ is pro-Artinian completion, and neither the unit nor the counit need be an isomorphism. 

\begin{prop}\label{limisrq}
	The functor $\varprojlim: \ubproart \to \ubdga$ is right Quillen.
	\end{prop}
\begin{proof}Clearly $\varprojlim$ preserves fibrations; we check that it preserves weak equivalences. Let $f$ be a weak equivalence in $\ubproart$, so that $f^\circ$ is a weak equivalence of coalgebras. In particular, $f^\circ$ is a quasi-isomorphism, and because the linear dual is exact, $f^{\circ*}$ is also a quasi-isomorphism. But $f^{\circ*}$ canonically agrees with $\varprojlim f$ by \ref{limsharplem}(2). Hence $\varprojlim f$ is a quasi-isomorphism too.
	\end{proof}

\subsection{(Co)connective objects}
We restrict our attention to connective pro-Artinian dgas and coconnective conilpotent coalgebras, where weak equivalences become simpler.
\begin{thm}[Pridham]\label{proartmodel}
	The category $\proart$ is a model category, with weak equivalences those maps $f$ for which each $H^n f$ is an isomorphism of profinite $k$-vector spaces, and fibrations those maps $f$ for which $\varprojlim f$ is a degreewise surjection.
\end{thm}
\begin{proof}
	The proof of \cite[4.3]{unifying} adapts to the noncommutative case. The idea of the proof is to use the Dold--Kan correspondence to identify $\proart$ with the category of pro-objects in simplicial ungraded Artinian algebras. One then shows that this is equivalently the category of simplicial objects in pro-(ungraded Artinian algebras). One may as well put a model structure on the opposite category, which is the category of cosimplicial objects in noncommutative formal spectra. The result follows from an application of Bousfield's theory of injective models \cite{bousfieldresln} to put model structures on categories of cosimplicial objects.
\end{proof}
\begin{lem}\label{indfinite}
	Let $V\in \cat{pro}(\cat{fd-vect}_k)$ be a profinite vector space. The assignment $V \mapsto \varinjlim V^*$ is a contravariant equivalence $\cat{pro}(\fdvect) \to \vect$.
	\end{lem}
\begin{proof}
	The colimit functor $\varinjlim: \cat{ind}(\fdvect)\to \vect$ is an equivalence: it is essentially surjective because every vector space is the colimit of its finite-dimensional subspaces, and it is fully faithful because finite-dimensional vector spaces are compact in $\vect$. In other words, every vector space is canonically ind-finite. Clearly the linear dual is a contravariant autoequivalence of $\fdvect$, and the result follows.
	\end{proof}
\begin{rmk}
	The spectral algebra analogue of this result appears in \cite{ciprospec}.
\end{rmk}

\begin{prop}\label{inclisrq}
	The inclusion functor $\iota:\proart \into \ubproart$ is right Quillen.
\end{prop}
\begin{proof}
	The functor $\iota$ is right adjoint to the levelwise truncation functor. It clearly preserves fibrations. It will hence be enough for us to show that it preserves weak equivalences. So suppose that $f:A\to B$ is a weak equivalence in $\proart$. We can assume that $f$ is a level map $\{f_\alpha:A_\alpha \to B_\alpha\}_\alpha$, and then $f^\circ$ is the dgc map $\varinjlim_\alpha(f_\alpha^*)$. Fixing an $n\in \Z$, we have $H^n(f^\circ)\cong \varinjlim_\alpha H^n(f_\alpha^*)\cong \varinjlim_\alpha (H^{-n}(f_\alpha)^*)$, because filtered colimits and linear duals are exact. By \ref{indfinite}, we know that a morphism $g$ of profinite vector spaces is an isomorphism if and only if $\varinjlim g^*$ is. Applying this to $g=H^{-n}f$, we see that $H^n(f^\circ)$ is an isomorphism if and only if $H^{-n}f$ is. But by assumption $f$ was a weak equivalence, so that every $H^{-n}f$ is an isomorphism. Hence $f^\circ$ is a quasi-isomorphism of coalgebras. By \ref{coconnlem}, a quasi-isomorphism between coconnective coalgebras is a weak equivalence. So $f^\circ$ is a weak equivalence, which means precisely that $\iota f$ is a weak equivalence.
\end{proof}
One can push the above reasoning further to show that the weak equivalences in the procategory $\proart$ are created by $\varprojlim$; this will be useful to us later.
\begin{lem}\label{limreflectsisos}
	The functor $\varprojlim: \cat{pro}(\cat{fd-vect}_k) \to \vect$ reflects isomorphisms.
\end{lem}
\begin{proof}Let $g:U \to V$ be a map of profinite vector spaces. We may take $g$ to be a level map $\{g_\alpha:U_\alpha \to V_\alpha\}_\alpha$. By \ref{indfinite}, $g$ is an isomorphism if and only if $\varinjlim_\alpha g_\alpha^*: \varinjlim_\alpha V_\alpha^* \to \varinjlim_\alpha U_\alpha^*$ is an isomorphism. But this map is an isomorphism if and only if its linear dual $\varprojlim_\alpha U_\alpha^{**} \to \varprojlim_\alpha V_\alpha^{**}$ is an isomorphism. But this latter map canonically agrees with $\varprojlim g = \varprojlim_\alpha g_\alpha$ since the $U_\alpha$ and $V_\alpha$ are finite-dimensional. So $g$ is an isomorphism if and only if $\varprojlim g$ is.
\end{proof}
\begin{lem}\label{limisexact}
	The functor $\varprojlim: \ubproart \to \ubdga$ is the homotopy limit functor.
\end{lem}
\begin{proof}
	Let $P$ be any object of $\ubproart$. Without loss of generality, by \ref{strictpro} we may assume that $P$ is strict. Let $I$ be the indexing set of $P$. Because every cofiltered set $I$ has a cofinal directed subset $I' \into I$ \cite[Expos\'e 1, 8.1.6]{sga4}, we may without loss of generality assume that $I$ is directed, and in particular is a Reedy category. It is now easy to see that $P$ is a Reedy fibrant diagram of dgas, and hence $\varprojlim P \simeq \holim P$, because one can compute homotopy limits as usual limits along resolved diagrams.
\end{proof}
\begin{rmk}
	The above proof is a generalisation of the fact that $\varprojlim$ is exact when restricted to Mittag-Leffler systems: indeed a Mittag-Leffler pro-object is precisely a strict pro-object indexed by a countable directed set.
\end{rmk}
\begin{prop}\label{limreflects}
	The functor $\varprojlim: \proart \to \ubdga$ both preserves and reflects weak equivalences.
\end{prop}
\begin{proof}
	Let $f$ be a morphism in $\proart$. By definition, $f$ is a weak equivalence if and only if each $H^nf \in \cat{pro}(\cat{fd-vect}_k)$ is an isomorphism. But $\varprojlim$ reflects isomorphisms by  \ref{limreflectsisos} and so $H^nf$ is an isomorphism if and only if $\varprojlim H^n f$ is an isomorphism. But $\varprojlim$ is exact by \ref{limisexact}, and so $\varprojlim H^n f$ is an isomorphism if and only if $H^n\varprojlim f$ is an isomorphism.
\end{proof}
\begin{rmk}\label{notpseudormk}
	The category $\dga$ is a model category, with weak equivalences quasi-isomorphisms and fibrations the surjections in strictly negative degrees. The model structure is transferred from that on simplicial vector spaces by the monoidal Dold--Kan correspondence \cite{ssmon}. From the above, it is easy to see that $\varprojlim: \proart \to \dga$ is right Quillen. However, the inclusion $\dga \into \ubdga$ is not right Quillen as it does not preserve fibrations. Because we want to work with unbounded pro-Artinian dgas, we will not use the model structure on $\dga$.
\end{rmk}
We end this section by turning our attention to the coalgebra side of the story.

\begin{thm}\label{ccdgcmod}
	The category $\cndgc^{\geq 0}$ of coconnective conilpotent dgcs admits a model structure where the weak equivalences are the quasi-isomorphisms and the cofibrations are the degreewise injections. The inclusion functor $\cndgc^{\geq 0} \into \cndgc$ is left Quillen. The equivalence $\proart\xrightarrow{\cong} (\cndgc^{\geq 0})^\text{op}$ is a Quillen equivalence.
	\end{thm}
\begin{proof}
	This is similar to \ref{ubpamodel} but transfers the Pridham model structure to the category of coalgebras. Clearly this transfer yields the claimed Quillen equivalence. The proof of \ref{inclisrq} shows that a map $f$ of connective dgas is a quasi-isomorphism if and only if $f^\circ$ is, and hence that a map $g$ of coconnective dgcs is a quasi-isomorphism if and only if $g^\sharp $ is. Hence the weak equivalences are as claimed. The proof of \ref{ubpamodel} shows that a map $g$ is a degreewise injection if and only if $\varprojlim g^\sharp$ is a degreewise surjection. Hence the cofibrations are as claimed. It is easy to see that the inclusion is left Quillen since a quasi-isomorphism of coconnective dgcs is a weak equivalence. 
	\end{proof}

\section{Koszul duality}\label{kdsn}
In this section, $k$ is any field. We study the Koszul dual of a dga, which is the linear dual of the bar construction. For a wide class of dgas, we show that the Koszul double dual can be interpreted as a sort of derived completion functor. We show that connective cohomologically locally finite dgas are derived complete; later this will give us extremely good control over prorepresenting objects. We apply this to the problem of lifting quasi-isomorphisms of limits of pro-Artinian dgas to weak equivalences. 

	\subsection{The Koszul dual}\label{kdsn1}
	\begin{defn}Let $A$ be an augmented dga. The \textbf{Koszul dual} of $A$ is $A^!\coloneqq (BA)^*$.
	\end{defn}
Clearly $A^!$ is itself an augmented dga.  Because both $B$ and the linear dual preserve quasi-isomorphisms, so does $A \mapsto A^!$. Loosely, the differential $d(x^*)$ is the signed sum of the products $x_1^*\cdots x_r^*$ such that the $x_i$ satisfy $\partial(x_1 | \cdots | x_r)=x$, where $\partial$ is the bar differential. The Koszul dual $A^!$ is a completed semifree dga, in the sense that the underlying graded algebra is a completed free algebra.

\begin{prop}\label{kdisrend}
Let $A$ be a connective augmented dga, and let $S$ be the $A$-module $ k$ with $A$-action given by augmentation $A \to k$. Then there is a natural quasi-isomorphism of dgas $A^!\simeq \R\enn_A(S)$.
\end{prop}
\begin{proof}
This is standard and appears as \cite[\S19, Exercise 4]{fhtrht} or \cite[2.2.2]{lodayvallette} for rings. The loose idea is that by taking the bar resolution of $S$, the bar construction $BA$ becomes a model for the derived tensor product $S\lot_A S$, which is naturally a coalgebra. Taking the linear dual, one obtains the desired statement.
\end{proof}

Because $BA$ is a tensor coalgebra, it follows that $A^!$ is a completed tensor algebra. Under an additional finiteness hypothesis on $A$, we can identify the generators of $A^!$ as the duals of the cogenerators of $BA$. The finiteness hypothesis we will need is the following.
\begin{defn}
	Say that a complex $X$ is \textbf{raylike} if it is locally finite, and either bounded above or bounded below.
\end{defn}   
Clearly a locally finite bounded complex is raylike. It is also easy to see that if $X$ is raylike then so is the linear dual $X^*$. The key property of a raylike complex is that the sum appearing in $$(X^{\otimes n})^m\cong \bigoplus_{i_1+\cdots i_n = m}X^{i_1}\otimes\cdots\otimes X^{i_n}$$	is finite, and hence the tensor product $X^{\otimes n}$ is itself locally finite. Hence we have natural isomorphisms $(X^{\otimes n})^* \cong (X^*)^{\otimes n}$. In particular, the dual of a raylike dga is a dgc. 
\begin{defn}\label{ncocondefn}
	Say that a dga $A$ is \textbf{$1$-coconnective} if it is coconnective and $A^0\cong k$. Say that $A$ is \textbf{$2$-coconnective} if it is $1$-coconnective and $A^1\cong 0$.
	\end{defn}
In other words, a 2-coconnective dga has nothing (except the unit) in degrees $<2$.
\begin{rmk}
	For $n \in \Z$, one can define $n$-coconnective similarly: $0$-coconnective dgas are just coconnective dgas. For $n>1$, $n$-coconnective dgas are $1$-coconnective and are trivial in degrees $1,\ldots,n-1$. For $n<0$, an $n$-coconnective dga is simply one concentrated in degrees $\geq n$. One can make an analogous definition of $n$-connective dga.
	\end{rmk}
\begin{rmk}
	In the graded algebra literature, a $1$-coconnective graded algebra is typically called \textbf{connected}. In the higher algebra literature, $1$-connective dgas or ring spectra are often also called \textbf{connected}. In view of this we do not use `connected' in this context.
	\end{rmk}
Part (3) of the following is well known and appears as e.g.\ {\cite[\S19, Exercise 3]{fhtrht}}:
	\begin{prop}\label{artprop}
	Let $A$ be an augmented raylike dga. Then
	\begin{enumerate}
		\item There is a natural dga isomorphism $A^! \cong \widehat{\Omega}(A^*)$ where $\widehat{\Omega}$ denotes the completed cobar construction.
		\item The natural map $\Omega(A^*)\to A^!$ is completion.
		\item If $A$ is connective then there is a natural dga isomorphism $A^!\cong \Omega(A^*)$.
		\item If $A$ is 2-coconnective then there is a natural dga isomorphism $A^!\cong \Omega(A^*)$.
		\end{enumerate}
	 \end{prop}
	\begin{proof} For brevity write $V$ for the shifted augmentation ideal $V\coloneqq\bar A[1]$, which is also raylike. For now, forget the differential and think about the underlying graded vector spaces. The dgc $BA$ is the direct sum total complex of the double complex whose rows are $V^{\otimes n}$. Hence, $A^!$ is the direct product total complex of the double complex with rows $(V^{\otimes n})^*$. But because $V$ was raylike, $A^!$ is the direct product total complex of the double complex with rows $(V^*)^{\otimes n}$. But this is precisely the completed tensor algebra on $V^*$, which is by definition $\widehat{\Omega}(A^*)$. We have shown that $A^! \cong \widehat{\Omega}(A^*)$ as graded algebras. But is it not hard to check that the bar differential dualises to the cobar differential, and it now follows that $A^! \cong \widehat{\Omega}(A^*)$ as dgas, which is precisely the first claim. To see the second claim, just observe that the natural map $\Omega(A^*)\to A^!$ is induced by the identity on generators. For the third claim, first observe that if $A$ is connective then $V$ is concentrated in strictly negative degrees, from which it follows that in each degree $BA$ is a finite direct sum. Hence $A^!$ is also the direct sum total complex of the double complex with rows $(V^*)^{\otimes n}$. In other words, $A^!\cong \Omega(A^*)$ as dgas; observe that the completion map is an isomorphism. The fourth claim is completely analogous to the third: because $A$ is 2-coconnective, this time $V$ is concentrated in strictly positive degrees.
	\end{proof}
\begin{rmk}\label{lpwzwrong}
	The dga $A^!$ does not always agree with  $\Omega(A^*)$. Indeed, let $A$ be the graded algebra $\frac{k[\epsilon]}{\epsilon^2}$ with $\epsilon$ in degree $1$. Then $BA$ is the polynomial coalgebra $k[\epsilon]$ concentrated in degree zero, and hence $A^!$ is the complete polynomial algebra $k\llbracket x \rrbracket$ on an element of degree zero. However, $\Omega(A^*)$ is $k[x]$. As expected, the completion of $\Omega(A^*)$ is $A^!$.
\end{rmk}

When $BA$ is raylike, the double Koszul dual of $A$ is its derived completion:
	\begin{prop}\label{completekd}
		Let $A$ be an augmented dga such that $BA$ is raylike. Then
		\begin{enumerate}
			\item There is a natural isomorphism $A^{!!}\cong \widehat{\Omega}BA$.
			\item If $A^!$ is connective then there is a natural quasi-isomorphism $A^{!!} \simeq A$.
			\item If $A^!$ is 2-coconnective then there is a natural quasi-isomorphism $A^{!!} \simeq A$.
			\end{enumerate}
	\end{prop}
	\begin{proof}
Because $BA$ is raylike so is $A^!$. We can hence apply \ref{artprop}(1) to conclude that $A^{!!}\cong \widehat{\Omega}(A^{!*})$. But because $BA$ is locally finite, $A^{!*}$ is $BA$ again. Hence $A^{!!}\cong \widehat{\Omega}BA$ as required. For the second statement, we can instead apply \ref{artprop}(3) and proceed as above to conclude that we have an isomorphism $A^{!!}\cong {\Omega}BA$. But $\Omega BA$ is naturally quasi-isomorphic to $A$. The third statement is the same but uses \ref{artprop}(4) instead.
	\end{proof}
\begin{rmk}
	We refer to $\widehat{\Omega}BA$ as the derived completion of $A$ as it is the completion of a cofibrant resolution of $A$; a priori completion need not actually have a derived functor as we have not shown that it is left Quillen. Derived completion functors for commutative rings have been studied by a number of authors \cite{gmdc, psydc, shauldc}, and it is possible to show that in this setting, the (derived) completion of a noetherian commutative ring is its Koszul double dual \cite{dgiduality, psycentralizer}. In the noncommutative setting, Efimov \cite{efimov} defines completion of dg categories via derived double centralisers; in our setting this reduces to the double Koszul dual.
	\end{rmk}

\begin{rmk}
	Using the homotopy theory of $A_\infty$-coalgebras (as in \cite[\S6]{inftycoalg} or the more general \cite{inftycoalghard}), one can weaken the hypotheses of \ref{completekd}(1) to requiring only that $BA$ is cohomologically raylike. Indeed, if $A$ is a raylike $A_\infty$-algebra then $A^*$ is an $A_\infty$-coalgebra, and vice versa. There is a completed cobar construction $\widehat{\Omega}$ (but not necessarily an uncompleted one) for $A_\infty$-coalgebras, and it is not hard to prove \ref{artprop}(1) when $A$ is assumed to be a raylike $A_\infty$-algebra. Minimal models exist for $A_\infty$-coalgebras, so if  $BA$ is cohomologically raylike, take an $A_\infty$ minimal model $C$, which is genuinely raylike. Unlike the usual cobar construction, the completed cobar construction preserves quasi-isomorphisms, by a standard filtration argument that fails in the non-complete case \cite[6.27]{inftycoalg}. Hence, using the above, we have a quasi-isomorphism $\widehat{\Omega}BA \simeq \widehat{\Omega} BC$. By the $A_\infty$ version of \ref{artprop}(1), we have an isomorphism $\widehat{\Omega} BC\cong (C^*)^!$. But $A^!$ is quasi-isomorphic to $C^*$ and the result follows.
\end{rmk}

\begin{prop}\label{kdforart}
	Let $A$ be an augmented locally finite dga. Suppose that $A$ is either connective or $2$-coconnective. Then there is a natural quasi-isomorphism $A^{!!} \simeq A$.
	\end{prop}
\begin{proof}
	If $A$ is connective, then it is raylike and \ref{artprop}(3) tells us that we have $A^!\cong \Omega(A^*)$. Applying the Koszul dual to both sides we get an isomorphism $A^{!!}\cong \Omega(A^*)^!$. But $ \Omega(A^*)^!$ is the linear dual of $B\Omega(A^*)$, which is weakly equivalent to $A^*$. A weak equivalence is a quasi-isomorphism, and applying the linear dual we get a quasi-isomorphism $A^{!!}\cong A$. On the other hand, if $A$ is $2$-coconnective, then one can use the same argument but with \ref{artprop}(4) instead.
	\end{proof}

\begin{lem}\label{tenslem}
	Let $V$ be a graded vector space. Assume $T(V)$, the tensor algebra on $V$, is locally finite. Then $V$ is locally finite, $V^0\cong 0$, and $V$ is either connective or coconnective.
	\end{lem}
\begin{proof}The point is to look at the subalgebra $T(V)^0$. It is clear that $V$ must be locally finite as we have an obvious linear embedding $V \into T(V)$ of rank $1$ tensors. If the subalgebra $T(V)^0$ has an element $x$ that is not a multiple of the unit, then $x$ must generate a polynomial subalgebra $k[x]$, and hence $T(V)^0$ is infinite-dimensional. Hence $T(V)^0$ must be $k$. Because $T(V^0)$ embeds into $T(V)^0$ as a subalgebra, we must have $T(V^0) \cong k$ and hence $V^0\cong 0	$. Suppose that $V^n \not \cong 0$ for some $n>0$. Then for any $m>0$, $(V^n)^{\otimes m}\otimes (V^{-m})^{\otimes n}$ is a linear subspace of $T(V)^0$ not containing the unit. Hence it must be $0$, and it follows that $V^m \cong 0$. In particular if $V$ contains anything in positive degree, then it is coconnective. Hence $V$ is either connective or coconnective.
	\end{proof}

\begin{thm}[{\cite[Theorem A]{lpwz}}]\label{lpwzthm}
	Let $A$ be an augmented dga such that $BA$ is locally finite. Then there is a natural quasi-isomorphism $A^{!!} \simeq A$.
	\end{thm}
\begin{proof}
By \ref{tenslem}, it follows that $V\coloneqq\bar{A}[1]$ is locally finite, either connective or coconnective, and $V^0 \cong 0$. It follows that $A$ is locally finite, and either connective or $2$-coconnective. In either case the theorem follows from \ref{kdforart}.
	\end{proof}

\begin{rmk}\label{lpwzremark}
	The original proof of \ref{lpwzthm} appearing in \cite{lpwz} is incorrect (although it works if $A$ is assumed connective). The problem is with \cite[1.15]{lpwz}, which claims that if $A$ is an augmented locally finite dga then $\Omega(A^*)\cong A^!$ as dgas. This is not true, as we have already seen in \ref{lpwzwrong}. Moreover, $A^*$ need not even be a dgc: consider the graded algebra $A=k[x,x^{-1}]$ with $x$ in degree $1$. Then $A^*$ is not a graded coalgebra, since the multiplication fails to dualise to a map $A^* \to A^* \otimes A^*$.
	\end{rmk}

\subsection{Completions}
One can improve \ref{lpwzthm} by being more careful: the basic idea is that if $BA$ is a reasonable coalgebra, then $\Omega BA$ is quasi-complete, in the sense that the completion map is a quasi-isomorphism. The proof will require some use of $A_\infty$ methods: for basic facts about $A_\infty$-algebras, we refer the reader to \cite{kellerainfty, lefevre}. All we really use is the existence of minimal models.

 \p We begin with the following observation about the zeroth cohomology of a Koszul dual, which essentially appears in \cite{segaldefpt}.
 
 \begin{defn}\label{hmcdef}
 Observe that if $E$ is an $A_\infty$-algebra, and $m_n$ is one of the $A_\infty$ operations on $E$, then it restricts to a map $m_n:(E^1)^{\otimes n} \to E^2$. The \textbf{homotopy Maurer--Cartan function} is the direct sum $m=\oplus_n m_n:T(E^1) \to E^2$. 
 \end{defn}
One can extend the terminology of \ref{ncocondefn} to $A_\infty$-algebras in the obvious way; in particular we say that an  $A_\infty$-algebra $E$ is $1$-coconnective if it is coconnective and $E^0\cong k$.
\begin{prop}\label{segalthm}
	Let $E$ be an augmented 1-coconnective $A_\infty$-algebra. Assume that $E^1$ is finite-dimensional. Then there is an algebra isomorphism $$H^0(E^!)\cong \frac{\hat{T}(E^{1*})}{m^*(E^{2*})}$$ where $m$ is the homotopy Maurer--Cartan function.
\end{prop}
\begin{proof}
For brevity put $V\coloneqq \bar E[1]$ the shifted augmentation ideal and $A\coloneqq E^!$ the Koszul dual. Let $\partial$ be the $A_\infty$ bar differential on the tensor coalgebra $TV$. Because $E$ was $1$-coconnective, $V$ is coconnective. Hence $BE$ is coconnective and so $A$ is connective. In particular, we have $$H^0(A)\cong \frac{A^0}{\partial^*(A^{-1})}.$$For brevity, write $G$ (`generators') for $V^0\cong E^1$ and $R$ (`relations') for $V^1\cong E^2$; by assumption $G$ is finite-dimensional. Because $V$ is coconnective it follows that $T^0(V)\cong T(G)$. Because $G$ is finite-dimensional it follows that $A^0$ is the completed tensor algebra $\hat{T}(G^*)$.
	
	\p I claim that $\langle\im(m^*)\rangle =\partial^*(A^{-1})$ as ideals of $A^0$; showing this claim will prove the theorem. For this, first observe that $T^1(V)$ is the double direct sum $$T^1(V)\cong\bigoplus_{i,j} G^{\otimes i}\otimes R \otimes G^{\otimes j}$$and hence dualising and using that $G$ is finite-dimensional we see that $A^{-1}$ is $$A^{-1}\cong\prod_{i,j} G^{*\otimes i}\otimes R^* \otimes G^{*\otimes j}\cong A^0 R^* A^0.$$Because $\partial^*$ vanishes on elements of $A^0$, we see that $\partial^*(A^{-1})=A^0 \partial^*(R^*)A^0$. So it will suffice to show that $\im(m^*)=\partial^*(R^*)$. But it is easy to see that the restriction of $\partial^*$ to $R^*$ is precisely $m^*$, because an element $t \in T(G)$ satisfies $\partial t \in R$ precisely when $\partial t = m t$.
	\end{proof}
We will need a lemma on the completion of graded algebras. Recall that if $A$ is an augmented graded algebra then $\hat A$ denotes the completion along its augmentation ideal.
\begin{lem}\label{cpltlem}
		Let $A$ be a connective augmented graded algebra with $A^0$ nilpotent. Then $A$ is complete.
	\end{lem}
\begin{proof}The idea is that in a large enough product of elements from $A$, one can group sufficiently many elements from $A^0$ together to make the product vanish. Let $\mathfrak m$ be the augmentation ideal of $A$. Fix a `degree' $i\leq0$ and an `exponent' $j>0$. By connectivity, for sufficiently large $n$ (depending on $i$ and $j$) all monomials of $(\mathfrak{m}^n)^i$, the degree $i$ part of the ideal $\mathfrak{m}^n$, must contain a product of at least $j$ nonidentity elements from $A^0$. But $A^0$ is nilpotent, so there is some fixed $N$ such that if $j>N$ then a product of $j$ nonidentity elements from $A^0$ must vanish. Hence it follows that for all sufficiently large $n$ (depending on $i$) we have $(\mathfrak{m}^n)^i=0$. So as vector spaces, we have $$(\hat A)^i\cong \varprojlim_n(A/\mathfrak{m}^n)^i \cong \varprojlim_nA^i/(\mathfrak{m}^n)^i \cong A^i$$ and hence $A$ is complete.
	\end{proof}
\begin{rmk}
	If $A$ is a connective graded algebra, and $A^0$ is a central subalgebra, then $A$ is complete if and only if $A^0$ is complete. More generally, $A^0$ itself need not be commutative; it suffices that the commutators $[A^0,A^i]$ vanish for $i<0$.
\end{rmk}
\begin{rmk}
	This shows that the `forget the grading' functor on graded algebras does not preserve limits. Indeed take for example $A=k[x]$ with $x$ in degree -1. In the category of graded algebras, the limit $\varprojlim_nA/x^n$ is just $A$ again by the above. But the limit in the category of ungraded algebras is $k\llbracket x \rrbracket$, which does not even admit a nontrivial grading.
	\end{rmk}

Our key technical result is the following:

\begin{thm}\label{htpycomplete}
Let	$E$ be an augmented $1$-coconnective locally finite dga such that $H^0(\Omega(E^*))$ is nilpotent. Then the completion map $c: \Omega(E^*) \to \widehat{\Omega}(E^*) $ is a quasi-isomorphism.
	\end{thm}
\begin{proof}
	For brevity put $\Omega\coloneqq \Omega(E^*)$ and $\widehat{\Omega}\coloneqq \widehat{\Omega}(E^*)$. The rough idea of the proof is that $H\Omega$ is complete by \ref{cpltlem}, and moreover completion is flat so $\widehat{H\Omega}\cong H\widehat{\Omega}$. The above is literally true if we are using symmetric algebras rather than tensor algebras, but we do not have access to results from commutative algebra so we will have to be more careful.
	
	\p We begin by studying $H^0c$. Because $E$ is $1$-coconnective we see that both $\Omega$ and $\widehat{\Omega}$ are connective. By \ref{artprop} we have a dga isomorphism $\widehat{\Omega}\cong E^!$. As in the proof of \ref{segalthm}, we see that $H^0\Omega\cong T(E^{1*}) / \langle m^*(E^{2*})\rangle$, where $m$ denotes the multiplication on $E$. Hence by using \ref{segalthm} we see that $H^0c$ is the completion map
	$$H^0c:\frac{T(E^{1*})}{\langle m^*(E^{2*})\rangle} \to \frac{\hat{T}(E^{1*})}{\langle m^*(E^{2*})\rangle}$$which is an isomorphism because a nilpotent algebra is complete
	
	\p First we show that $c$ is a quasi-injection. Denote the differential in $\Omega $ by $\partial$ and denote the differential in $\widehat{\Omega}$ by $\hat \partial$. Let $\bar W \subseteq H\Omega$ be a space of generators for the algebra $H\Omega$. Let $W \subseteq \ker \partial \subseteq \Omega$ be a space of lifts of $\bar W$, and let $K\subseteq \ker \partial \subseteq \Omega$ be the subalgebra of $\Omega$ generated by $W$. Because $W$ lifts the generators of $H\Omega$, we see that the natural map $K \to H\Omega$ is a surjection. Complete this natural map to obtain a commutative diagram 
$$\begin{tikzcd}
K \ar[r] \ar[d,two heads]& \hat K \ar[d]
\\ H \Omega \ar[r] & \widehat{H\Omega}.\end{tikzcd}$$	
Using \ref{cpltlem} we see that $H\Omega$ is complete, and hence the bottom map is an isomorphism, which implies that $K \onto H\Omega$ extends to a surjection $\hat K \onto H\Omega$. Viewing $\hat K$ as a subalgebra of $\widehat{\Omega}$, we obtain another commutative diagram 
$$\begin{tikzcd}
\hat K \ar[r,"\id"] \ar[d,two heads]& \hat K \ar[d]
\\ H \Omega \ar[r] & H\widehat{\Omega}.
\end{tikzcd}$$
Suppose that $x\in H^i\Omega$ satisfies $Hc(x)=0$. Lift $x$ to an element $\tilde x$ of $\hat K$. Because $\tilde x$ maps to $0$ under the projection $\hat K \to H\widehat{\Omega}$, we must have $\tilde x = \hat \partial y$ for some $y$. But then $\tilde x$ maps to $0$ under the projection $\hat K \onto H\Omega$. Because $\tilde x$ lifted $x$, we see that $x$ was $0$. Hence $Hc$ is a levelwise injection, meaning precisely that $c$ is a quasi-injection.

\p The proof that $c$ is a quasi-surjection is in some sense dual. The inclusion $\Omega \into \widehat{\Omega}$ gives maps of graded algebras $$\frac{\Omega}{\langle \im \partial\rangle} \onto \frac{{\Omega}}{\langle \im \hat\partial\rangle\cap \Omega} \into \frac{\widehat{\Omega}}{\langle \im \hat\partial\rangle}. $$Moreover, these maps are compatible with cohomology, in the sense that we have a commutative diagram $$\begin{tikzcd}
H\Omega  \ar[r,hook] \ar[d]&  \frac{\Omega}{\langle \im \partial\rangle} \ar[d]
\\ H \widehat{\Omega} \ar[r, hook] & \frac{\widehat{\Omega}}{\langle \im \hat\partial\rangle}.
\end{tikzcd}$$But because $H^0(\widehat{\Omega})$ is nilpotent, it follows that $H\widehat{\Omega}$ is a polynomial algebra, in the sense that the image of $H\widehat{\Omega}$ is contained in the subalgebra $\frac{{\Omega}}{\langle \im \hat\partial\rangle\cap \Omega}$. Now we can lift: take $x\in H\widehat{\Omega}$, regarded as a subspace of  $\frac{{\Omega}}{\langle \im \hat\partial\rangle\cap \Omega}$. Then find $\tilde x \in \Omega$ mapping to $x$. Because $x$ was represented by a coboundary, $\tilde x $ must be a coboundary, and its class in $H\Omega$ is a lift of $x$. Hence $Hc$ is a levelwise surjection, i.e. $c$ is a quasi-surjection. 

\p So $c$ is a quasi-isomorphism, as required.
	\end{proof}
\begin{rmk}We see that \ref{htpycomplete} may fail to be true if $H^0(\Omega(E^*))$ is not complete, for the same reason as in \ref{lpwzwrong}.
	\end{rmk}

\begin{thm}\label{kdfin}
	Let $A$ be a connective augmented cohomologically locally finite dga with $H^0(A)$ Artinian local. Then there is a natural quasi-isomorphism $A^{!!} \simeq A$.
\end{thm}
\begin{proof}
	Let $H$ be a minimal $A_\infty$-algebra model for $A$, the existence of which is ensured by \cite{kadeishvili}. In other words, as a graded vector space $H$ is the cohomology algebra of $A$, and we equip $H$ with higher multiplications making it $A_\infty$-quasi-isomorphic to $A$. Let $BH$ be the $A_\infty$ bar construction on $H$, which is a dgc; as a graded vector space it is $T(\bar{H}[1])$ and the differential incorporates the higher multiplications into the usual bar differential. Because $H$ is connective and locally finite, we see that $BH$ is also connective and locally finite. Moreover, $(BH)^0$ is just $k$ (i.e.\ $BH$ is $1$-connective). In particular, $(\bar{B}{H})[-1]$ is connective, and it follows that $\Omega BH$ and $\widehat{\Omega}BH$ are connective. Because $BH$ is raylike, the dga $H^!$ is raylike and so by \ref{artprop} we have a dga isomorphism $H^{!!}\cong \widehat{\Omega}BH$. Because the $A_\infty$ Koszul dual sends $A_\infty$-quasi-isomorphisms to dga quasi-isomorphisms, we see that $H^{!!}\simeq A^{!!}$ as dgas. Similarly, because $\Omega B$ sends $A_\infty$-quasi-isomorphisms to dga quasi-isomorphisms, we see that $\Omega B H \simeq \Omega B A \simeq A$ as dgas. It is now easy to see that $H^!$ satisfies the hypotheses of  \ref{htpycomplete} and, using that $H^{!*}\cong BH$, we hence see that the completion map $\Omega BH \to \widehat{\Omega}BH$ is a quasi-isomorphism. We hence have a chain of quasi-isomorphisms $$A \simeq \Omega BH \simeq \widehat{\Omega}BH \cong H^{!!}\simeq A^{!!}$$ and the claim is proved. 
	\end{proof}

\begin{rmk}
In the above, one does not necessarily need $A$ to be connective; it suffices that $H(A^!)$ is $1$-coconnective, as this is what is required to make the cobar construction on its dual connective. 
	\end{rmk}

\begin{rmk}\label{andreyrmk}
	Andrey Lazarev has suggested to the author that the `correct' version of the preceding theorem should be something like the following. Let $A \to k$ be an augmented dga. Then under some mild conditions on $A$, the Koszul double dual $A \to A^{!!}$ is quasi-isomorphic to the Bousfield localisation of the right $A$-module $A$ with respect to the homology theory $M \mapsto \tor_*^A(M,k)$. The idea is that the formal completion $A^{!!}$, the nilpotent completion $\widehat{A}_k$, and the Bousfield localisation $L_k(A)$ should all agree. Moreover, if $A$ is cohomologically locally finite with $H^0(A)$ Artinian local, then the Bousfield localisation of $A$ is $A$ again. To prove this second statement, he suggests that one should show that the cobar spectral sequence associated to $A^{!!}$ converges to the cohomology of the localisation of $A$, in a similar manner to the convergence of the $E$-Adams spectral sequence. The relevant computations ought to be similar to those of \cite{bousfieldloc} or \cite{dwyerexotic}. The result we want may also follow from the recent preprint \cite{mantovani}. This approach may also extend to the setting of ring spectra.
\end{rmk}

\begin{rmk}\label{parmk}
	Let $A$ be any augmented connective dga. Let $B^\sharp(A)$ denote the \textbf{continuous Koszul dual}: one takes the bar construction on $A$ and then applies the $(-)^\sharp$ functor of \ref{csharp} to obtain a pro-Artinian dga. If one takes the levelwise Koszul dual to obtain a pro-Artinian dga $\left(B^\sharp(A)\right)^!$, it is clear that $\varprojlim\left(B^\sharp(A)\right)^!$ is quasi-isomorphic to $A$ by applying \ref{kdforart} levelwise. However, it is far from clear that the same applies when we forget that $B^\sharp A$ is pro-Artinian; i.e.\ take $(\varprojlim B^\sharp(A))^!$ instead.
\end{rmk}

\begin{rmk}\label{mythesisrmk}
	In \cite{me}, a different - and less direct - proof of \ref{kdfin} is given. The loose idea is that the Postnikov tower of $A$ exhibits $A$ as an iterated tower of homotopy square-zero extensions, starting from $H^0(A)$. It is easy to show directly that $H^0(A)^{!!}\simeq H^0(A)$. Then a technical analysis of derivations and an induction on $n$ shows that each level $A_n$ of the Postnikov tower of $A$ also satisfies $A_n^{!!}\simeq A_n$. Then one studies the homotopy limit of this tower to conclude that in fact we have $A^{!!}\simeq A$.
	\end{rmk}

	\subsection{Lifting weak equivalences}
Suppose that $A$ and $A'$ are limits of two connective pro-Artinian dgas. When does a quasi-isomorphism $A\simeq A'$ lift to a weak equivalence of pro-Artinian dgas? We show that this can always be done when $A$ and $A'$ are cohomologically locally finite. The key observation is that the Koszul double dual is closely related to cofibrant resolutions for pro-Artinian algebras. We begin with the following useful observation:

\begin{lem}\label{kdfinite}
	Let $A$ be a connective augmented dga. If $A$ is cohomologically locally finite then so are $BA$ and $A^!$.
\end{lem}
\begin{proof}Since the linear dual is exact, the statement for $A^!$ is implied by the statement for $BA$. To prove the latter, filter $BA$ by the tensor powers of $A$ to obtain a spectral sequence with $E_1$ page ${H^p(A^{\otimes q}) \Rightarrow H^{p-q}(BA)}$. Since there are only finitely many nonzero $H^p(A^{\otimes q})$ with $p-q$ fixed, and they are all finite-dimensional, $H^{p-q}(BA)$ must also be finite-dimensional.
\end{proof}
\begin{rmk}
	One can also prove \ref{kdfinite} by applying the $A_\infty$ bar construction to an $A_\infty$ minimal model for $A$, which yields a locally finite model for $BA$.	
\end{rmk}
\begin{thm}\label{doubledictionary}
	Let $\mathcal{A}\in \ubproart$ be a pro-Artinian dga and put $A\coloneqq \varprojlim \mathcal A$.
	\begin{enumerate}
		\item The natural map $B^\sharp\Omega(\mathcal{A}^\circ) \to \mathcal{A}$ is a cofibrant resolution of pro-Artinian dgas.
		\item The natural map $\Omega(\mathcal{A}^\circ)^! \to A$ is a dga quasi-isomorphism.
		\item If $\mathcal{A}$ is connective then there is a natural dga isomorphism $\Omega\mathcal{A}^\circ\cong \varinjlim\mathcal{A}^!$, where $\mathcal{A}^!$ is the ind-dga obtained from $\mathcal{A}$ by applying the Koszul dual levelwise.
		\item If $\mathcal{A}$ is connective and $BA$ is cohomologically locally finite then there is a natural dga quasi-isomorphism $\varinjlim \mathcal{A}^! \to A^!$.
		\item  If $\mathcal{A}$ is connective and $BA$ is cohomologically locally finite then the natural map  $B^\sharp(A^!) \to \mathcal{A}$ is a cofibrant resolution of pro-Artinian dgas.
		\item  If $\mathcal{A}$ is connective and $BA$ is cohomologically locally finite then the natural map $A^{!!} \to A$ is a cofibrant resolution.
	\end{enumerate}
\end{thm}
\begin{proof}
	For (1), just use that $B\Omega$ is a fibrant dgc resolution along with the fact that $\sharp$ and $\circ$ are inverse Quillen equivalences. For (2), use that $\varprojlim$ preserves weak equivalences (by \ref{limisrq}) to conclude that the natural map $\Omega(\mathcal{A}^\circ)^! \cong\varprojlim B^\sharp\Omega(\mathcal{A}^\circ) \to \varprojlim \mathcal{A}\cong A$ is a quasi-isomorphism. For (3), first note that by construction, $\mathcal{A}^\circ=\varinjlim \mathcal{A}^*$. Because $\Omega $ is a left adjoint, it is cocontinuous, and so we have an isomorphism $\Omega\mathcal{A}^\circ\cong \varinjlim \Omega (\mathcal{A}^*)$. But if $A$ is a connective Artinian dga then $\Omega (A^*)\cong A^!$ by \ref{artprop}. So $\Omega (\mathcal{A}^*) \cong \mathcal{A}^!$ as ind-dgas and the claim follows. 
	
	To prove (4) will take some more work. For brevity put $C\coloneqq BA$ and $\mathcal{C}\coloneqq B\mathcal{A}$; note that because $B$ is a right adjoint it is continuous and hence $C\cong \varprojlim \mathcal C$ as coalgebras. Let $\phi$ be the natural map $\phi:\varinjlim \mathcal{C}^* \to C^*$ that we wish to prove is a quasi-isomorphism.
	
	For $n \in \Z$, consider the induced linear map $$\psi_n:\qquad\varinjlim(H^n(\mathcal{C}^*)) \xrightarrow{\cong}H^n(\varinjlim\mathcal{C}^*) \xrightarrow{H^n\phi} H^n(C^*) \xrightarrow{\cong}H^{-n}(C)^*$$where we have used exactness of filtered colimits and the linear dual, and dualise it to obtain a map $$\chi_n:\qquad H^{-n}(C) \to H^{-n}(C)^{**} \xrightarrow{\psi_n^*} (\varinjlim(H^n(\mathcal{C}^*)))^* \xrightarrow{\cong} \varprojlim H^n(\mathcal{C^*})^{*} \xrightarrow{\cong} \varprojlim H^{-n}(\mathcal{C}^{**})$$where we have used exactness of the linear dual again along with the fact that Hom preserves limits. Because $C$ is cohomologically locally finite,  $H^{-n}(C) \to H^{-n}(C)^{**}$ is an isomorphism. Similarly, each level ${\mathcal{C}_\alpha}$ of $\mathcal{C}$ is locally finite, since it is the bar construction on a connective Artinian dga. In particular, the natural map $H^{-n}(\mathcal{C}_\alpha) \to H^{-n}(\mathcal{C}_\alpha^{**})$ which sends $[v]$ to $[\mathrm{ev}_v]$ is an isomorphism. Let $[u] \in H^{-n}(C)$; one can compute that $\chi_n([u])=([\mathrm{ev}_{u_\alpha}])_\alpha$, where $u_\alpha$ is the image of $u$ under the natural map $C \to \mathcal{C}_\alpha$. Hence, the composition $H^{-n}(C) \xrightarrow{\chi_n}\varprojlim H^{-n}(\mathcal{C}^{**}) \xrightarrow{\cong} \varprojlim H^{-n}(\mathcal{C})$ of $\chi_n$ with the inverse to the natural isomorphism sends $[u]$ to $[u_\alpha]_\alpha$. But this is precisely the natural map $H^{-n}(C) \to \varprojlim H^{-n} (\mathcal{C})$. But as in \ref{limisexact}, this natural map is an isomorphism: because we may choose $\mathcal{A}$ to be strict, we may assume that $\mathcal{A}$ is a Reedy fibrant diagram of dgas. Because $B$ is right Quillen, it follows that $\mathcal{C}$ is a Reedy fibrant diagram of dgcs and hence $C\simeq \holim \mathcal C$. Now it follows that $\chi_n$, $\psi_n$, and $H^n\phi$ are isomorphisms for all $n$. Hence, $\phi$ is a quasi-isomorphism as claimed.
	
	For (5), first combine (4) and (3) to get a natural dga quasi-isomorphism $\Omega \mathcal{A}^\circ \to A^!$. Apply $B^\sharp$ to this to get a natural weak equivalence $B^\sharp(A^!) \to B^\sharp\Omega(\mathcal{A}^\circ)$. Compose with the weak equivalence of (1)  to see that there is a natural weak equivalence $B^\sharp(A^!) \to \mathcal{A}$, which is a cofibrant resolution because dgcs in the image of $B$ are fibrant.
	
	The proof of (6) is the same as (2): just take limits of (5).
\end{proof}

\begin{prop}\label{liftingpa}
	Let $\mathcal A$ and $\mathcal{A}'$ be connective pro-Artinian dgas and put $A\coloneqq\varprojlim \mathcal{A}$, $A'\coloneqq\varprojlim \mathcal{A}'$. Suppose that $A$ is quasi-isomorphic to $A'$ and that moreover both are cohomologically locally finite. Then $\mathcal A$ and $\mathcal{A}'$ are weakly equivalent.
	\end{prop}
\begin{proof}
Because $A$ is cohomologically locally finite, so is $BA$ by \ref{kdfinite}. So applying \ref{doubledictionary}(5) we see that $B^\sharp(A^!) \to\mathcal{A}$ is a cofibrant resolution. Similarly, $B^\sharp(A'^{!}) \to \mathcal{A}'$ is a cofibrant resolution. But the functor $X \mapsto B^\sharp(X^!)$ sends quasi-isomorphisms to weak equivalences, and so the result follows.
	\end{proof}

\subsection{Derived mapping spaces}
		Our result \ref{doubledictionary} is quite powerful, and will allow us to compare derived mapping spaces between connective pro-Artinian dgas and derived mapping spaces between their limits. Our main theorem here will later be used to give us a more concrete description of representing objects for framed deformations.
		
		\begin{lem}\label{cdlem}
			Let $C$ be a coaugmented conilpotent dgc and let $\Gamma$ be an Artinian local dga. 
			\begin{enumerate}
				\item There is an isomorphism $$\hom_{\ubproart}(C^\sharp, \Gamma) \cong \hom_{\augdga}(C^*, \Gamma).$$
				\item There is an isomorphism $\widehat{C^*} \cong C^\sharp$ of pro-Artinian dgas.
			\end{enumerate}	
			\end{lem}
\begin{proof}
	The point is that $\Gamma$ is finite-dimensional. We begin with (1). By \ref{sharpprop} we have an isomorphism $$\hom_{\ubproart}(C^\sharp, \Gamma) \cong \hom(\Gamma^*, C)$$where we take the right-hand hom in the category of conilpotent coalgebras, so it suffices to show that there is an isomorphism $\hom(\Gamma^*, C)\cong\hom_{\augdga}(C^*, \Gamma)$. Any coalgebra morphism $\Gamma^* \to C$ gives an algebra morphism $C^* \to \Gamma$ by dualising, and it is clear that the corresponding map $\hom(\Gamma^*, C) \to\hom_{\augdga}(C^*, \Gamma)$ is injective. For surjectivity, suppose given a morphism of augmented algebras $C^* \to \Gamma$, and dualise to obtain a linear map $\Gamma^* \to C^{**}$. The functionals in the image of $\Gamma^*$ all have finite support since they vanish on the cofinite subspace $\ker(C^* \to \Gamma)$ of $C^*$. Hence the map $\Gamma^*\to C^{**}$ factors through the canonical map $C\into C^{**}$. Dualising the resulting map $\Gamma^* \to C$, which is a coalgebra morphism,  obtains the desired morphism $C^* \to \Gamma$.
	
	\p For (2), first recall that if $A$ is a dga, then $\widehat{A}$ is the left adjoint of the $\varprojlim$ functor; concretely $\widehat{A}$ is the cofiltered set of Artinian quotients of $A$. If $C$ is a coalgebra, the proof of (1) shows that the dual is an order-reversing bijection from the set of finite-dimensional subdgcs of $C$ to the set of Artinian quotients of $C^*$. Because the pro-Artinian dga $C^\sharp$ is exactly the levelwise dual of the cofiltered system of finite-dimensional subdgcs of $C$, the claim follows. 
	\end{proof}

\begin{lem}\label{mslem}
Let $C$ be a coconnective coaugmented conilpotent dgc and let $\mathcal{B}$ be a pro-Artinian dga. Then there is a natural isomorphism $$\hom_{\proart}(C^\sharp,\mathcal{B})\cong \hom_{\augdga^{\leq 0}}(C^*,\varprojlim \mathcal{B}).$$
\end{lem}
\begin{proof}
By \ref{cdlem}(2) we have a natural isomorphism $$\hom_{\proart}(C^\sharp,\mathcal{B})\cong \hom_{\proart}(\widehat{C^*},\mathcal{B}).$$By the adjunction $\widehat{(-)} \dashv \varprojlim$ we have a natural isomorphism$$\hom_{\proart}(\widehat{C^*},\mathcal{B})\cong\hom_{\augdga^{\leq 0}}(C^*,\varprojlim \mathcal{B})$$and so we are done.
\end{proof}

\begin{thm}\label{dermapsthm}
	Let $\mathcal A$ and $\mathcal{A}'$ be connective pro-Artinian dgas and put $A\coloneqq\varprojlim \mathcal{A}$, $A'\coloneqq\varprojlim \mathcal{A}'$. Suppose that the bar construction $BA$ is cohomologically locally finite. Then there is a weak equivalence of derived mapping spaces $$\rmap_{\proart}(\mathcal{A},\mathcal{A}')\simeq \rmap_{\augdga^{\leq 0}}(A,A').$$
	\end{thm}
\begin{proof}The idea is to use a simplicial frame on $\mathcal{A}'$ together with the explicit resolutions given by \ref{doubledictionary}. By \ref{doubledictionary}(5), the natural map $B^\sharp (A^!) \to \mathcal{A}$ is a cofibrant resolution in $\proart$. Moreover, all objects of $\proart$ are fibrant, so as in \cite[\S5]{hovey} we may use a simplicial frame $\mathcal{A}'_\bullet$ on $\mathcal{A}'$ to compute derived mapping spaces. We hence get weak equivalences
	\begin{align*} \rmap_{\proart}(\mathcal{A},\mathcal{A}')  &\simeq \hom_{\proart}(B^\sharp (A^!),\mathcal{A}'_\bullet)& \\
		&\cong  \hom_{\augdga^{\leq 0}}(A^{!!},\varprojlim(\mathcal{A}'_\bullet)) &\text{by \ref{mslem} levelwise.}
		\end{align*}
	By \ref{limisrq}, the functor $\varprojlim$ is right Quillen. Hence by \cite[5.6.1]{hovey}, the simplicial object $\varprojlim(\mathcal{A}'_\bullet)$ is a simplicial frame on $A'$. Moreover, by \ref{doubledictionary}(6), the natural map $A^{!!} \to A$ is a cofibrant resolution, and every object in $\augdga^{\leq 0}$ is fibrant. Hence the simplicial set $\hom_{\augdga^{\leq 0}}(A^{!!},\varprojlim(\mathcal{A}'_\bullet))$ is a model for $\rmap_{\augdga^{\leq 0}}(A,A')$ as required.
	\end{proof}

\begin{cor}\label{dermapscor}
	Let $A$ be a connective augmented cohomologically locally finite dga with $H^0(A)$ Artinian local. Let $\mathcal{A}'$ be any connective pro-Artinian dga and put $A'\coloneqq\varprojlim \mathcal{A}'$. Then there is a weak equivalence of derived mapping spaces $$\rmap_{\proart}(B^\sharp(A^!),\mathcal{A}')\simeq \rmap_{\augdga^{\leq 0}}(A,A').$$
	\end{cor}
\begin{proof}
	Put $\mathcal{A}\coloneqq B^\sharp(A^!)$. By \ref{kdfin}, we have a natural quasi-isomorphism $\varprojlim \mathcal{A} \cong A^{!!}\simeq A$. By \ref{kdfinite} and the invariance of $B$ under quasi-isomorphisms, $BA$ is cohomologically locally finite. Hence we may apply \ref{dermapsthm} to conclude that we have a weak equivalence $$\rmap_{\proart}(\mathcal{A},\mathcal{A}')\simeq \rmap_{\augdga^{\leq 0}}(A,A')$$as required.
	\end{proof}

\section{Deformations and prorepresentability}\label{defmthy}
In this section, we work over a field $k$ of characteristic zero. Our goal in this section is to use Koszul duality to explicitly identify prorepresenting objects for the `algebraic' deformation problem of deforming a module over a dga. We will work with deformation functors valued in simplicial sets, and our representability statement will really be a homotopy representability statement giving a weak equivalence between our functor and a derived mapping complex. We work with nonunital dgas throughout; this will help in the next section when we rigidify by considering framed deformations, which on the side of representing objects corresponds loosely to forgetting about the unit.

\p The input to our derived deformation functors will be connective Artinian dgas, which we will denote by $\Gamma$ and their maximal ideals by $\mathfrak{m}_\Gamma$. We note that if one wants to consider non-connective dgas as input, then one needs to consider some sort of stacky deformations because non-connective dgas may have nontrivial Maurer-Cartan elements (\ref{connectivereasons}).

\p We remark that many of the results of this section are true when $k$ has positive characteristic: key here is that we are deforming over connective noncommutative dgas, which are Quillen equivalent to simplicial $k$-algebras. In positive characteristic the equivalence between connective cdgas and simplicial commutative algebras breaks down, and correspondingly commutative deformation theory in positive characteristic is significantly more difficult. We note that in positive characteristic, a version of the Lurie--Pridham correspondence for commutative formal moduli problems has recently been given by Brantner and Mathew \cite{partitionlie}. The only place we really need $k$ to be characteristic zero is in the polynomial  Poincar\'e Lemma \ref{poincare}, which allows us to build simplicial resolutions of dgas via $\pdf$, the dga of polynomial differential forms. One might be able to reproduce all of our results in arbitrary characteristic using a different simplicial resolution for $k$ instead of the polynomial differential forms $\pdf$.

	\subsection{The Maurer--Cartan and Deligne functors}
	It is known that deformations of modules are given by a certain functor called the \textbf{Deligne functor}, which is constructed as a homotopy quotient. Our representability statement will consist of showing that the Deligne functor is (homotopy) representable. In this section we construct the Deligne functor.
\begin{defn}
	Let $U$ be a (possibly nonunital) dga. The set of \textbf{Maurer--Cartan elements} (or just \textbf{MC elements}) of $U$ is the set $$\mcs(U)\coloneqq \{x \in U^1: \ dx+x^2=0\}.$$
\end{defn}
\begin{rmk}
	Note that if $U$ is unital, with differential $d$, then $x \in \mcs(U)$ if and only if the map $u\mapsto d(u)+xu$ is a differential on $U$.
\end{rmk}
\begin{rmk}
	A (possibly nonunital) dga $U$ canonically becomes a dgla when equipped with the commutator bracket, and the set of MC elements of the dgla $U$ is the same as the set of MC elements of the dga $U$.
\end{rmk}

\begin{defn}
	Let $E$ be a dga. The \textbf{Maurer--Cartan functor} $$\mc(E):\dgart \to \cat{Set}$$ sends an Artinian dga $\Gamma$ to the set $\mc(E)(\Gamma)\coloneqq \mcs(E\otimes \mathfrak{m}_\Gamma)$. 
\end{defn}

\begin{defn}
	Let $E$ be a dga. The \textbf{gauge group} functor $$\ggr(E):\dgart \to \cat{Grp}$$ sends $\Gamma$ to the set $1+(E\otimes \mathfrak{m}_\Gamma)^0$, which is a group under multiplication.
\end{defn}

\begin{prop}
	Let $E$ be a nonunital dga and $\Gamma$ an Artinian local dga. Then  $\ggr(E)(\Gamma)$ acts on $\mc(E)(\Gamma)$ via the formula $g.x=gxg^{-1}+gd(g^{-1})$.
\end{prop}
\begin{proof}
	This is an easy verification.
\end{proof}
\begin{rmk}Regarding $d+x$ as a twisted differential on $E\otimes \mathfrak{m}_\Gamma$, the action of the gauge group is the conjugation action on the space of differentials.
\end{rmk}

\begin{rmk}
	If $L$ is a dgla, its gauge group has as elements formal symbols $\exp(a)$ for $a \in L^0$, and multiplication given by the Baker--Campbell--Hausdorff formula \cite{manettidgla}. In characteristic zero, if $E\otimes \mathfrak{m}_\Gamma$ is made into a dgla using the commutator bracket then its dgla gauge group is isomorphic to the gauge group defined above, via the map that sends each formal exponential $\exp(a)$ to the sum $\sum_n\frac{a^n}{n^!}$, which exists because $\mathfrak{m}_\Gamma$ is nilpotent. Note that in positive characteristic the sum may not be defined. The exponential of the dgla gauge action \cite[V.4]{manetti} is the gauge action described above.
\end{rmk}

\begin{defn}
	Let $E$ be a dga. The \textbf{Deligne functor} is the quotient functor $$\del(E)\coloneqq \mc(E)/\ggr(E).$$
\end{defn}
Often $\del(E)$ is called the deformation functor associated to $E$. 

\begin{rmk}
	By taking the groupoid quotient rather than the set quotient, one can immediately enhance $\del$ to a groupoid-valued functor. However, we will see that $\del$ has a natural enhancement to a functor $\sdel$ valued in simplicial sets, and we would like the groupoid $\del$ to be the 1-truncation of $\sdel$. However, nontrivial 2-simplices in $\sdel$ induce homotopies between gauges, and one has to quotient these out to get the correct fundamental groupoid; see \cite[\S5]{ELO} or \cite[Proof of 3.2]{manettidgla} in the commutative setting. In the literature, both groupoid-valued functors are referred to as the \textbf{Deligne groupoid}. When deforming along ungraded Artinian algebras the two definitions coincide, so the difference between them only becomes apparent when deforming along genuinely derived objects.
\end{rmk}

\begin{prop}[{\cite[8.1]{ELO}}]\label{setdelqinvt}
	If $E$ and $E'$ are quasi-isomorphic dgas then the functors $\del(E)$ and $\del(E')$ are isomorphic.
\end{prop}

Now we enhance all of our constructions to functors valued in simplicial sets. The loose idea is to take (co)simplicial resolutions to get derived mapping spaces, which are simplicial sets. We use the explicit simplicial enhancement of Hinich \cite{hinichhom}, who is generalising the work of Bousfield and Gugenheim \cite{bousfieldgugenheim}.

\begin{defn}
	Let $\pdf$ denote the simplicial cdga of polynomial differential forms on the standard cosimplicial space $\Delta^\bullet$; in simplicial level $n$ it is freely generated as a cdga by indeterminates $t_0,\ldots,t_n$ in degree zero, modulo the relation $t_0+\cdots + t_n=1$. See e.g.\ \cite[\S1]{bousfieldgugenheim} or  \cite[4.8.1]{hinichhom} for an explicit definition.
\end{defn}
\begin{prop}[{Polynomial Poincar\'e Lemma \cite[1.3]{bousfieldgugenheim}}]\label{poincare}
	The simplicial dga $\pdf$ is quasi-isomorphic to the constant simplicial dga $k$.
\end{prop}
\begin{proof}
	Because we have an isomorphism $\Omega(\Delta^q)\cong \Omega(\Delta^1)^{\otimes q}$, it suffices to check that $\Omega(\Delta^1)\cong k[t,dt]$ is quasi-isomorphic to $k$. To do this, it  is enough to show that the cohomology of $k[t,dt]$ is $k$, because then the unit map $k \to k[t,dt]$ is a quasi-isomorphism.
	\end{proof}
\begin{rmk}
	In characteristic $p$, the polynomial Poincar\'e Lemma does not hold. Indeed, one can check that $\Omega(\Delta^1)\cong k[t,dt]$ has cohomology $k[t^p]$, concentrated in degree zero. This is because the degree $i$ cocycles look like $Q(t^p)dt^i$, where $Q$ is any polynomial. Since $tQ(t^p)$ maps to $Q(t^p)dt$ under the differential, if $i>0$ then these cocycles are all coboundaries.
	\end{rmk}
\begin{defn}
	Let $E$ be a dga. The \textbf{simplicial Maurer--Cartan functor} $\smc$ sends $E$ to the simplicial set $\smc(E)\coloneqq \mc(E\otimes \pdf)$.
\end{defn}
Unwinding the definitions, we hence have $\smc(E)(\Gamma)=\mcs(E\otimes \pdf \otimes \mathfrak{m}_\Gamma)$. 
\begin{rmk}
	It is not true that $\mc\cong \pi_0\smc$, because the right-hand side has elements identified by homotopies coming from 1-simplices in $\smc$. All we have is a quotient map $\mc \to \pi_0 \smc$. In fact, $\pi_0 \smc$ is $\del$, which follows by combining \ref{pi0del} and \ref{smcisdel} below.
\end{rmk}

\begin{prop}\label{mcinvtqiso}
	If $E$ and $E'$ are quasi-isomorphic dgas then the functors $\smc(E)$ and $\smc(E')$ are weakly equivalent.
	
\end{prop}

\begin{proof}
	This appears in \cite{jonddefsartin}, but the statement there is for DDCs (\textbf{derived deformation complexes}) so requires some translation. Loosely, a DDC  $F$ is a $\N\times \N$-graded functor $F^n_m:\dgart \to \cat{Set}$ which is a simplicial object in the $m$ direction and has a structure similar to a cosimplicial object in the $n$ direction. Given a dga $E$, define a $\N\times \N$-graded functor $\mathcal{E}$ by $\mathcal{E}^n_m(\Gamma)\coloneqq \left(E\otimes \Omega(\Delta^m)\otimes \mathfrak{m}_\Gamma\right)^n$. This is not quite a DDC, but does give a DDC $F$ with
	 $$F^n_m=\begin{cases} 1+ \mathcal{E}^0_m & n=0 \\
	 	\mathcal{E}^n_m & \text{else.}
\end{cases}$$One can define a $\mathfrak{MC}$ functor for DDCs, and one has $\mathfrak{MC}(F)\cong \smc(E)$. Quasi-isomorphisms of dgas give quasi-isomorphisms of DDCs, and the result now follows from {\cite[2.18]{jonddefsartin}}.
	\end{proof}

\begin{defn}
	The \textbf{simplicial Deligne functor}	is the homotopy quotient\footnote{See \cite[Chapter V]{goerssjardine} or \cite[1.23]{jonddefsartin} for the definition of homotopy quotients.}$$\sdel(E)\coloneqq [\smc(E)/\ggr(E)].$$ 
\end{defn}

\begin{lem}\label{pi0del}
	There is an isomorphism	$\del\cong \pi_0\sdel$.
\end{lem}
\begin{proof}
	As in \cite[1.27]{jonddefsartin}, this follows by considering the long exact sequence of homotopy groups associated to the fibration $X \to [X/G] \to BG$.
\end{proof}

\begin{prop}[{\cite[2.21]{jonddefsartin}}]\label{smcisdel}
	The quotient map $\smc \to \sdel$ is a weak equivalence.
\end{prop}
\begin{proof}
		The idea is that the classifying space of the gauge group is contractible and so taking the homotopy quotient does not affect the weak equivalence type of $\smc$.
	\end{proof}

\subsection{DG and simplicial categories}
We provide a very brief introduction to the theory of dg categories, simplicially enriched categories, and the interplay between them, primarily to set notation. Survey articles on dg categories include \cite{keller} and \cite{toendglectures}. For the relevant material on simplicial categories, see \cite{bergnermodelscat}.

	\begin{defn}
	A ($k$-linear) \textbf{dg category} is a category $\mathcal{C}$ enriched over the monoidal category $(\cat{dgvect}_k,\otimes)$ of dg $k$-vector spaces with the usual tensor product. In other words, to every pair of elements $(x,y)\in\mathcal{C}^2$ we assign a chain complex $\dgh_\mathcal{C}(x,y)$, to every triple $(x,y,z)$ we assign a chain map $\mu_{xyz}:\dgh_\mathcal{C}(x,y)\otimes \dgh_\mathcal{C}(y,z) \to \dgh_\mathcal{C}(x,z)$ satisfying associativity, and for every $x \in \mathcal{C}$ we assign a map $\eta_x: k \to \dgh_\mathcal{C}(x,x)$ which is a unit with respect to composition.
\end{defn}
\begin{defn}
	A \textbf{dg functor} $F:\mathcal{C}\to\mathcal{D}$ between two dg categories is an enriched functor; i.e.\ a map of objects $\mathcal{C}\to \mathcal{D}$ together with, for every pair $(x,y)\in\mathcal{C}^2$, a map of complexes \linebreak $F_{xy}:\dgh_\mathcal{C}(x,y) \to \dgh_\mathcal{D}(Fx,Fy)$. We require that $F$ satisfies the associativity condition $\mu_{Fx\ Fy \ Fz}\circ (F_{xy}\otimes F_{yz}) = F_{xz}\circ \mu_{xyz}$ and the unitality condition $F_{xx}\circ \eta_x = \eta_{Fx}$.
\end{defn}
\begin{defn}
	Let $\mathcal{C}$ be a dg category. The \textbf{homotopy category} of $\mathcal{C}$ is the $k$-linear category $[\mathcal{C}]$ whose objects are the same as $\mathcal{C}$ and whose hom-spaces are given by\linebreak $\hom_{[\mathcal{C}]}(x,y)\coloneqq H^0(\dgh_\mathcal{C}(x,y))$. Composition is inherited from $\mathcal{C}$. We sometimes write $[x,y]\coloneqq \hom_{[\mathcal{C}]}(x,y)$.
\end{defn}
\begin{defn}Let $F:\mathcal{C}\to\mathcal{D}$ be a dg functor.\begin{itemize}
		\item $F$ is \textbf{quasi-fully faithful} if all of its components $F_{xy}$ are quasi-isomorphisms.
		\item $F$ is \textbf{quasi-essentially surjective} if the induced functor $[F]:[\mathcal{C}]\to[\mathcal{D}]$ is essentially surjective.
		\item $F$ is a \textbf{quasi-equivalence} if it is quasi-fully faithful and quasi-essentially surjective.
	\end{itemize}
\end{defn}

\begin{defn}\label{dgcatlist}
	If $A$ is a dga, then $D_\mathrm{dg}(A)$ denotes the dg category of cofibrant dg modules over $A$.
\end{defn}
The dg category $D_\mathrm{dg}(A)$ is \textbf{pretriangulated}, meaning essentially that its homotopy category is canonically triangulated. It enhances the usual derived category, in the sense that one has an equivalence of triangulated categories $[D_\mathrm{dg}(A)]\cong D(A)$. Note that $\dgh$ is a model for the derived hom $\R\hom$.
\p The category of all dg categories admits a model structure whose weak equivalences are the quasi-equivalences:
\begin{thm}[Tabuada \cite{tabuadamodel}]\label{tabmod}
	The category of all small dg categories admits a (cofibrantly generated) model structure where the weak equivalences are the quasi-equivalences. The fibrations are the objectwise levelwise surjections that lift isomorphisms. Every dg category is fibrant.
\end{thm}

Now we turn to simplicial categories.
	\begin{defn}
	A \textbf{simplicial category} is a category enriched in $\sset$. A functor between simplicial categories is an enriched functor. The category of all simplicial categories is denoted $\scat$.
\end{defn}
\begin{rmk}
	Note that every simplicial category is a simplicial object in $\cat{Cat}$, but not every simplicial object in $\cat{Cat}$ is a simplicial category.
\end{rmk}
\begin{defn}
	Let  $\mathcal{D}$ be a simplicial category. The \textbf{homotopy category} of $\mathcal{D}$ is the category whose objects are the same as $\mathcal{D}$ and whose homsets are given by taking taking $\pi_0$ of the morphism complexes in $\mathcal{D}$. Composition is inherited from $\mathcal{D}$.
	\end{defn}
Bergner proved that $\scat$ admits a model structure analogous to Tabuada's (equivalently, a many-object version of the classical model structure on simplicial sets):
\begin{thm}[Bergner \cite{bergnermodelscat}]
	The category $\scat$ admits a model structure where the weak equivalences are the \textbf{Dwyer--Kan-equivalences}: those functors which induce weak equivalences on derived mapping spaces and which induce isomorphisms on $\pi_0$. We refer to these as \textbf{DK-equivalences} for short. The fibrant simplicial categories are precisely those enriched in Kan complexes. 
	\end{thm}
If $\mathcal{A}$ is an abelian category, the Dold--Kan correspondence \cite[III.2]{goerssjardine} gives a Quillen equivalence between simplicial objects in $\mathcal{A}$ and connective complexes in $\mathcal{A}$. We describe a many-object version due to Tabuada:
\begin{defn}
	Let $\mathcal{C}$ be a dg category. Let $\mathcal{C}_{\leq 0}$ denote the associated dg category obtained by taking the good truncation to nonpositive degrees of the morphism complexes. Applying the Dold--Kan correspondence to the connective morphism complexes of $\mathcal{C}_{\leq 0}$ gives a category enriched in simplicial vector spaces\footnote{Composition is given by the Alexander--Whitney map.}, and we may forget the linear structure to obtain a simplicial category we denote by $\mathcal{C}_s$. Similarly, given a simplicial category $\mathcal{D}$ we may linearise the simplicial homsets to obtain a category enriched in simplicial vector spaces, and the Dold--Kan correspondence gives us a connective dg category\footnote{Composition is now given by the Eilenberg--Zilber map.}, which we denote by $\mathcal{D}_\mathrm{dg}$.
	\end{defn}
\begin{thm}[Tabuada \cite{tabuadadgvs}]
The functors $\mathcal{C}\mapsto \mathcal{C}_s$ and $\mathcal{D}\mapsto \mathcal{D}_\mathrm{dg}$ give a Quillen adjunction $\dgcat \longleftrightarrow \scat$, where the left hand side has the Tabuada model structure and the right hand side has the Bergner model structure. 
\end{thm}
\begin{rmk}
	In fact, the above Quillen adjunction restricts to a Quillen equivalence between connective dg categories and categories enriched in simplicial vector spaces, when both are equipped with modifications of the above model structures. In the nonconnective world, one can view a dg vector space as a spectrum object in connective vector spaces; applying Dold--Kan levelwise one then gets a spectrum object in simplicial $k$-modules, or (Quillen) equivalently an $Hk$-module spectrum. This gives a Quillen equivalence between dg categories and $Hk$-module spectrally enriched categories \cite{tabuadaspectral}.
	\end{rmk}
\p We will be interested only in quasi-isomorphisms between deformations, and not arbitrary morphisms. However, we will be interested in all higher morphisms between these quasi-isomorphisms: as such, in our applications, it will suffice to work with simplicial (or dg) groupoids. We finish by describing a core functor for simplicial categories, as well as an enriched nerve construction for turning simplicial categories into simplicial sets.

\begin{defn}\label{scathequivs}	
	If $\mathcal{D}$ is a simplicial category, let $c(\mathcal{D})$ denote the subcategory on those morphisms which induce isomorphisms on $\pi_0\mathcal{D}$; in \cite{bergnermodelscat} these are called homotopy equivalences. 
\end{defn}	
We think of $c$ as a sort of core functor. Note that $c(\mathcal{D})$ is DK-equivalent to a simplicial groupoid in the sense of \cite[V.7]{goerssjardine}. For the next construction, observe that the nerve of a simplicial category is naturally a bisimplicial set, with the enrichment in one direction and the nerve in the other. One can totalise this bisimplicial set to obtain a simplicial set, and doing this in the homotopically correct manner gives us a homotopy coherent nerve for simplicial categories.
\begin{defn}		
	Let $\bar W: \cat{ssSet} \to \cat{sSet}$ denote the right adjoint of Illusie's d\'ecalage functor; see \cite[1.6]{jondmodss} for a concrete definition. Given a simplicial category $\mathcal{D}$ let $\bar W \mathcal{D}$ denote $\bar W$ applied to the nerve of $\mathcal{D}$.
\end{defn}
\begin{rmk}\label{diagw}
	In \cite{diagonalw} it is proved that for a bisimplicial set $X$, the canonical morphism $\mathrm{diag}X \to \bar W X$ is a weak equivalence; see \cite[1.7]{jondmodss} for further discussion. If $\mathcal{G}$ is a simplicial groupoid, it follows that $\bar W \mathcal{G}$ is weakly equivalent to the homotopy coherent nerve $N(\mathcal{G})$ \cite{hinichnerve}, which is a quasicategory in the sense of \cite{luriehtt}. See \cite[\S1.1.5]{luriehtt} for the details of the homotopy coherent nerve; we note that $N(\mathcal{C}_s)$ is weakly equivalent to the dg nerve of $\mathcal{C}$ \cite[\S1.3.1]{lurieha}. We use $\bar W \mathcal{G}$ instead of $N(\mathcal{G})$  for consistency with \cite{jondmodss}.
\end{rmk}
\begin{defn}
	Let $\mathcal{C}$ be a dg category. Write $\mathcal{W}(\mathcal C)\coloneqq \bar W(c(\mathcal{C}_s))$. 
\end{defn}
In other words, $\mathcal{W}(\mathcal C)$ is a model for the homotopy coherent nerve of the underlying simplicial groupoid of $\mathcal{C}_s$.
\begin{rmk}
	To define $\mathcal{W}(\mathcal C)$, we could have instead taken the `dg core' of $\mathcal{C}$, converted to a simplicial category, and applied $\bar W$ to obtain something weakly equivalent.
	\end{rmk}
\begin{rmk}\label{wholims}
	The functor $\mathcal{W}$ is a right derived functor. Indeed, because every dg category is fibrant, $\mathcal{C}\mapsto \mathcal{C}_s$ is its own right derived functor. Moreover, the composition $\bar W c$ is a right derived functor, because it is weakly equivalent to the derived mapping space functor $\R\mathrm{Map}(\bullet,-)$ from the initial simplicial category. Hence, $\mathcal{W}$ is a composition of right derived functors and so is itself a right derived functor.
\end{rmk}

\begin{rmk}\label{piloww}
	Let $\mathcal{D}$ be a simplicial category and let $N \mathcal{D}$ be its homotopy coherent nerve. One can check directly that the 0-simplices of $N \mathcal{D}$ are the objects of $\mathcal{D}$, the 1-simplices of $N \mathcal{D}$ are the morphisms of $ \mathcal{D}$, and the 2-simplices of $N \mathcal{D}$ are the triangles $$\begin{tikzcd} & y \ar[dr,"g"]& \\
		x \ar[ur,"f"] \ar[rr,"h"]&& z 
		\end{tikzcd}$$ together with a homotopy $h \to gf$. If $\mathcal C$ is a dg category, it now follows that $\mathcal {W (C)} \simeq N(c(\mathcal{C}_s))$ has homotopy in low degrees given by
	\begin{itemize}
		\item  $\pi_0\mathcal {W (C)} \cong \{\text{homotopy classes of objects in }\mathcal{C}\}$
		\item  $\pi_1\mathcal {W} (\mathcal{C},x) \cong \mathrm{hAut}(x)\coloneqq \frac{\{\text{homotopy automorphisms of }x\}}{(\text{homotopy equivalence})}$
		\end{itemize}
	recalling that a homotopy equivalence between two objects is precisely a morphism that becomes an isomorphism on $H^0$, and a homotopy automorphism is a homotopy equivalence that is also an endomorphism. When $\mathcal C$ is the dg derived category of a dga $A$ then $\mathrm{hAut}(x)$ is the group of units of the ring $\ext^0_A(x,x)$.
	\end{rmk}

	\subsection{Deformations of modules}
	Now we may define our deformation functors. Recall that a classical deformation of an $A$-module $X$ over an Artinian local ring $\Gamma$ is an $A\otimes \Gamma$-module $\tilde X$, flat over $\Gamma$, that reduces to $X$ modulo $\mathfrak{m}_\Gamma$. A derived deformation is defined similarly:
	\begin{defn}
		Let $A$ be a dga and $X$ an $A$-module. Let $\Gamma$ be an Artinian local dga. A \textbf{derived deformation} of $X$ over $\Gamma$ is a pair $(\tilde X, f)$ where $\tilde X$ is an $A\otimes \Gamma$-module and $f:\tilde X \lot_\Gamma k \to X$ is an isomorphism in $D(A)$. An \textbf{isomorphism} of derived deformations is an isomorphism $\phi:\tilde X_1 \to \tilde X_2$ in $D(A\otimes\Gamma)$ such that $f_1 = f_2 \circ ({\phi \lot_\Gamma k})$.
	\end{defn}
	Deformations are functorial with respect to algebra maps: given a map $\Gamma \to\Gamma'$ of Artinian local dgas, and a derived deformation $\tilde X$ of $X$ over $\Gamma$, then the derived base change $\tilde X \lot_\Gamma \Gamma'$ is a derived deformation of $X$ over $\Gamma'$.
	\begin{defn}
		Let $A$ be a dga and $X$ an $A$-module. The functor $\defm_A(X):\dgart \to \cat{Set}$ sends an Artinian local dga $\Gamma$ to the set $$\defm_A(X)(\Gamma)\coloneqq \frac{\{\text{derived deformations of }X \text{ over }\Gamma\}}{(\text{isomorphisms})}.$$
	\end{defn}	
We will just write $\defm(X)$ if there is no ambiguity. The functor of derived deformations is well-known to be a Deligne functor:

	\begin{thm}\label{setdefisdel}
			Let $A$ be a dga and $X$ an $A$-module. Let $E\coloneqq \R\enn_A(X)$ be the derived endomorphism dga of $X$. Then the functors $\defm_A(X)$ and $\del(E)$ are isomorphic.
		\end{thm}
	\begin{proof}
		This appears in Efimov--Lunts--Orlov \cite[6.1]{ELO}; in fact they consider the natural enhancement of  $\defm_A(X)$ to a groupoid-valued functor and show that it is equivalent to the (non-na\"\i ve) Deligne groupoid of $E$; our desired statement follows by taking $\pi_0$. The basic idea of the proof is that if $X$ is cofibrant then a deformation of $X$ is a deformation of the differential on $X\otimes_k \Gamma$, which is exactly a Maurer--Cartan element of $E\otimes \mathfrak{m}_\Gamma$. Quasi-isomorphisms between cofibrant deformations are homotopy equivalences, and the homotopy class of a deformation is its orbit under the gauge action. The result follows.
		\end{proof}
	We note that the statement of  \ref{setdefisdel} makes sense because $\del$ is invariant under quasi-isomorphisms by \ref{setdelqinvt}. We will later see a more explicit version of \ref{setdefisdel} in \ref{explicitwist}; the more explicit version will not be of use to us yet.
	\begin{rmk}\label{connectivereasons}
		Take $A=X=k$, so that the derived endomorphism algebra of $X$ is $k$. If $\Gamma$ is any Artinian dga (not necessarily connective), it is easy to see that $\del(k)(\Gamma)$ is $\mcs(\mathfrak{m}_\Gamma)/(1+\mathfrak{m}_\Gamma^0)$. In particular, putting $\Gamma=k[\varepsilon]/{\varepsilon^2}$ with $\varepsilon$ in degree $1$, we have $\del(k)(\Gamma)\cong k$. Hence if  \ref{setdefisdel} were to remain true when we take nonconnective Artinian dgas as input, we would require the $k$-vector space $k$ to have nontrivial deformations over $\Gamma$. Intuitively these should be some sort of stacky deformations: loosely, the connective part of an unbounded dga $\Gamma$ records information about the derived geometry of $\Gamma$, whilst the coconnective part of $\Gamma$ records stacky information. This is why we need to restrict to connective Artinian dgas when doing deformation theory, at least in this generality.
	\end{rmk}

	Now we may define our $\sset$-valued deformation functors.
		\begin{defn}\label{sdefmdefn}
		Let $A$ be a dga and let $X$ be an $A$-module. Let $X_\mathrm{dg}$ denote the dga $\R\enn_A(X)$ considered as a one-object dg category. There is an obvious inclusion dg functor $X_\mathrm{dg} \into D_\mathrm{dg}(A)$ whose image is $X$. If $\Gamma\in \dgart$ then there is a `reduce modulo $\mathfrak{m}_\Gamma$' dg functor $D_\mathrm{dg}(A\otimes \Gamma) \to D_\mathrm{dg}(A)$ which sends an $A\otimes\Gamma$-module $M$ to the $A$-module $M\lot_\Gamma k$. Let $\mathrm{dgDef}_A(X)(\Gamma)$ be the homotopy fibre product of dg categories $$\mathrm{dgDef}_A(X)(\Gamma)\coloneqq  X_\mathrm{dg} \times^h_{D_\mathrm{dg} (A)} D_\mathrm{dg} (A \otimes \Gamma).$$
		Let $\sdefm_A(X)(\Gamma)$ denote the simplicial set $\mathcal{W}\left(\mathrm{dgDef}_A(X)(\Gamma)\right)$. 
	\end{defn}
	It is easy to check that $\sdefm_A(X)$ is a functor from $\dgart$ to $\sset$. Observe that the dg category $\mathrm{dgDef}(X)$ is clearly `too big', since it contains many maps which are not isomorphisms after passing to the homotopy category. This is why we take the core.
	\begin{rmk}\label{dgffdefs}
		If $A$ and $A'$ are dgas with a fully faithful dg functor $D_\mathrm{dg}(A) \into D_\mathrm{dg}(A')$, then we obtain a weak equivalence $\sdefm_A(X) \xrightarrow{\simeq}\sdefm_{A'}(X)$.
	\end{rmk}	
$\sdefm$ enhances $\defm$:
\begin{prop}\label{sdefmliftssetdefm}
	There is an isomorphism of functors $\defm_A(X)\cong \pi_0\sdefm_A(X)$.
	\end{prop}
\begin{proof}
	This follows by combining \ref{defisdel} below with \ref{pi0del} and \ref{setdefisdel} to see that $$\pi_0\sdefm_A(X) \cong \pi_0 \sdel_A(X) \cong \del_A(X) \cong \defm_A(X).$$A more direct proof goes as follows: chasing the definitions, we see that $\pi_0(\sdefm_A(X)(\Gamma))$ agrees with the set of isomorphism classes in the homotopy category $[\mathrm{dgDef}_A(X)(\Gamma)]$. But the objects of $[\mathrm{dgDef}_A(X)(\Gamma)]$ are precisely the derived deformations of $X$ over $\Gamma$, and two objects are isomorphic in $[\mathrm{dgDef}_A(X)(\Gamma)]$ if and only if the corresponding derived deformations are isomorphic.
	\end{proof}
The enhanced functor of derived deformations is also known to be a Deligne functor:
	\begin{thm}\label{defisdel}
		Let $A$ be a dga and $X$ an $A$-module. Let $E\coloneqq \R\enn_A(X)$ be the endomorphism dga of $X$. Then there is a weak equivalence of $\sset$-valued functors $\sdefm_A(X) \xrightarrow{\simeq}\sdel(E)$.
	\end{thm}
	\begin{proof}This appears as e.g.\ \cite[4.2.6]{dinatalethesis} or \cite[4.6]{jonddefsartin}. The basic idea of the proof is the same as that of  the set-valued case \ref{setdefisdel}; we give a sketch. The proof of \cite[3.13]{jonddefsartin} shows that the simplicial groupoid of deformations of $X$ is the same as the simplicial subgroupoid of $\sdel(E)$ on constant objects, which we will denote $\mathcal{C}$. By \cite[\S4.1]{jonddefsartin}, $\mathcal{C}$ and $\sdel(E)$ are both derived deformation functors. The inclusion induces a map on tangent spaces which is a weak equivalence and so the two functors are weakly equivalent by an obstruction-theoretic argument.
	\end{proof}	
	We note that the statement of  \ref{defisdel} makes sense because $\del$ is invariant under quasi-isomorphisms by \ref{mcinvtqiso} and \ref{smcisdel}.	
		\begin{rmk}
		Taking $\pi_0$ of \ref{defisdel} recovers \ref{setdefisdel}, and taking fundamental groupoids recovers the groupoid-valued version appearing in \cite{ELO}.
	\end{rmk}
\begin{rmk}
	As in \ref{connectivereasons}, to extend \ref{defisdel} to non-connective dgas would require the incorporation of some sort of stacky deformations.
\end{rmk}
\begin{rmk}
	Abstractly, homotopy representability for deformation functors follows from a generalised Brown representability theorem \cite{jardinerep}; this perspective appears in \cite{glstreview}. Indeed, in the commutative setting Pridham's deformation functors \cite{unifying} are precisely the homotopy prorepresentable functors.
	\end{rmk}

	\begin{cor}\label{defismc}
	Let $A$ be a dga and $X$ an $A$-module. Let $E\coloneqq \R\enn_A(X)$ be the endomorphism dga of $X$. Then the $\sset$-valued functors $\sdefm_A(X)$ and $\smc(E)$ are weakly equivalent.
\end{cor}
\begin{proof}
	Combine \ref{defisdel} with \ref{smcisdel}.
\end{proof}

\begin{rmk}\label{cospanrmk}
	If $\mathcal{A} \to \mathcal{B} \from \mathcal{C}$ is a cospan of pointed dg categories, with homotopy pullback $\mathcal{P}$, then by \ref{wholims} we get a weak equivalence $\mathcal{WP} \simeq \mathcal{W}\mathcal{A}\times^h _ \mathcal{WB} \mathcal{W}\mathcal{C}$. Homotopy pullback diagrams of pointed simplical sets give rise to long exact Mayer--Vietoris sequences on homotopy, so we get a long exact sequence of groups and sets $$\cdots \to  \pi_1\mathcal{WP}\to \pi_1\mathcal{WA}\times \pi_1\mathcal{WC} \to \pi_1\mathcal{WB}\to \pi_0\mathcal{WP}\to \pi_0\mathcal{WA}\times \pi_0\mathcal{WC}.$$If $\mathcal{A} \to \mathcal{B}$ is homotopy fully faithful then $\pi_1\mathcal{WA} \to \pi_1\mathcal{WB}$ is the identity by \ref{piloww} and hence we get an injection $\pi_0\mathcal{WP}\into \pi_0\mathcal{WA}\times \pi_0\mathcal{WC}$. If in addition $\mathcal A$ has one object then $\pi_0\mathcal{WA}$ vanishes and so we get an injection $\pi_0\mathcal{WP}\into \pi_0\mathcal{WC}$. Applying this to the diagram of dg categories $X_\mathrm{dg} \to D_\mathrm{dg}(A) \from D_\mathrm{dg}(A\otimes \Gamma)$, where the right-hand map is pointed by the trivial deformation, we get an injection $\defm_A(X) \into \{\text{isomorphism classes of objects in }D(A\otimes \Gamma)\}$.

	\end{rmk}

	\subsection{Prorepresentability}
	We prove a prorepresentability statement for set-valued functors, and then we enhance this to $\sset$-valued functors. Essentially everything we use here (at least for set-valued functors) can be found in Loday--Vallette \cite[Chapter 2]{lodayvallette} or Positselski \cite{positselski}. We will need to use nonunital dgas in order to get the correct prorepresentability statements; we will later see that the use of nonunital dgas can be avoided if one rigidifies to consider framed deformations.

	\begin{defn}[{see e.g.\ \cite[6.2]{positselski}}]
		Let $E$ be a nonunital dga and let $C$ be a nonunital dgc. Then the complex $\hom_k(C,E)$ of $k$-vector spaces is a nonunital dga under the product given by $fg\coloneqq \mu_E \circ (f\otimes g)\circ \Delta_C$. This dga is the \textbf{convolution algebra}. A Maurer--Cartan element of the convolution algebra is a \textbf{nonunital twisting morphism}; the set of all nonunital twisting morphisms is denoted $\mathrm{Tw}(C,E)$.
	\end{defn}
\begin{rmk}
	In the (co)augmented setting, one should add the condition that twisting morphisms are zero when composed with the augmentation or coaugmentation.
	\end{rmk}
	\begin{lem}\label{twnumc}
		Let $E, Z$ be nonunital dgas, with $Z$ finite-dimensional. Then there is a natural isomorphism $$\mathrm{Tw}(Z^*,E)\cong \mcs(E \otimes Z).$$
	\end{lem}
	\begin{proof}
		There is a standard linear isomorphism $E \otimes Z \to \hom_k(Z^*,E)$, and one can check that this is a map of nonunital dgas after equipping $\hom_k(Z^*,E)$ with the convolution product. Hence the MC elements of both sides agree.
	\end{proof}

\begin{defn}\label{nubar}
	Let $E$ be a nonunital dga. The \textbf{nonunital bar construction} is the (coaugmented) dgc $$B_{\mathrm{nu}}(E)\coloneqq B(E\oplus k).$$
	\end{defn}
	We caution that if $E$ is an augmented dga, then $B_{\mathrm{nu}}(E)$ does not agree with $B(E)$ as $B_{\mathrm{nu}}(E)$ will contain elements corresponding to the unit of $E$.
	\begin{defn}
		Let $C$ be a noncounital conilpotent dgc. The \textbf{nonunital cobar construction} is the (augmented) dga $$\Omega_{\mathrm{nu}}(C)\coloneqq \Omega(C\oplus k).$$
	\end{defn}

	The functor of twisting morphisms is (up to units) representable on either side:
	\begin{thm}[{\cite[2.2.6]{lodayvallette}}]\label{barcobaradj}
		If $E$ is a nonunital dga and $C$ is a noncounital conilpotent dgc, then there are natural isomorphisms $$\hom_{\augdga}(\Omega_{\mathrm{nu}} C, E\oplus k)\cong\mathrm{Tw}(C,E)\cong \hom_{\cndgc}(C\oplus k, B_{\mathrm{nu}}E).$$
	\end{thm}
	We recall from \ref{csharp} that if $C$ is a (counital) conilpotent dgc then $C^\sharp$ denotes the pro-Artinian dga constructed by levelwise dualising the filtered system of finite sub-dgcs of $C$. If $E$ is a nonunital dga, write $B_{\mathrm{nu}}^\sharp E\coloneqq (B_{\mathrm{nu}}E)^\sharp$ for the \textbf{continuous nonunital Koszul dual}.
	
	\begin{prop}\label{mcprorep}
		Let $E$ be a nonunital dga. Then the functor $\mc(E)$ is prorepresented by $B_{\mathrm{nu}}^\sharp E$, in the sense that $\mc(E)$ and $\hom_{\ubproart}(B_{\mathrm{nu}}^\sharp E, -)$ are naturally isomorphic.
	\end{prop}
	\begin{proof}
		Let $\Gamma$ be a connective Artinian local dga. It is easy to see that $\Gamma^*$ is a conilpotent dgc. By \ref{sharpprop}, we have an isomorphism $$\hom_{\cndgc}(\Gamma^*, B_{\mathrm{nu}}E) \cong \hom_{\ubproart}(B_{\mathrm{nu}}^\sharp E, \Gamma^{*\sharp}).$$Because $\Gamma^*$ is Artinian we have an isomorphism $\Gamma^{*\sharp}\cong \Gamma$ and it follows that we have isomorphisms $$\hom_{\ubproart}(B_{\mathrm{nu}}^\sharp E, \Gamma)\cong \hom_{\cndgc}(\Gamma^*, B_{\mathrm{nu}}E)\cong \mathrm{Tw}(\mathfrak{m}_\Gamma^*,E)$$where the second isomorphism is \ref{barcobaradj}. By \ref{twnumc} we have an isomorphism $$\mathrm{Tw}(\mathfrak{m}_\Gamma^*,E)\cong \mcs(E \otimes \mathfrak{m}_\Gamma)$$and so we are done.
	\end{proof}

	Now that we have our prorepresentability result, we will enhance it to $\sset$-valued functors. This will not be too hard; we just need to identify the correct simplicial mapping spaces in $\ubproart$. Note that if $\Gamma$ is an Artinian dga, then $\Gamma\otimes\pdf $ will not be a simplicial Artinian dga, so the answer is not as simple as `replace $\Gamma$ by the simplicial resolution $\Gamma\otimes\pdf $'. We get around this by using the Quillen adjunction $\Omega \dashv B$.
	\begin{thm}\label{smcisrmap}
		Let $E$ be a nonunital dga. Then there is a weak equivalence of functors $$\smc(E)\simeq \R\mathrm{Map}_{\ubproart}(B_{\mathrm{nu}}^\sharp E, -).$$
		\end{thm}
	\begin{proof}
	If $E$ is a nonunital dga write $E^\bullet$ for the simplicial nonunital dga $E^\bullet\coloneqq E \otimes \pdf$. By definition, we have $$\smc(E)(\Gamma)\coloneqq \mcs(E^\bullet\otimes\mathfrak{m}_\Gamma).$$Applying \ref{barcobaradj} levelwise we see that we have an isomorphism $$\mcs(E^\bullet\otimes\mathfrak{m}_\Gamma)\cong \hom_{\augdga}(\Omega_{\mathrm{nu}}(\mathfrak{m}_\Gamma^*),E^\bullet \oplus \mathrm{const}(k))\cong \hom_{\augdga}(\Omega(\Gamma^*),E^\bullet \oplus \mathrm{const}(k))$$where $\mathrm{const}(k)$ denotes the constant simplicial dga on $k$. But because $\mathrm{const}(k) \to\pdf$ is a levelwise quasi-isomorphism by \ref{poincare}, we have a quasi-isomorphism of simplicial dgas $E^\bullet \oplus \mathrm{const}(k) \simeq (E\oplus k)^\bullet$. Because $\Omega(\Gamma^*)$ is cofibrant and all dgas are fibrant, it follows that we have a weak equivalence $$\hom_{\augdga}(\Omega(\Gamma^*),E^\bullet \oplus \mathrm{const}(k))\simeq \R\mathrm{Map}_{\augdga}(\Omega(\Gamma^*),E\oplus k).$$Now because $\Omega$ is part of a Quillen equivalence by \ref{bcquillen} we have $$\R\mathrm{Map}_{\augdga}(\Omega (\Gamma^*),E\oplus k)\simeq \R\mathrm{Map}_{\cndgc}(\Gamma^*,B_{\mathrm{nu}}E)$$
	and by \ref{ubpamodel} we have $$ \R\mathrm{Map}_{\cndgc}(\Gamma^*,B_{\mathrm{nu}}E)\simeq \R\mathrm{Map}_{\ubproart}(B^\sharp_{\mathrm{nu}}E,\Gamma^{*\sharp}).$$As before we have $\Gamma^{*\sharp}\cong \Gamma$ and hence we are done.
		\end{proof}
	We obtain our desired representability theorem:
		\begin{thm}[Prorepresentability]\label{proreps}
		Let $A$ be a dga and $X$ an $A$-module. Let $E\coloneqq \R\enn_A(X)$ be the derived endomorphism dga of $X$. Then the $\sset$-valued functors $\sdefm_A(X)$ and $\R\mathrm{Map}_{\ubproart}(B_{\mathrm{nu}}^\sharp E, -)$ are weakly equivalent.
	\end{thm}	
	\begin{proof}
		Combine \ref{defismc} with \ref{smcisrmap}.
	\end{proof}

\begin{rmk}\label{lurie528}When everything is connective, the deformation-theoretic part of \ref{proreps} is a theorem of Lurie. Indeed, let $A$ be a connective dga and let $X$ be a connective $A$-module. By \cite[5.2.8 and 5.2.14]{luriedagx}, there is a canonical equivalence, natural in $\Gamma\in\dgart$, $$\sdefm_A(X)(\Gamma)\simeq \rmap_{\augdga}(\Gamma^!, k\oplus \R\enn_A(X)).$$By \ref{artprop}(3) we have $\Gamma^!\cong \Omega(\Gamma^*)$, and now as in the proof of \ref{smcisrmap} we can identify the right hand side of the above equivalence with $\R\mathrm{Map}_{\ubproart}(B_{\mathrm{nu}}^\sharp \R\enn_A(X), \Gamma)$, as required. More generally, when $A$ is any $k$-dga and $X$ is any $A$-module then \cite[5.2.8]{luriedagx} together with the above argument shows that $B_{\mathrm{nu}}^\sharp \R\enn_A(X)$ prorepresents the completion of the functor $\sdefm_A(X)$ to a formal moduli problem, in the sense of \cite[3.0.3]{luriedagx}. So one alternate method to prove \ref{proreps} would be to prove that $\sdefm_A(X)$ is already a formal moduli problem. We remark that \ref{proreps} combined with Lurie's results above shows that $\sdefm_A(X)$ is indeed a formal moduli problem, even in the nonconnective case.
	\end{rmk}

\begin{rmk}\label{yekrmk}
Yekutieli obtains commutative analogues of several of our results in \cite{yekpronil}. More specifically, he works with pronilpotent dg Lie algebras and deforms over complete local commutative noetherian algebras (with maximal ideal $\mathfrak{m}$). These methods ought to adapt to our setting to give more direct proofs of some of our statements about set-valued functors and classical groupoid-valued functors. More specifically, \cite{yekpronil} shows that a quasi-isomorphism between pronilpotent dglas of the form $\mathfrak{m} \widehat{\otimes}\mathfrak{g} \to \mathfrak{m} \widehat{\otimes}\mathfrak{h}$ induces an isomorphism $\mcs(\mathfrak{m} \widehat{\otimes}\mathfrak{g}) \to \mcs(\mathfrak{m} \widehat{\otimes}\mathfrak{h})$, as well as an equivalence of (na\"{i}ve) Deligne groupoids. Presumably this is a set-level manifestation of the existence of a model structure on the category of pronilpotent dglas, for which the levelwise Chevalley--Eilenberg functor is a Quillen equivalence to the category of commutative dgcs. One would have to repeat our arguments in the commutative/Lie setting; the goal would be to obtain a weak equivalence roughly of the form $\smc(\mathfrak{g})\simeq \rmap(\mathrm{CE}^\sharp(\mathfrak{g}),-)$.
	\end{rmk}

\subsection{Inclusion-truncation adjunctions}
We have just seen that $\sdefm_A(X)$ is (homotopy) prorepresented by the possibly unbounded pro-Artinian dga $B_{\mathrm{nu}}^\sharp E$. But because we only allow connective Artinian dgas as input, one can use the inclusion-truncation adjunction to deduce a sharper result:
\begin{prop}\label{truncatedproreps}
	Let $A$ be a dga and $X$ an $A$-module. Let $E\coloneqq \R\enn_A(X)$ be the derived endomorphism dga of $X$. Then there is a weak equivalence of $\sset$-valued functors $$\sdefm_A(X)\simeq\R\mathrm{Map}_{\proart}(\tau_{\leq 0}B_{\mathrm{nu}}^\sharp E, -).$$
	\end{prop}
\begin{proof}
By \ref{inclisrq}, the inclusion functor $\proart\into\ubproart$ is right Quillen; hence its left adjoint, the (levelwise) truncation functor, is left Quillen. Because every object in $\proart$ is fibrant, the inclusion functor is its own right derived functor. Hence for $\Gamma \in \proart$ we have weak equivalences $$\R\mathrm{Map}_{\ubproart}(B_{\mathrm{nu}}^\sharp E, \Gamma)\simeq\R\mathrm{Map}_{\proart}(\mathbb{L}\tau_{\leq 0}B_{\mathrm{nu}}^\sharp E, \Gamma).$$By \ref{doubledictionary}(1), we may compute left derived functors as $$\mathbb{L}\tau_{\leq 0}B_{\mathrm{nu}}^\sharp E \simeq \tau_{\leq 0} B^\sharp\Omega(B_{\mathrm{nu}}^{\sharp\circ} E).$$But $\sharp$ and $\circ$ are inverses, so we have a quasi-isomorphism of dgas $$\Omega(B_{\mathrm{nu}}^{\sharp\circ} E) \simeq \Omega B_\mathrm{nu}(E)=\Omega B(E\oplus k) \simeq E\oplus k$$ by the definition of the nonunital bar construction. Because $B^\sharp $ sends quasi-isomorphisms to weak equivalences, we see that $$\mathbb{L}\tau_{\leq 0}B_{\mathrm{nu}}^\sharp E \simeq \tau_{\leq 0} B_{\mathrm{nu}}^\sharp E$$ and so we have a weak equivalence $$\R\mathrm{Map}_{\ubproart}(B_{\mathrm{nu}}^\sharp E, \Gamma)\simeq\R\mathrm{Map}_{\proart}(\tau_{\leq 0}B_{\mathrm{nu}}^\sharp E, \Gamma).$$Now the result follows from \ref{proreps}.
	\end{proof}
Similarly, one may restrict the input to ungraded Artinian algebras. We can prove more here, but this will require more work. We begin with two preliminary lemmas.
\begin{lem}\label{trivqa}
	Let $\cat{pro}(\cat{Art}_k)$ be the category of pro-objects in ungraded Artinian algebras, and give it the trivial model structure. Then there is a Quillen adjunction $$H^0:\proart \longleftrightarrow \cat{pro}(\cat{Art}_k):\iota$$ where the inclusion functor $\iota$ is the right adjoint.
	\end{lem}
\begin{proof}
	First we verify that the claimed adjunction is actually an adjunction. The point is that if  $A \to  A'$ is a map of Artinian dgas, where $A$ is connective and $A'$ is ungraded, then it must factor through $H^0(A)$. Hence we get an adjunction on the level of Artinian dgas, which prolongs to an adjunction between procategories. To complete the proof we just need to verify that $H^0$ is left Quillen. It obviously preserves cofibrations. It preserves weak equivalences, because a weak equivalence between connective Artinian dgas gives an isomorphism (of ungraded pro-Artinian algebras) on $H^0$ by the definition of the Pridham model structure \ref{proartmodel}.
	\end{proof}
\begin{lem}\label{sharpcommuteswithcohomology}
	Let $E$ be a nonunital dga. There is an isomorphism of pro-Artinian dgas $$(H^0B_{\mathrm{nu}} E)^\sharp \cong H^0(B_{\mathrm{nu}}^\sharp E).$$
\end{lem}
\begin{proof}
	Let $C$ be the coconnective dgc $C\coloneqq \tau_{\geq 0}B_{\mathrm{nu}} E$; observe that $H^0C$ is a subcoalgebra of $C$. The subdgcs of $C$ that are concentrated in degree zero are precisely the subcoalgebras of $H^0C$. Hence it follows that $\tau_{\geq0}(C^\sharp) \cong (H^0C)^\sharp$. But $\tau_{\geq0}(C^\sharp)=\tau_{\geq0}(\tau_{\geq 0}
	B_{\mathrm{nu}}E)^\sharp\cong \tau_{\geq0}\tau_{\leq 0}(
	B_{\mathrm{nu}}^\sharp E)\cong H^0(B_{\mathrm{nu}}^\sharp E)$. Hence we have $H^0(B_{\mathrm{nu}} E)^\sharp \cong H^0(B_{\mathrm{nu}}^\sharp E)$ as required.
	\end{proof}

\begin{prop}\label{algebraisation}
	Let $E$ be a nonunital dga and let $\Gamma$ be an ungraded Artinian local algebra. Then there is an isomorphism $$\hom_{\cat{pro}(\cat{Art}_k)}(H^0(B_{\mathrm{nu}}^\sharp E), \Gamma)\cong \hom_{\cat{aug.alg}_k}(H^0(B^*_{\mathrm{nu}}E), \Gamma).$$
	\end{prop}
\begin{proof}
	First apply \ref{sharpcommuteswithcohomology} to get an isomorphism 
	$$\hom_{\cat{pro}(\cat{Art}_k)}(H^0(B_{\mathrm{nu}}^\sharp E), \Gamma)\cong \hom_{\cat{pro}(\cat{Art}_k)}((H^0B_{\mathrm{nu}} E)^\sharp , \Gamma)$$ and then apply \ref{cdlem}(1) to get an isomorphism $$\hom_{\cat{pro}(\cat{Art}_k)}((H^0B_{\mathrm{nu}} E)^\sharp , \Gamma)\cong \hom_{\cat{aug.alg}_k}((H^0B_{\mathrm{nu}} E)^*,\Gamma).$$The linear dual is exact, so we have $(H^0B_{\mathrm{nu}} E)^*\cong H^0(B^*_{\mathrm{nu}} E)$ and the result follows.
	\end{proof}

\begin{thm}[Representability for ungraded algebras]\label{repug}
Let $A$ be a dga and $X$ an $A$-module. Let $E\coloneqq \R\enn_A(X)$ be the derived endomorphism dga of $X$.  Then the functors $\defm_A(X)\vert_{\cat{Art}_k}$ and $\hom_{\cat{aug.alg}_k}(H^0(B^*_{\mathrm{nu}}E), -)$ are isomorphic.
	\end{thm}
\begin{proof}
		Let $\Gamma$ be an ungraded Artinian algebra. Then \ref{truncatedproreps} gives a weak equivalence $$\sdefm_A(X)(\Gamma)\simeq\R\mathrm{Map}_{\proart}(\tau_{\leq 0}B_{\mathrm{nu}}^\sharp E, \Gamma).$$Applying the Quillen adjunction of \ref{trivqa}, and noting that $H^0$ is its own left derived functor by definition of the Pridham model structure \ref{proartmodel}, we get weak equivalences  $$\sdefm_A(X)(\Gamma)\simeq\R\mathrm{Map}_{\proart}(\tau_{\leq 0}B_{\mathrm{nu}}^\sharp E, \Gamma)\simeq \R\mathrm{Map}_{\cat{pro}(\cat{Art}_k)}(H^0(B_{\mathrm{nu}}^\sharp E), \Gamma).$$Taking $\pi_0$ of this weak equivalence gets us isomorphisms
		$$\pi_0\sdefm_A(X)(\Gamma)\cong \hom_{\mathrm{Ho}(\cat{pro}(\cat{Art}_k))}(H^0(B_{\mathrm{nu}}^\sharp E), \Gamma) \cong \hom_{\cat{pro}(\cat{Art}_k)}(H^0(B_{\mathrm{nu}}^\sharp E), \Gamma).$$By \ref{sdefmliftssetdefm}, we have an isomorphism $\pi_0\sdefm_A(X)(\Gamma)\cong \defm_A(X)(\Gamma)$, and by \ref{algebraisation}, we have an isomorphism $\hom_{\cat{pro}(\cat{Art}_k)}(H^0(B_{\mathrm{nu}}^\sharp E), \Gamma)\cong \hom(H^0(B^*_{\mathrm{nu}}E), \Gamma)$. Putting it all together, we get isomorphisms $$\defm_A(X)(\Gamma)\cong \hom_{\cat{pro}(\cat{Art}_k)}(H^0(B_{\mathrm{nu}}^\sharp E), \Gamma) \cong \hom_{\cat{aug.alg}_k}(H^0(B^*_{\mathrm{nu}}E), \Gamma)$$ as required.
	\end{proof}

When deforming modules over rings, a derived deformation over an ungraded algebra is the same thing as a classical deformation; we briefly recall the classical deformation functor.
\begin{defn}
	Let $A$ be a $k$-algebra and let $X$ be an $A$-module. Let $\Gamma \in \cat{Art}_k$ be an ungraded Artinian algebra. A \textbf{classical deformation} of $X$ over $\Gamma$ is an $A\otimes\Gamma$-module $\tilde X$, flat over $\Gamma$, such that $\tilde X \otimes_\Gamma k \cong X$. An \textbf{isomorphism} of classical deformations is an isomorphism $\phi:\tilde X_1 \to \tilde X_2$ of $A\otimes\Gamma$-modules such that $f_1 = f_2 \circ ({\phi \otimes_\Gamma k})$. The functor of classical deformations of $X$ sends an ungraded Artinian algebra $\Gamma$ to the set $$\defm^\mathrm{cl}_A(X)(\Gamma)\coloneqq \frac{\{\text{classical deformations of }X \text{ over }\Gamma\}}{(\text{isomorphisms})}.$$
	\end{defn}
\begin{prop}\label{cldefs}Let $A$ be a $k$-algebra and let $X$ be an $A$-module. There is a natural isomorphism
	$\defm^\mathrm{cl}_A(X)\cong\defm_A(X)\vert_{\cat{Art}_k}$.
	\end{prop}
\begin{proof}
	Let $\Gamma \in \cat{Art}_k$. Let $\tilde X$ be a classical deformation of $X$ over $\Gamma$. It is easy to see that $\tilde X \lot_\Gamma k \cong X$ inside the derived category $D(A\otimes \Gamma)$. Hence $\tilde X$ is a derived deformation of $X$. Moreover, if two classical deformations are isomorphic, they are clearly isomorphic as derived deformations, and hence we obtain a map of sets $\Phi:\defm^\mathrm{cl}_A(X)(\Gamma)\to \defm_A(X)(\Gamma)$. It is injective, because $A\otimes\Gamma \text{-}\cat{Mod}$ embeds in $D(A\otimes \Gamma)$; here is where we are using that $A$ is ungraded. Observe that if $\tilde X\in D(A\otimes \Gamma)$ is a derived deformation of $X$ over $\Gamma$, then it must actually be (quasi-isomorphic to) an $A\otimes\Gamma$-module concentrated in degree zero. Because we have $\tilde X \lot_\Gamma k \simeq X$, we have $\tor^\Gamma_i(\tilde X,k)\cong 0$ for $i>0$. Because $\Gamma$ is Artinian local, this implies Tor-vanishing for all $\Gamma$-modules, and hence $\tilde X$ is a flat $\Gamma$-module. Hence $\tilde X$ is in the image of $\Phi$, and so $\Phi$ is a surjection and thus an isomorphism of sets.
	\end{proof}

Finally, we can deduce a representability theorem for classical deformations:
\begin{thm}[Representability for classical deformations]\label{repclass}
Let $A$ be a $k$-algebra and let $X$ be an $A$-module.	Let $E\coloneqq \R\enn_A(X)$ be the derived endomorphism dga of $X$. Then the functors $\defm^\mathrm{cl}_A(X)$ and $\hom_{\cat{aug.alg}_k}(H^0(B^*_{\mathrm{nu}}E), -)$ are isomorphic.
	\end{thm}
\begin{proof}
	This follows immediately from \ref{repug} and \ref{cldefs}.
	\end{proof}

\subsection{Twisted differentials}
We outline in detail the promised correspondence between deformations and twisted differentials sketched in the proof of \ref{setdefisdel}. This will be useful to us soon in \ref{wefunc} where we establish some extra functoriality in our prorepresentability results, and later when we will try to compute with universal prodeformations in \ref{uprodefholimthm}.

\p Let $E$ be  a dga and let $M$ be a right $A$-module. Let $\Gamma$ be an Artinian dga and let $\pi:\Gamma^* \to E$ be a twisting cochain. The dg vector space $M\otimes_k \Gamma$ is an $E$-$\Gamma$-bimodule; let $d$ denote its differential. The underlying graded module of $M\otimes_k\Gamma$ admits an endomorphism $\pi$ defined by $\pi(m\otimes x) = (-1)^{\vert m \vert}\sum_{x=yz}m.\pi(y)\otimes z$. Because $\pi$ was a twisting cochain, the sum $d+\pi$ is a differential on $M\otimes_k\Gamma$. We denote the dg $E$-$\Gamma$-bimodule $(M\otimes_k\Gamma,d+\pi)$ by $M\otimes_\pi\Gamma$ and refer to it as the \textbf{twist of $M$ by $\pi$}.

\begin{rmk}\label{posfdrmk}
	If $\Gamma$ is an Artinian dga, then the categories of locally finite right $\Gamma$-modules and locally finite right $\Gamma^*$-comodules are equivalent, via the linear dual functor. In particular, if $M$ is locally finite, then $M\otimes_\pi \Gamma$ is also locally finite, and dual to the $\Gamma^*$-comodule $M^*\otimes^\pi \Gamma^*$ constructed in \cite[6.2]{positselski}.
\end{rmk}

Let $A$ be a dga, and let $X$ be an $A$-module. Let $\tilde X \to X$ be a cofibrant resolution of $X$ and put $E\coloneqq \enn_A(\tilde{ X}) \simeq \R\enn_A(X)$. Let $\Gamma$ be an Artinian dga and let $f:B_\mathrm{nu}^\sharp \R\enn_A(X) \to \Gamma$ be a morphism in the homotopy category of pro-Artinian dgas. We may as well replace $f$ with a morphism $g:B_\mathrm{nu}^\sharp E \to \Gamma$ in the homotopy category. Applying the linear dual, we obtain a map $g^\circ:\Gamma^* \to B_\mathrm{nu}E$ in the homotopy category of dgcs. Because bar constructions are fibrant, we can represent $g^\circ$ as the homotopy class of a genuine coalgebra map $\Gamma^* \to B_\mathrm{nu}E$. This corresponds to a twisting cochain $\pi: \Gamma^* \to E\oplus k$. Note that $X$ is still a $E\oplus k$-module. We write $ X \lot_f \Gamma \coloneqq \tilde X \otimes_\pi \Gamma$ and refer to it as the\footnote{A priori, this notation is abusive as it depends on both the choice of the cofibrant resolution of $X$ and on the choice of lift of $g^\circ$. However, we will see that up to quasi-isomorphism these choices do not matter.} \textbf{twist of $X$ by $f$}. Note that $ X \lot_f \Gamma$ is an $A$-$\Gamma$-bimodule. It in fact is a deformation of $X$, and moreover we obtain all deformations in this manner: the following is essentially a more explicit version of \ref{setdefisdel}.
\begin{prop}\label{explicitwist}
	Let $A$ be a dga and let $X$ be an $A$-module. Let $\Gamma$ be an Artinian dga. Then the map \begin{align*}
		\hom_{\mathrm{Ho}(\ubproart)}(B_{\mathrm{nu}}^\sharp \R\enn_A(X) , \Gamma) &\to \defm_A(X) (\Gamma)\\
		f &\mapsto [ X \lot_f \Gamma]
	\end{align*}
	is a well-defined bijection. In particular, the twist $ X \lot_f \Gamma$ is well-defined up to isomorphism of derived deformations.
\end{prop}
\begin{proof}
	As above, let $\pi$ be the associated twisting cochain of $f$. The proof of \ref{smcisrmap} shows that the map $\hom_{\mathrm{Ho}(\ubproart)}(B_{\mathrm{nu}}^\sharp E, \Gamma)  \to \del(E)(\Gamma)$ which sends $f$ to the homotopy class $[\pi]$ is a well-defined bijection. The proof of \cite[6.1]{ELO} shows that the map $\del(E)(\Gamma) \to \defm_A(X)(\Gamma)$ which sends $[\pi]$ to the equivalence class of the deformation $X \lot_f \Gamma$ is a bijection.
\end{proof}
We use this correspondence to give an extra naturality statement to our prorepresentability theorems. Let $f:A \to B$ be a map of dgas. There is an induced adjunction $$L\coloneqq -\lot_A B:D(A)\longleftrightarrow D(B): \R\hom_B(B,-)\eqqcolon R$$between their derived categories which lifts to the level of dg categories. Fix $\Gamma \in \dgart$. If $X$ is an $A$-module, it is not hard to see that we have an induced map $\mathrm{dgDef}_A(X)(\Gamma) \to \mathrm{dgDef}_B(LX)(\Gamma)$, natural in both $\Gamma$ and $X$. Hence we get an induced map of functors $\sdefm_A(X)(\Gamma) \to \sdefm_B(LX)(\Gamma)$. Similarly, if $Y$ is a $B$-module, we get an induced map $\sdefm_B(Y)(\Gamma) \to \sdefm_A(RY)(\Gamma)$. It is easy to see that these maps are natural in $\Gamma$; in other words we have natural transformations \begin{align*}
	Lf: \sdefm_A(X)&\to \sdefm_B(LX)\\
	Rf: \sdefm_B(Y) &\to \sdefm_A(RY).
	\end{align*}

\begin{rmk}
These maps are also natural in $X$ and $Y$ separately, in the following sense. For $A$ fixed, let $\mathrm{Pairs}(A)$ be the category of pairs $(X,\Gamma)$ with $X$ an $A$-module and $\Gamma$ an Artinian dga, and maps the obvious ones. Given $f:A \to B$ there is a base change functor $\mathrm{LPairs}(f):\mathrm{Pairs}(A) \to \mathrm{Pairs}(B)$ sending $(X,\Gamma)$ to $(LX,\Gamma)$. Then $\sdefm_A$ is a functor $\mathrm{Pairs}(A) \to \sset$, and we have a natural transformation $$\sdefm_A \to \sdefm_B \circ \mathrm{LPairs}(f).$$The correct way to encode this data is probably as some sort of 2-functor.
\end{rmk}
Because $L$ and $R$ enhance to dg functors, we get dga maps \begin{align*}
	Lf: \R\enn_A(X)&\to \R\enn_B(LX)\\
	Rf: \R\enn_B(Y)&\to \R\enn_A(RY)
\end{align*}
inducing maps of their associated Deligne functors.
\begin{prop}\label{wefunc}
Let $A$ be a dga and let $X$ be an $A$-module. Let $f:A \to B$ be a dga map and let $Y$ be a $B$-module. Then there are homotopy commutative squares $$\begin{tikzcd} \sdefm_A(X) \ar[r,"Lf"]& \sdefm_B(LX) \\ \R\mathrm{Map}_{\ubproart}(B_{\mathrm{nu}}^\sharp \R\enn_A(X), -)\ar[u,"\simeq"] \ar[r,"\left(B^\sharp_\mathrm{nu}Lf\right)^*"]& \R\mathrm{Map}_{\ubproart}(B_{\mathrm{nu}}^\sharp  \R\enn_B(LX), -)\ar[u,"\simeq"]
	\end{tikzcd}$$\phantom{qq}
$$\begin{tikzcd} \sdefm_B(Y) \ar[r,"Rf"]& \sdefm_A(RY) \\ \R\mathrm{Map}_{\ubproart}(B_{\mathrm{nu}}^\sharp \R\enn_B(Y), -)\ar[u,"\simeq"] \ar[r,"\left(B^\sharp_\mathrm{nu}Rf\right)^*"]& \R\mathrm{Map}_{\ubproart}(B_{\mathrm{nu}}^\sharp  \R\enn_A(RY), -)\ar[u,"\simeq"]
\end{tikzcd}$$

with vertical maps the weak equivalences of \ref{proreps}.
	\end{prop}
\begin{proof}
	Let $\Gamma$ be an Artinian dga. We begin with the square for $Y$, as it is slightly easier. The statement boils down to showing that if $\pi:\Gamma \to \R\enn_B(Y)\oplus k$ is a twisting cochain, with induced twisting cochain $\rho:\Gamma \to \R\enn_B(Y)\oplus k \to \R\enn_A(RY)\oplus k$, then we have $Rf(Y\otimes_\pi \Gamma)\simeq Y\otimes_\rho \Gamma$. But this is clear: because $R$ is the forgetful functor, the action of $\Gamma$ on $Y$ via $\rho$ is precisely the action via $\pi$, but where we forget the $B$-module structure and just view it as an $A$-module. For the $X$ square, we can argue similarly: if $\pi:\Gamma \to \R\enn_A(X)\oplus k$ is a twisting cochain, with induced twisting cochain $\rho:\Gamma \to \R\enn_A(X)\oplus k \to \R\enn_B(LX)\oplus k$, we need to show that $Lf(X\otimes_\pi \Gamma) \simeq (LX) \otimes_\rho\Gamma$. To do this, recall that we may assume that $X$ is cofibrant, and this then reduces to showing that $\left(X\otimes_\pi \Gamma\right)\otimes_AB \simeq \left(X\otimes_AB\right) \otimes_\rho\Gamma$. As graded objects, these two are identical, so we need only worry about the differential. Working out the differential, we see that it suffices to show that $\left(x.\pi(\gamma) \right)\otimes b = (x\otimes b).\rho(\gamma) $ for all $x,\gamma, b$. But this follows easily from the definition of $\rho$.
	\end{proof}

\section{Framed deformations}
Let $A$ be a dga and let $X$ be an $A$-module. By \ref{proreps}, the functor of deformations of $X$ is prorepresented by the pro-Artinian dga $B_{\mathrm{nu}}^\sharp E$, where $E\coloneqq \R\enn_A(X)$. If  $E$ happens to be augmented, does the functor prorepresented by $B^\sharp E$ admit a deformation-theoretic interpretation? In this section we will show that when deforming a one-dimensional module over a ring, one can interpret $\R\mathrm{Map}_{\cat{pro}(\cat{dgArt}_k)}(B^\sharp E, -)$ in terms of rigidified deformations: the data we will need to add to deformations to rigidify will be that of a framing. A framed deformation of a module $X$ over a dga $A$ is essentially a deformation of $X$ that restricts to the trivial deformation of the underlying vector space of $X$. We give a prorepresentability result for framed deformations of a point, as well as a concrete identification of the set-valued functor of framed deformations. Restricting our prorepresentability result to the classical case, we obtain a new proof of a representability result of Segal. In the final part, we give an application involving Braun--Chuang--Lazarev's derived quotient, and give a new proof of a representability theorem of Donovan--Wemyss.

\subsection{Framings}
In this part we will investigate framed deformations in generality; in the next part we will restrict to one-dimensional modules to get a prorepresentability statement. 

\p Let $A$ be a dga and let $X$ be an $A$-module. Via the forgetful functor from $A$-modules to vector spaces, a deformation of the $A$-module $X$ gives rise to a deformation of the vector space $X$; this gives us a natural transformation $\sdefm_A(X) \to \sdefm_k(X)$. Observe that the functor $\sdefm_k(X)$ is pointed by the trivial deformation.
\begin{defn}\label{framingdefsset}
	Let $A$ be a dga and let $X$ be an $A$-module. The functor of \textbf{framed deformations} of $X$ is the homotopy fibre $$\frmdef_A(X)\coloneqq \mathrm{hofib}\left(\sdefm_A(X) \to \sdefm_k(X)\right).$$
	\end{defn}
	In other words, one restricts to those deformations of $X$ which are trivial deformations of the underlying dg-vector space. As for the unframed case, we would like to check that $\pi_0\frmdef_A(X)$ has a natural interpretation as `framed deformations modulo framed isomorphisms'. We start by defining the desired quotient functor.
	\begin{defn}\label{framingdefset}
		Let $A$ be a dga and let $X$ be an $A$-module. Let $\Gamma$ be a connective Artinian dga. Let $\tilde X$ be a deformation of $X$ over $\Gamma$. A \textbf{framing} of $\tilde X$ is a quasi-isomorphism $\nu:U(\tilde X) \to X\lot_k\Gamma$, where $U:D(A\otimes \Gamma) \to D(k \otimes \Gamma)$ is the forgetful functor. A \textbf{framed deformation} of $X$ is a pair $(\tilde X, \nu)$ consisting of a deformation and a framing. A \textbf{framed isomorphism} $F:(X,\nu_X) \to (Y,\nu_Y)$ is an isomorphism $F:X \to Y$ of deformations satisfying $\nu_X=\nu_Y\circ UF$. The set-valued functor of framed deformations of $X$ is $$\frmdefset_A(X)(\Gamma)\coloneqq \frac{\{\text{framed deformations of }X \text{ over }\Gamma\}}{(\text{framed isomorphisms})}.$$
	\end{defn}
	Our proof that $\pi_0\frmdef_A(X)\cong \frmdefset_A(X)$ will reduce to the case of groupoid-valued functors, where the situation is better understood. We will need a few facts about the homotopy theory of groupoids, which we take from \cite[\S6]{stricklandgpds}.
	\begin{prop}The category of groupoids admits a model structure where the weak equivalences are the equivalences of categories, the fibrations are the isofibrations, and the cofibrations are the functors injective on objects.
	\end{prop}
	\begin{prop}\label{grpdhfiblem}
		Suppose that $* \to B$ is a pointed groupoid, and $F: A \to B$ is a functor between groupoids. The homotopy fibre of $F$ is the groupoid with objects the pairs $(a,u)$ with $a \in A$ and $u:Fa \to * \ $, and morphisms $(a,u) \to (a',u')$ those maps $v:a \to a'$ such that $u=u'\circ Fv$.
	\end{prop}
	The set of connected components of a homotopy fibre of simplicial sets can be computed by taking the homotopy fibre on the level of fundamental groupoids:
	\begin{lem}\label{hofibgrpd}
		Let $Z$ be a pointed simplicial set, and let $f:Y \to Z$ be a map of simplicial sets. Let $\Pi_1$ be the fundamental groupoid functor. Then there is an isomorphism $$\pi_0(\mathrm{hofib}(f))\cong \pi_0(\mathrm{hofib}(\Pi_1f)).$$
	\end{lem}
	\begin{proof}
		Factorise $f = Y \xrightarrow{f'} Y' \xrightarrow{f''} Z$ into a weak equivalence followed by a fibration. Because $\Pi_1$ preserves weak equivalences and fibrations, the diagram $\Pi_1Y \xrightarrow{\Pi_1f'} \Pi_1Y' \xrightarrow{\Pi_1f''} \Pi_1Z$ is a factorisation of $\Pi_1f$ into a weak equivalence followed by a fibration. Because the classical model structure on simplicial sets is right proper \cite[II.9.6]{goerssjardine} the homotopy fibre of $f$ is weakly equivalent to the fibre of $f''$. Moreover because every groupoid is fibrant, the model structure on groupoids is also right proper, and so the homotopy fibre of $\Pi_1f$ is weakly equivalent to the fibre of $\Pi_1f''$. It is not hard to see that if $g$ is a map of simplicial sets with pointed codomain, then there is an isomorphism $\pi_0\mathrm{fib}(g)\cong \pi_0\mathrm{fib}(\Pi_1g)$. Hence we have isomorphisms $$\pi_0(\mathrm{hofib}(f))\cong \pi_0(\mathrm{fib}(f''))\cong \pi_0(\mathrm{fib}(\Pi_1f''))\cong \pi_0(\mathrm{hofib}(\Pi_1f))$$as required.
	\end{proof}
			\begin{thm}\label{sfrmdefslifts}
				Let $A$ be a dga and let $X$ be an $A$-module. Then there is an isomorphism of functors $\pi_0\frmdef_A(X)\cong \frmdefset_A(X)$.
\end{thm}

\begin{proof}
	Let $\mathscr{H}$ denote the homotopy fibre $$\mathscr{H}\coloneqq\mathrm{hofib}\left(\Pi_1\sdefm_A(X) \to \Pi_1\sdefm_k(X)\right).$$
	By the definition of $\frmdef_A(X)$ along with \ref{hofibgrpd}, we have an isomorphism $\pi_0\frmdef_A(X)\cong\pi_0\mathscr{H}$. As in \cite{ELO}, the groupoid $\Pi_1 \sdefm_A(X)$ has objects the $A$-deformations of $X$, and the morphisms are certain homotopy classes of isomorphisms of deformations (the precise structure of the homotopies is not important here). Using \ref{grpdhfiblem} it is not hard to see that $\mathscr{H}$ is the groupoid whose objects are the framed deformations, and whose morphisms are the homotopy classes of framed isomorphisms. Certainly if two framed deformations are linked by a framed isomorphism, they are linked in $\mathscr{H}$, and conversely if two framed deformations are linked in $\mathscr{H}$ there must be some framed isomorphism linking them. So it follows that $\pi_0\mathscr{H}$ is isomorphic to $\frmdefset_A(X)$, which is precisely the claim.
\end{proof}

\begin{rmk}
	The groupoid $\mathscr{H}$ appearing in the proof of \ref{sfrmdefslifts} is not the fundamental groupoid of $\frmdef_A(X)$; it is rather just a groupoid with the same $\pi_0$. We could use the equivalence$$\Pi_1\frmdef_A(X)\simeq \Pi_1\mathrm{hofib}\left(\Pi_2\sdefm_A(X) \to \Pi_2\sdefm_k(X)\right)$$ to compute $\Pi_1\frmdef_A(X)$, but this would require a detailed analysis of homotopies.
	\end{rmk}

	\subsection{Prorepresentability}
	Recall that we call a module $X$ over a $k$-dga $A$ \textbf{one-dimensional} if it is concentrated in degree zero and $\dim_k(X^0)=1$; i.e.\ as a $k$-module we simply have $X\cong k$. We specialise our above results to the situation when $X$ is a one-dimensional module, where we can obtain a prorepresentability result for framed deformations. We specialise further to the set-valued and classical deformation functors, where we can recover a result of Segal.
			\begin{lem}\label{extzlem}
		Let $A$ be a connective dga and let $S$ be a one-dimensional $A$-module. Then $\ext^0_A(S,S)\cong k$. In particular, $\R\enn_A(S)$ is augmented.
	\end{lem}
	\begin{proof}
		The idea is to use the t-structure on $D(A)$, which exists because $A$ is connective. By definition we have an isomorphism $\ext^0_A(S,S)\cong \enn_{D(A)}(S)$. By connectivity, the derived category $D(A)$ admits a t-structure such that the inclusion $\cat{Mod}$-$H^0(A) \into D(A)$ is an equivalence onto the heart, with inverse given by taking zeroth cohomology \cite{hkmtstrs, amiotcluster, kelleryangmutn, kalckyang}. In particular, we have an isomorphism $\enn_{D(A)}(S)\cong \enn_{H^0(A)}(H^0S)$. But $S$ is one-dimensional, hence concentrated in degree zero, so we have an isomorphism $\enn_{D(A)}(S)\cong \enn_{H^0(A)}(S)$. But it is clear that $\enn_{H^0(A)}(S)$ is just $k$; it can be no bigger because $\enn_k(S)$ is just $k$. For the second statement, let $E$ be any model for $\R\enn_A(S)$ and observe that the natural dga map $E\to H^0E \cong \ext^0_A(S,S)\cong k$ is an augmentation. Since it did not matter which model we chose, this is quasi-isomorphism invariant and so we can say that $\R\enn_A(S)$ is augmented.
	\end{proof}

	\begin{thm}[Prorepresentability for framed deformations]\label{prorepfrm}
	Let $A$ be a connective dga and let $S$ be a one-dimensional $A$-module. Let $E\coloneqq \R\enn_A(S)$ be the derived endomorphism dga of $S$. Then there is a weak equivalence $$\frmdef_A(S)\simeq \R\mathrm{Map}_{\ubproart}(B^\sharp E, -).$$
\end{thm}
	\begin{proof}
	The idea is that $\bar E \to E \to k$ is a homotopy fibre sequence of dgas. By \ref{extzlem}, $E$ is augmented. Because the augmentation $E \to k$ is surjective, there is a weak equivalence of nonunital dgas $\mathrm{hofib}(E \to k)\simeq \bar E$ where $\bar E$ denotes the augmentation ideal. The homotopy fibre sequence $$\bar E \to E \to k$$ gives us a homotopy fibre sequence $$BE \to B_{\mathrm{nu}}E \to B_{\mathrm{nu}}k$$ of coalgebras, because $B$ is right Quillen and all dgas are fibrant. Applying $\sharp$ now gives us a homotopy cofibre sequence $$B^\sharp_{\mathrm{nu}}k \to B^\sharp_{\mathrm{nu}}E \to B^\sharp E$$ of pro-Artinian dgas, and taking derived mapping spaces now gives a homotopy fibre sequence $$\R\mathrm{Map}_{\ubproart}(B^\sharp E, -)\to \R\mathrm{Map}_{\ubproart}(B^\sharp _{\mathrm{nu}}E, -) \to \R\mathrm{Map}_{\ubproart}(B^\sharp _{\mathrm{nu}}k, -)$$of representable functors. Because $S\cong k$ as $k$-modules, we have $\R\enn_k(S)\simeq k$. By \ref{wefunc}, the weak equivalences of \ref{proreps} assemble into a commutative diagram $$\begin{tikzcd}
	\R\mathrm{Map}_{\ubproart}(B^\sharp _{\mathrm{nu}}E, -) \ar[r]\ar[d,"\simeq"]& \R\mathrm{Map}_{\ubproart}(B^\sharp _{\mathrm{nu}}k, -)\ar[d,"\simeq"] \\
	\sdefm_A(S) \ar[r]& \sdefm_k(S)
	\end{tikzcd}$$where the upper horizontal map is induced by $E \to k$ and the lower horizontal map is the forgetful map. It follows that the homotopy fibres of the rows are weakly equivalent, which is the desired claim.
\end{proof}
\begin{rmk}
	It is not hard to compute that $B^\sharp_\mathrm{nu}k$ is $k\llbracket x \rrbracket$ with $x$ in degree 1, equipped with its obvious pro-Artinian structure. In particular, if $\Gamma$ is connective Artinian then there is a weak equivalence $$\R\mathrm{Map}_{\ubproart}(B^\sharp _{\mathrm{nu}}k, \Gamma)\simeq \R\mathrm{Map}_{\proart}(k, \Gamma)$$by the inclusion-truncation adjunction.
		\end{rmk}
Now we have a prorepresentability result, we can specialise to what happens on $\pi_0$.
\begin{prop}[Set-valued prorepresentability for framed deformations]\label{prorepfrmset}
	Let $A$ be a connective dga and let $S$ be a one-dimensional $A$-module. Let $E\coloneqq \R\enn_A(S)$ be the derived endomorphism dga of $S$. Then there is an isomorphism $$\frmdefset_A(S)\simeq \hom_{\mathrm{Ho}(\ubproart)}(B^\sharp E, -).$$
	\end{prop}
\begin{proof}
	Take $\pi_0$ of \ref{prorepfrm} and appeal to \ref{sfrmdefslifts}.
	\end{proof}
The functors $\frmdefset$ and $\defm$ are very similar, and in fact the latter is the quotient of the former by inner automorphisms:
\begin{thm}[Set-valued representability up to automorphism]\label{gensegal}
		Let $A$ be a connective dga and let $S$ be a one-dimensional $A$-module. Let $E\coloneqq \R\enn_A(S)$ be the derived endomorphism dga of $S$. Then there is an isomorphism
		$$\defm_A(S)(\Gamma)\cong \frac{\hom_{\mathrm{Ho}(\ubproart)}(B^\sharp E, \Gamma) }{(\text{\normalfont inner automorphisms of }\Gamma)}.$$
	\end{thm}
\begin{proof}Note that this is a prorepresentability result for unframed deformations. The idea is that an analysis of the set-valued Deligne functor shows that $\del(E)$ is the quotient of $\del(\bar E)$ by inner automorphisms, and the result then follows via known equivalences. Combining \ref{prorepfrm} with the weak equivalence of \ref{defisdel}, we obtain a commutative square
	$$\begin{tikzcd} 
	\frmdef_A(S) \ar[r] \ar[d,"\simeq"] &\sdefm_A(S) \ar[d,"\simeq"] \\
	\sdel(\bar E)\ar[r]& \sdel(E)
	\end{tikzcd}$$where the vertical maps are weak equivalences. Taking $\pi_0$ of this square we get a commutative square $$\begin{tikzcd} 
	\frmdefset_A(S) \ar[r] \ar[d,"\cong"] &\defm_A(S) \ar[d,"\cong"] \\
	\del(\bar E)\ar[r]& \del(E)
	\end{tikzcd}$$with vertical maps isomorphisms, where the identification of the top line is \ref{sfrmdefslifts} and \ref{sdefmliftssetdefm} and the identification of the bottom line is \ref{pi0del}. If $\Gamma$ is a connective Artinian dga, we see that $\mcs(\bar E)(\Gamma) = \mcs(E)(\Gamma)$. Moreover, it is not hard to check that the sequence of gauge groups $\ggr(\bar E) \to \ggr(E) \to \ggr(k)$ is split exact. It follows that $\del(\bar E) \to \del(E)$ is a surjection, and moreover we can identify $\del(E)$ as the quotient of $\del(\bar E)$ by the gauge group $\ggr(k)$. The weak equivalences $\sdel \simeq \R\mathrm{Map}_{\ubproart}(B^\sharp_\mathrm{nu} , -)$ give isomorphisms $\del\cong\hom_{\mathrm{Ho}(\ubproart)}(B^\sharp _{\mathrm{nu}}, -)$, compatible with the action of  $\ggr(k)$, and it follows that we have an isomorphism 
	$$\defm_A(S)\cong \hom_{\mathrm{Ho}(\ubproart)}(B^\sharp E, -) / \ggr(k).$$
	Because $\Gamma$ is connective the units of $\Gamma$ are the same as the units of $\Gamma^0$, and because $\Gamma^0$ is an Artinian local algebra its units are $1+\mathfrak{m}^0_\Gamma$. It follows by definition that $\ggr(k)(\Gamma)\cong \Gamma^\times$. Moreover, $\Gamma^\times$ acts on $\hom_{\mathrm{Ho}(\ubproart)}(B^\sharp E, \Gamma)$ by inner automorphisms, and the result follows.
	\end{proof}
\begin{rmk}
	The homotopy fibre sequence $$\frmdef_A(S) \to \sdefm_A(S) \to \sdefm_k(S)$$ yields a long exact sequence on homotopy groups $$\cdots\to\pi_1\sdefm_k(S)\to\frmdefset_A(S)\to \defm_A(S)\to*$$from which we can immediately see that the natural map $\frmdefset_A(S)\to\defm_A(S)$ is a surjection: in other words, every deformation admits a framing. The proof above identifies the action of  $\pi_1\sdefm_k(S)$ on $\frmdefset_A(S)$ as the gauge action of $\ggr(k)$.
	\end{rmk}
We finish by recovering a theorem of Segal:
	\begin{thm}[{cf.\ \cite[2.13]{segaldefpt}}]\label{mysegal}
	Let $A$ be a $k$-algebra and let $S$ be a one-dimensional $A$-module. Assume that $\ext^1_A(S,S)$ is finite-dimensional. Let $T$ be the $k$-algebra$$T\coloneqq\frac{\hat{T}(\ext^1_A(S,S)^*)}{m^*(\ext^2_A(S,S)^*)}$$where $m$ is the homotopy Maurer--Cartan function (\ref{hmcdef}). Then there is an isomorphism $$\defm^\mathrm{cl}_A(S)(\Gamma)\cong\frac{\hom_{\cat{aug.alg}_k}(T, \Gamma)}{ (\text{\normalfont inner automorphisms of $\Gamma$})}.$$
\end{thm}
\begin{proof}Let $E$ be the derived endomorphism algebra of $S$ and let $E'$ be an $A_\infty$ minimal model for $E$. Then $E' $ satisfies the hypotheses of \ref{segalthm}, and using that along with the quasi-isomorphism invariance of the Koszul dual we see that $T\cong H^0(E^!)$. As in the proof of \ref{repug}, if $\Gamma$ is ungraded then we have an isomorphism $$\hom_{\mathrm{Ho}(\ubproart)}(B^\sharp E, \Gamma) \cong\hom_{\cat{aug.alg}_k}(H^0(E^!), \Gamma)\cong \hom_{\cat{aug.alg}_k}(T, \Gamma).$$
By \ref{cldefs} and \ref{gensegal} we have $$\defm^\mathrm{cl}_A(S)(\Gamma)\cong\hom_{\mathrm{Ho}(\ubproart)}(B^\sharp E, \Gamma) / (\text{inner automorphisms})$$and so we are done.
	\end{proof}

\begin{rmk}
	In particular, if $A$ is noetherian then $S$ admits some finitely generated projective resolution $\tilde S$, from which it is easy to see that $\hom_A(\tilde S, S)$ is finite-dimensional  in each degree. But the cohomology of $\hom_A(\tilde S, S)$ is exactly the Ext-algebra of $S$, and it follows that the conditions of the theorem are satisfied.
	\end{rmk}

\begin{rmk}
	Let $A$ be a $k$-dga and let $X$ be an $A$-module. Say that $X$ is \textbf{homotopy one-dimensional} if $H^*(X)\cong k[0]$. By taking truncations, it is easy to see that a homotopy one-dimensional module is necessarily quasi-isomorphic to a strictly one-dimensional module. In particular, if $A$ is an ungraded $k$-algebra and $X$ is an $A$-module, then $X$ is one-dimensional if and only if it is homotopy one-dimensional. It is not hard to check that all of our above results hold when $S$ is assumed to be homotopy one-dimensional, and not just one-dimensional: one can either rectify by passing to a strictly one-dimensional module, or observe that the proofs work just as well for homotopy one-dimensional modules.
	\end{rmk}

\subsection{The derived quotient}
We explore in more detail the links between one-dimensional modules and idempotents. We will give a deformation-theoretic meaning to Braun--Chuang--Lazarev's derived quotient \cite{bcl}, which is a way to quotient an algebra (or a dga) by an idempotent in a homotopically well-behaved manner.

\p  Let $A$ be a $k$-algebra with an idempotent $e\in A$. Suppose that the quotient $A/AeA$ is a local $k$-algebra with residue field $k$. Let $S$ be the quotient of $A/AeA$ by its radical. Because $A/AeA$ was local with residue field $k$, it follows that $S$ is a one-dimensional $A$-module, since $S$ is isomorphic to the residue field of $A/AeA$ as $A/AeA$-modules. It is easy to see that $A/AeA$ is an augmented $k$-algebra.

\begin{ex}
	Let $Q$ be a finite quiver with relations, and let $A$ be the path algebra of $Q$. Then $A$ comes with a set of primitive orthogonal idempotents $e_i$, one for each vertex of $Q$. By construction, the algebras $A/e_iA$ are local and the corresponding modules $S_i$ are precisely the vertex simple modules of $A$.
	\end{ex}
	
	If $A$ is a dga then we call a dga $B$ with a map $A \to B$ an $A$-dga.
	\begin{defn}[\cite{bcl}]
	Let $A$ be a $k$-algebra and let $e\in A$ be an idempotent. The \textbf{derived quotient} $\dq$ is the universal $A$-dga homotopy annihilating $e$. It is defined up to quasi-isomorphism of $A$-dgas.
\end{defn}
\begin{prop}[{\cite[3.10 and 3.4]{bcl}}]
	The derived quotient exists and is unique up to quasi-isomorphism of $A$-dgas. 
\end{prop}
We summarise the key properties we will need.
\begin{prop}[{\cite[4.15]{bcl}}]\label{dqemb}
	The map $A \to \dq$ induces an embedding of derived categories $D(\dq)\into D(A)$, with image those complexes $X$ such that $eX$ is acyclic.
\end{prop}
\begin{prop}[{\cite[3.2.4]{scatsviadq}}]\label{dqcon}
	The derived quotient of an ungraded algebra $A$ by an idempotent $e$ is a connective dga, with $H^0(\dq)\cong A/AeA$.
	\end{prop}
\begin{rmk}
	The derived quotient is a special case of the derived localisation, which can be constructed much more generally: see \cite{bcl} for the full theory.
\end{rmk}	

The following is a generalisation of an argument given in the proof of \cite[5.5]{kalckyang}.
\begin{prop}\label{kytype}
	Let $A$ be a $k$-algebra with an idempotent $e\in A$. Suppose that $A/AeA$ is a local $k$-algebra with residue field $k$, and let $S$ be the quotient of $A/AeA$ by its radical. Then $\dq$ is augmented, and there is a quasi-isomorphism $\R\enn_A(S)\simeq\dq^!$.
	\end{prop}
\begin{proof}
	By \ref{dqcon}, the derived quotient $\dq$ admits an algebra map $$\dq \to H^0(\dq)\cong A/AeA.$$Hence, because $A/AeA$ is augmented, $\dq$ is also augmented. For the second statement, we may regard $S$ as a module over $\dq$ and hence as an object of the derived category $D(\dq)$. By \ref{dqemb} we have an embedding $D(\dq)\into D(A)$ and we may hence compute $\R\enn_A(S)\simeq \R\enn_{\dq}(S)$. But now we can use \ref{kdisrend} to conclude that we have a quasi-isomorphism $\R\enn_{\dq}(S)\simeq\dq^!$ as required.
	\end{proof}
\begin{cor}\label{dqcor}
	Let $A$ be a $k$-algebra with an idempotent $e\in A$. Suppose that $A/AeA$ is a local $k$-algebra with residue field $k$, and let $S$ be the quotient of $A/AeA$ by its radical. Then there is a weak equivalence $$\frmdef_A(S)\simeq \R\mathrm{Map}_{\ubproart}(B^\sharp (\dq^!), -).$$
	\end{cor}
\begin{proof}
	Combine \ref{prorepfrm} with \ref{kytype}.
	\end{proof}

\begin{thm}[Prorepresentability by $\dq$]\label{almostdq}
		Let $A$ be a $k$-algebra with an idempotent $e\in A$. Suppose that $A/AeA$ is an Artinian local $k$-algebra and let $S$ be the quotient of $A/AeA$ by its radical. Suppose that $\dq$ is cohomologically locally finite. Then there is a weak equivalence 
	$$\frmdef_A(S)\simeq \R\mathrm{Map}_{\augdga^{\leq 0}}(\dq, -).$$
	\end{thm}
\begin{proof}
	Let $\Gamma$ be a connective Artinian dga. We have natural weak equivalences 
\begin{align*}
\frmdef_A(S)(\Gamma) & \simeq \R\mathrm{Map}_{\ubproart}(B^\sharp (\dq^!), \Gamma) & \text{by \ref{dqcor}}\\
& \simeq \R\mathrm{Map}_{\augdga^{\leq 0}}(\dq, \Gamma) & \text{by \ref{dermapscor}}
\end{align*}and so we are done.
	\end{proof}
Restricting to classical deformations, we obtain the following:
\begin{thm}[cf.\ {\cite[3.9]{contsdefs}}]\label{dwrep}
Let $A$ be a $k$-algebra with an idempotent $e\in A$. Suppose that $A/AeA$ is an Artinian local $k$-algebra and let $S$ be the quotient of $A/AeA$ by its radical. Suppose that $\dq$ is cohomologically locally finite. Then there is an isomorphism $$\defm^\mathrm{cl}_A(S)(\Gamma)\cong\frac{\hom_{\cat{aug.alg}_k}(A/AeA, \Gamma)}{ (\text{\normalfont inner automorphisms of }\Gamma)}.$$
	\end{thm}
\begin{proof}
	Using \ref{dqcor} then as in the proof of \ref{mysegal} we can identify $\defm^\mathrm{cl}_A(S)$ as the quotient of $\hom_{\cat{aug.alg}_k}(H^0(\dq^{!!}), -)$ by the inner automorphisms. But by \ref{dqcon} and \ref{kdfin} we have $H^0(\dq^{!!})\cong H^0(\dq)\cong A/AeA$.
	\end{proof}

\begin{rmk}
	The cohomology spaces of $\dq$ are completely known: aside from $H^0$ and $H^{-1}$, they are all of the form $\tor^{eAe}_*(Ae,eA)$ \cite[3.2.4]{scatsviadq}. Under some geometric hypotheses, we get cohomological finiteness of $\dq$ for free: if $A=\enn_R(R\oplus M)$ is a noncommutative partial resolution of a Gorenstein ring $R$, and across this isomorphism the idempotent $e$ corresponds to $\id_R$, then the results of \cite{scatsviadq} show that $\dq$ is cohomologically locally finite (indeed, its cohomology spaces are all Ext groups in the singularity category of $R$, which is hom-finite).
	\end{rmk}

\section{The universal prodeformation}
\p By taking homotopy limits, the formalism of functors valued in $\sset$ allows us to deform modules over pro-Artinian dgas. We set up the theory of prodeformations, and show that our prorepresentability statements give us a universal prodeformation. We do some explicit computations towards the identification of the universal prodeformation, and we finish by tracking this universal prodeformation across quasi-equivalences.
	
\subsection{Prodeformations}
	\begin{defn}
		Let $F:\dgart \to \sset$ be any functor. Denote by $\hat F$ the functor $\proart \to \sset$ which sends an inverse system $\{\Gamma_\alpha\}_\alpha$ to the homotopy limit $\holim_\alpha F(\Gamma_\alpha)$. Call $\hat F$ the \textbf{pro-completion} of $F$.
	\end{defn}
	\begin{defn}
		Let $A$ be a dga and let $X$ be an $A$-module. The functor of \textbf{prodeformations of $X$} is the functor $\prodef_A(X)$. The \textbf{set-valued functor of prodeformations of $X$} is the functor $\prodefset_A(X)\coloneqq \pi_0\prodef_A(X)$. A \textbf{prodeformation} of $X$ is an element of the set $\prodefset_A(X)$.
		\end{defn}
		It is easy to see that if $\Gamma $ is a connective Artinian local dga then $\prodef_A(X)(\Gamma)\simeq \sdefm_A(X)(\Gamma)$ and hence $\prodefset_A(X)(\Gamma)\cong \defm_A(X)(\Gamma)$. The reader is warned that $\prodefset_A(X)(\Gamma)$ may not agree with the inverse limit $\varprojlim_\alpha \defm_A(X)(\Gamma)$, as $\pi_0$ need not commute with homotopy limits; prodeformations are intrinsically homotopical in nature.

	\begin{rmk}\label{cdefsareprodefs}
	If $\Gamma=\{\Gamma_\alpha\}_\alpha$ is pro-Artinian then there is a weak equivalence
		$$\prodef_A(X)(\Gamma)\simeq \mathcal{W}\left(X_\mathrm{dg} \times^h_{D_\mathrm{dg}(A)} \holim_\alpha D_\mathrm{dg}(A\otimes \Gamma_\alpha)\right)$$given by passing the homotopy limit through $\mathcal{W}$ using \ref{wholims} and commuting homotopy limits with homotopy pullbacks. There is an obvious system of maps of dgas $\varprojlim \Gamma \to \Gamma_\alpha$ and this gives a dg functor $$D_\mathrm{dg}(A\otimes \varprojlim \Gamma) \to \holim_\alpha D_\mathrm{dg}(A\otimes \Gamma_\alpha).$$Although $\varprojlim \Gamma$ may not be Artinian, it is still augmented, so one can use it as a base for deformations. Extending the notation of \ref{sdefmdefn} in the obvious way, we hence obtain a map in the homotopy category of simplicial sets $$\sdefm_A(X)(\varprojlim \Gamma) \to \prodef_A(X)(\Gamma)$$ and so a deformation of $X$ over $\varprojlim \Gamma$ gives a prodeformation of $X$ over $\Gamma$. In particular, letting $\Lambda$ be a complete local augmented dga, we get a natural morphism $$\sdefm_A(X)(\Lambda) \to \holim_n\sdefm_A(X)(\Lambda / \mathfrak{m}^n)$$which should be relevant for a noncommutative version of Artin--Lurie representability, as in \cite{luriedagxiv,jonrepdstacks}.
		\end{rmk}

We are about to give a representability statement for prodeformations; before we do so we prove a subsidiary lemma.
\begin{lem}\label{holimpro}
Let $\Gamma=\{\Gamma_\alpha\}_\alpha$ be a pro-Artinian dga. Then there is a weak equivalence
$$\R\mathrm{Map}_{\ubproart}(-,\Gamma)\simeq \holim_\alpha\R\mathrm{Map}_{\ubproart}(-,\Gamma_\alpha).$$
	\end{lem}
\begin{proof}
	Let $\Gamma'$ be the cofiltered diagram of Artinian dgas that defines $\Gamma$, but regarded as a cofiltered diagram in $\ubproart$. The Yoneda lemma gives us an isomorphism $\varprojlim \Gamma' \cong \Gamma$. Moreover, in the exact same manner as the proof of \ref{limisexact} we see that $\holim \Gamma' \simeq \Gamma$ as well, because homotopy limits are just limits. Now use that $\R\mathrm{Map}$ commutes with homotopy limits in the second variable.
	\end{proof}

	\begin{prop}\label{prodefrep}
		Let $A$ be a dga and let $X$ be an $A$-module. Let $E\coloneqq \R\enn_A(X)$ be the endomorphism dga of $X$. Let $\Gamma\in \proart$. Then there is a canonical weak equivalence
		$$\prodef_A(X)(\Gamma)\simeq \R\mathrm{Map}_{\ubproart}(B_{\mathrm{nu}}^\sharp E,\Gamma).$$
	\end{prop}
\begin{proof}
	Put $\Gamma=\{\Gamma_\alpha\}_\alpha$. We have
	\begin{align*}
	\prodef_A(X)(\Gamma)\coloneqq& \holim_\alpha\sdefm_A(X)(\Gamma_\alpha) &\text{by definition}\\
	\simeq& \holim_\alpha\R\mathrm{Map}_{\ubproart}(B_{\mathrm{nu}}^\sharp E, \Gamma_\alpha) &\text{by \ref{proreps} levelwise}\\
	\simeq& \R\mathrm{Map}_{\ubproart}(B_{\mathrm{nu}}^\sharp E, \Gamma) &\text{by \ref{holimpro}}
	\end{align*}as required.
	\end{proof}
\begin{cor}\label{setprocor}
		Let $A$ be a dga and let $X$ be an $A$-module. Let $E\coloneqq \R\enn_A(X)$ be the endomorphism dga of $X$. Let $\Gamma\in \proart$. Then there is an isomorphism
	$$\prodefset_A(X)(\Gamma)\cong \hom_{\mathrm{Ho}(\ubproart)}(B_{\mathrm{nu}}^\sharp E,\Gamma).$$
	\end{cor}
There are analogous versions of the previous theorems for framed prodeformations.
			\begin{defn}
	Let $A$ be a dga and let $X$ be an $A$-module. The functor of \textbf{framed prodeformations of $X$} is the functor $\profrmdef_A(X)$. The \textbf{set-valued functor of framed prodeformations of $X$} is the functor $\profrmdefset_A(X)\coloneqq \pi_0\prodef_A(X)$. A \textbf{framed prodeformation} of $X$ is an element of the set $\profrmdefset_A(X)$.
\end{defn}
Again, $\frmdef$ agrees with $\profrmdef$ when the input is an Artinian local dga.
\begin{lem}\label{prohofiblem}
	Let $A$ be a dga and let $X$ be an $A$-module. Then there is a weak equivalence $$\profrmdef_A(X)\simeq \mathrm{hofib}\left(\prodef_A(X) \to \prodef_k(X)\right).$$
\end{lem}
\begin{proof}
	Commute the homotopy limit in the definition of $\profrmdef_A(X)$ through the homotopy fibre in the definition of $\frmdef_A(X)$.
\end{proof}

	\begin{prop}\label{frmprodefrep}
	Let $A$ be a connective dga and let $S$ be a one-dimensional $A$-module. Let $E\coloneqq \R\enn_A(S)$ be the endomorphism dga of $S$. Then $E$ is augmented. Let $\Gamma\in \proart$. Then there is a canonical weak equivalence
	$$\profrmdef_A(S)(\Gamma)\simeq \R\mathrm{Map}_{\ubproart}(B^\sharp E,\Gamma).$$
\end{prop}
\begin{proof}This is similar to the proof of \ref{prodefrep}: apply \ref{prorepfrm} levelwise and then use \ref{holimpro}.
\end{proof}
\begin{cor}\label{setfrmprocor}
	Let $A$ be a connective dga and let $S$ be a one-dimensional $A$-module. Let $E\coloneqq \R\enn_A(S)$ be the endomorphism dga of $S$. Then $E$ is augmented.  Let $\Gamma\in \proart$. Then there is an isomorphism
	$$\profrmdefset_A(X)(\Gamma)\cong \hom_{\mathrm{Ho}(\ubproart)}(B^\sharp E,\Gamma).$$
\end{cor}

\begin{rmk}\label{prodefob}
	It is worth discussing what type of object a prodeformation `is'. Using \ref{cdefsareprodefs}, then exactly as in \ref{cospanrmk} we obtain an injection $$\prodefset_A(X)(\Gamma) \into \{\text{homotopy classes of objects in }\holim_\alpha D_\mathrm{dg}(A\otimes \Gamma_\alpha)\}.$$Unfortunately, the author does not understand homotopy limits of dg categories well enough to simplify the right-hand object further. 
	\end{rmk}

\begin{rmk}\label{prodeflim}
	Let $\Gamma = \{\Gamma_\alpha\}_\alpha$ be a connective pro-Artinian dga and put $\Lambda\coloneqq  \varprojlim_\alpha \Gamma_\alpha$. As in \ref{cdefsareprodefs}, we have a map $$\defm_A(X)(\Lambda) \to \prodefset_A(X)(\Gamma)$$ and one could hope that most prodeformations of interest are in its image. Here is one possible construction of a right inverse. Given a prodeformation, use \ref{prodefob} to regard it as an object $\eta$ of $\holim_\alpha D_\mathrm{dg}(A\otimes \Gamma_\alpha)$. One can restrict along the structure maps of this homotopy limit to obtain a collection of objects $\eta_\alpha \in D_\mathrm{dg}(A\otimes \Gamma_\alpha)$. We can now discard the dg structure and simply work with triangulated categories. For brevity, write
	\begin{itemize} 
		\item $L_\alpha\coloneqq -\lot_\Lambda \Gamma_\alpha: D(A\otimes \Lambda) \to D(A\otimes \Gamma_\alpha)$
		\item $t_{\alpha\beta} \coloneqq -\lot_{\Gamma_\alpha} \Gamma_\beta:D(A\otimes \Gamma_\alpha) \to D(A\otimes \Gamma_\beta)$ for all $\alpha \to \beta$
		\item $R_\alpha$ for the right adjoint of $L_\alpha$, and $t_{\alpha\beta}^\perp$ for the right adjoint of $t_{\alpha\beta}$.
		\end{itemize}
	It is easy to see that if $\alpha \to \beta$, then we have $ t_{\alpha\beta}L_\alpha \cong L_\beta$. A calculation with adjoints and the Yoneda lemma now shows that we have $R_\beta\cong R_\alpha t_{\alpha\beta}^\perp$. Put $\zeta_\alpha \coloneqq R_\alpha \eta_\alpha.$ The unit $\id \to  t_{\alpha\beta}^\perp t_{\alpha\beta}$ combined with the fact that $ \eta_\beta \cong t_{\alpha\beta}\eta_\alpha$ gives us a map $\eta_\alpha \to t_{\alpha\beta}^\perp \eta_\beta$, and applying $R_\alpha$ now gives us a map $\zeta_\alpha \to \zeta_\beta$. Hence, the collection $\{\zeta_\alpha\}_\alpha$ forms a pro-object in $D(A\otimes \Lambda)$. Let $\zeta$ be its homotopy limit. Ideally, we would now like to prove that:
	\begin{enumerate}
		\item $\zeta$ is a deformation of $X$ over $\Lambda$
		\item the map $\defm_A(X)(\Lambda) \to \prodefset_A(X)(\Gamma)$ sends $\zeta$ to $\eta$.
		\end{enumerate}
	Unfortunately neither of these seem particularly easy to do.
	\end{rmk}

\subsection{Universal elements}

Now that we are able to work with prodeformations on the level of prorepresenting objects, we can immediately obtain a universal prodeformation. Since our deformation functors accept connective dgas as input, we need to truncate the representing object. Observe that if $P$ is any pro-Artinian dga and $\Gamma$ is a connective pro-Artinian dga then the proof of \ref{truncatedproreps} gives a natural equivalence $$\rmap_{\ubproart}(P, \Gamma)\simeq \rmap_{\proart}(\tau_{\leq 0}P,\Gamma)$$and hence there is no danger in truncation.

\begin{defn}
	Let $A$ be a dga and let $X$ be an $A$-module. Let $E\coloneqq \R\enn_A(X)$ be the endomorphism dga of $X$. The \textbf{universal prodeformation} of $X$ is the prodeformation of $X$ over $\tau_{\leq 0}B_{\mathrm{nu}}^\sharp E$ corresponding to the element $\id \in \hom_{\mathrm{Ho}(\proart)}(\tau_{\leq 0}B_{\mathrm{nu}}^\sharp E,\tau_{\leq 0}B_{\mathrm{nu}}^\sharp E)$ across the isomorphism of \ref{setprocor}.
	\end{defn}
\p One can also define the universal framed prodeformation:
\begin{defn}
	Let $A$ be a connective dga and let $S$ be a one-dimensional $A$-module. Let $E\coloneqq \R\enn_A(S)$ be the endomorphism dga of $S$, which is augmented. The \textbf{universal framed prodeformation} of $S$ is the framed prodeformation of $S$ over $\tau_{\leq 0}B^\sharp E$ corresponding to the element $\id \in \hom_{\mathrm{Ho}(\proart)}(\tau_{\leq 0}B^\sharp E,\tau_{\leq 0}B^\sharp E)$ across the isomorphism of \ref{setfrmprocor}.
\end{defn}
\begin{rmk}
	When $\tau_{\leq 0}B_{\mathrm{nu}}^\sharp E$ is weakly equivalent to an Artinian dga, then one could refer to the corresponding universal prodeformation as simply the \textbf{universal deformation}, and similarly for framed deformations.
	\end{rmk}

As in \ref{prodefob} and \ref{prodeflim}, it is not clear precisely what type of object the universal (framed) deformation is. However, in the one-dimensional case, in the spirit of \ref{prodeflim} one can identify a suitable candidate. If $A$ is a connective dga and $S$ is a one-dimensional $A$-module, put $E\coloneqq \R\enn_A(S)$. It is not hard to see that the $A$-$E^!$-bimodule $E^!$ is a deformation of $S$ over $E^!$. In view of this we make the following conjecture:
\begin{conj}\label{ufpdconj}
	Let $A$ be a connective dga and let $S$ be a one-dimensional $A$-module. Let $E\coloneqq \R\enn_A(S)$ be the derived endomorphism dga of $S$. Then the image of the deformation $E^!\in \frmdefset_A(S)(E^!)$ along the map $\frmdefset_A(S)(E^!) \to \profrmdefset_A(S)(B^\sharp E)$ of \ref{prodefob} is the universal framed prodeformation of $S$ as an $A$-module.
\end{conj}

In the rest of this section, we will make some progress towards identifying the universal framed prodeformation of a one-dimensional module, and hence the  resolution of \ref{ufpdconj}. We will first need a framed version of \ref{explicitwist}. Let $A$ be a dga and let $S$ be a one-dimensional $A$-module. Given an Artinian dga and a morphism $f:B^\sharp\R\enn_A(S) \to \Gamma$, then exactly as in \S5.6 one can associate a twisting cochain $\pi: \Gamma^* \to \R\enn_A(S)$. Hence one can associate a twist $S\otimes_\pi\Gamma$, which we still denote by $S \lot_f \Gamma$. This is necessarily a framed deformation. 

\begin{prop}\label{explicitwistfr}
	Let $A$ be a connective dga and let $S$ be a one-dimensional $A$-module. Let $\Gamma$ be an Artinian dga. Then the map \begin{align*}
		\hom_{\mathrm{Ho}(\ubproart)}(B^\sharp \R\enn_A(X), \Gamma) &\to \defm^\mathrm{fr}_A(S) \\
		f &\mapsto [\tilde S \otimes_f \Gamma]
	\end{align*}
	is a bijection.
\end{prop}
\begin{proof}
	This is similar to the proof of \ref{explicitwist} but uses \ref{prorepfrmset} instead.
\end{proof}

\begin{thm}\label{uprodefholimthm}
	Let $A$ be a connective dga and let $S$ be a one-dimensional $A$-module. Let $E\coloneqq\R\enn_A(S)$ be the derived endomorphism algebra of $S$ and write $B^\sharp E = \{\Gamma_\alpha\}_\alpha$. Let $f_\alpha:B^\sharp E \to \Gamma_\alpha$ be the canonical map. Then we have a quasi-isomorphism $$E^! \simeq \holim_\alpha\left(S\lot_{f_\alpha}\Gamma_\alpha\right)$$of $A$-$E^!$-bimodules.
	\end{thm}

\begin{proof}Let $\tilde S$ be a cofibrant resolution for $S$ and replace $E$ by $ \enn_A(\tilde S)$. Let $\pi: BE \to E$ denote the universal twisting cochain. Recall from \cite[6.2]{positselski} there is a functor $-\otimes^\pi BE$ going from $E$-modules to $BE$-comodules, constructed in a similar manner to the twist functor above. Moreover, by \cite[6.5]{positselski} it descends to a functor on derived categories. If $C \into BE$ is a sub-dgc, then we get an induced twisting cochain $\rho:C \to E$. It is easy to see that if $N$ is any $E$-module, we get an inclusion $N\otimes^{\rho}C \into N\otimes^\pi BE$. In particular, each $C_\alpha\coloneqq \Gamma_\alpha^*$ is a subdgc of $BE$, and moreover $BE$ is the colimit of the $C_\alpha$. Let $\pi_\alpha:C_\alpha \to E$ be the induced twisting cochain. We obtain a system of maps $N\otimes^{\pi_\alpha}C_\alpha \to N\otimes^\pi BE$. Because tensor products preserve colimits, and $BE$ is the colimit of the $C_\alpha$, we see that the natural map $$\varinjlim_\alpha \left( N\otimes^{\pi_\alpha}C_\alpha\right) \xrightarrow{\cong} N\otimes^\pi BE$$is an isomorphism. We may replace the colimit by the homotopy colimit (as it is a colimit along inclusions) and equivalently we have a quasi-isomorphism of $BE$-comodules $$\hocolim_\alpha \left( N\otimes^{\pi_\alpha}C_\alpha\right) \xrightarrow{\simeq} N\otimes^\pi BE.$$Dualising this we obtain a quasi-isomorphism $$\holim_\alpha \left( N\otimes^{\pi_\alpha}C_\alpha\right)^* \xleftarrow{\simeq} \left(N\otimes^\pi BE\right)^*$$of $E^!$-modules. Putting $N=\tilde S^*$, we see from \cite[6.5]{positselski} that $N\otimes^\pi BE\simeq BE$, and so the right-hand side is quasi-isomorphic to $E^!$. Moreover, for each $\alpha$ we have a quasi-isomorphism $$\left( N\otimes^{\pi_\alpha}C_\alpha\right)^* \simeq \tilde S\otimes_{\pi_\alpha}\Gamma_\alpha$$using that $\tilde S \simeq k$ and so we can pass the duals into the tensor product. But $\tilde S\otimes_{\pi_\alpha}\Gamma_\alpha$ is precisely $S\lot_{f_\alpha}\Gamma_\alpha$. The claim follows.
\end{proof}
In particular, the previous theorem says that the natural maps $\Gamma_\alpha \to S\lot_{f_\alpha}\Gamma_\alpha$ assemble into a quasi-isomorphism after taking the homotopy limit.

\begin{rmk}
 If the method outlined in \ref{prodeflim} works, then one can use it together with \ref{uprodefholimthm} to prove \ref{ufpdconj}.
\end{rmk}

\begin{rmk}
	One has an isomorphism of pro-objects $\hom_k(BE,E)\cong E \otimes B^\sharp E$. Taking the limit, one gets an isomorphism $$\hom_k(BE,E)\cong \varprojlim (E \otimes B^\sharp E) \eqqcolon E \widehat{\otimes} E^!$$between the convolution algebra and the completed tensor product. Loosely, this isomorphism sends the universal twisting cochain $\pi$ to the Casimir element $$\sum_e e\otimes e^* \in \mcs(E \widehat{\otimes} E^!)$$where we sum over a (possibly infinite!) basis of $E$ and $e^*$ denotes the dual basis vector corresponding to $e$. We can twist the differentials by $\pi$ to obtain an isomorphism $$\hom^\pi_k(BE,E)\cong \varprojlim( E \otimes_\pi B^\sharp E) \eqqcolon E \widehat{\otimes}_\pi E^!.$$As before, we can twist the differential on $\tilde S \widehat{\otimes} E^!$, and it is possible to show that $\tilde S \widehat{\otimes}_\pi E^!\simeq E^!$ as bimodules. In some sense, this is a computation of the universal prodeformation. As in \cite{negron} or \cite{herscovich}, the dga $\hom^\pi_k(BE,E)$ is quasi-isomorphic to the dga $\R\hom_{E^e}(E,E)$ which computes the Hochschild cohomology of $E$. This Hochschild dga should be thought of as the universal algebra deformation of $E$, and so we see that loosely, the universal deformation of $S$ is obtained by base change from the universal deformation of $E$.
\end{rmk}

		\subsection{Deformations and quasi-equivalences}
		Let $A$ and $B$ be two dgas. Suppose that $F:D(A) \to D(B)$ is a derived equivalence given by tensoring with an $A$-$B$-bimodule (this is not really a restriction as up to homotopy, all quasi-equivalences $D(A) \to D(B)$ are of this form \cite[Corollary 4.8]{toenmorita}). We analyse how deformations behave under $F$. Let $X$ be an $A$-module and put $Y\coloneqq F(X)$.  Observe that we can enhance $F$ to a quasi-equivalence $F:D_\mathrm{dg}(A) \to D_\mathrm{dg}(B)$ of dg categories.

		\begin{lem}\label{fdefwe}
			$F$ induces a weak equivalence
			$$F:\sdefm_A(X) \xrightarrow{\simeq} \sdefm_B(Y).$$
		\end{lem}
		\begin{proof}For each Artinian dga $\Gamma$, the quasi-equivalence $F$ induces a commutative diagram of dg categories
			$$\begin{tikzcd}
			X_\mathrm{dg} \ar[r,hook]\ar[d,"F"] & D_\mathrm{dg}(A)\ar[d,"F"] & D_\mathrm{dg}(A \otimes \Gamma)\ar[d,"F"] \ar[l]\\
			Y_\mathrm{dg} \ar[r,hook]& D_\mathrm{dg}(B) & D_\mathrm{dg}(B \otimes \Gamma) \ar[l]
			\end{tikzcd}$$where the vertical maps are quasi-equivalences, and in the third column we have used that $F$ is tensoring by a bimodule. So we get a quasi-equivalence $$F:\left(X_\mathrm{dg} \times^h_{D_\mathrm{dg}(A)} D_\mathrm{dg}(A \otimes \Gamma)\right) \xrightarrow{\simeq} \left(Y_\mathrm{dg} \times^h_{D_\mathrm{dg}(B)} D_\mathrm{dg}(B \otimes \Gamma)\right)$$between the homotopy pullbacks and hence a weak equivalence $F:\sdefm_A(X) \xrightarrow{\simeq} \sdefm_B(Y)$.
		\end{proof}

		\begin{lem}\label{fprodefwe}
			$F$ induces a weak equivalence $F:\widehat{\sdefm}_A(X) \xrightarrow{\simeq}	\widehat{\sdefm}_B(Y)$.
		\end{lem}
		\begin{proof}
			Take homotopy limits of \ref{fdefwe}.
		\end{proof}
		
		Put $E_A\coloneqq \R\enn_A(X)$ and $E_B\coloneqq \R\enn_B(Y)$. The component of $F$ at $X$ gives us a quasi-isomorphism $E_A \to E_B$, and hence a weak equivalence $\prodef_B(Y)(B^\sharp E_B) \to \prodef_B(Y)(B^\sharp E_A)$. Combined with the weak equivalence of \ref{fprodefwe}, we get a weak equivalence $\prodef_A(X)(B^\sharp E_A) \simeq \prodef_B(Y)(B^\sharp E_B)$ and it is easy to see that this preserves the universal prodeformation. We repeat our analysis in the framed setting.
		\begin{lem}\label{frmprodefwe}
			$F$ induces a weak equivalence $F:\profrmdef_A(X) \xrightarrow{\simeq}	\profrmdef_B(Y)$.
		\end{lem}
		\begin{proof}
			The weak equivalences of \ref{fprodefwe} commute with the forgetful functors and using \ref{prohofiblem} it follows that we get an induced weak equivalence between framed prodeformations.
		\end{proof}
		\begin{rmk}
			We could also have proved this by checking that we get a weak equivalence $F:\frmdef_A(X) \xrightarrow{\simeq} \frmdef_B(Y)$ and taking homotopy limits.
		\end{rmk}
		Our presumptive universal framed prodeformations are preserved by $F$:

		\begin{thm}\label{qeupdf}
			Let $A$ and $B$ be connective dgas. Let $S_A$ and $S_B$ be one-dimensional $A$ and $B$-modules respectively. Suppose that there is a derived equivalence $F:D(A) \to D(B)$ given by tensoring with a complex of bimodules, satisfying $F(S_A)\simeq S_B$. Put $U_A\coloneqq \R\enn_A(S_A)^!$ and $U_B\coloneqq \R\enn_B(S_B)^!$. Then there is a quasi-isomorphism of $B$-modules $F(U_A)\simeq U_B$.
		\end{thm}
	Before we begin the proof, we note that morally this is true because $U_A$ is the universal framed prodeformation of $S_A$, and similarly for $U_B$, and $F$ preserves the universal prodeformation. Indeed, this easily follows from the results above if \ref{ufpdconj} is true. The idea of our proof is essentially to use \ref{uprodefholimthm} and compare the homotopy limits across the two equivalent categories.
		\begin{proof}
			 As above, put $E_A\coloneqq \R\enn_A(S_A)$ and $E_B\coloneqq \R\enn_B(S_B)$. Put $V_A\coloneqq B^\sharp E_A$ and $V_B\coloneqq B^\sharp E_B$, so that one has $U_A\simeq \varprojlim V_A$ and $U_B\simeq \varprojlim V_B$. Because $F$ is a quasi-equivalence, one gets a quasi-isomorphism $E_A \to E_B$, hence a weak equivalence $V_B \to V_A$ and hence a quasi-isomorphism $U_B \to U_A$. Write $V_A = \{\Gamma_\alpha\}_\alpha$ with each $\Gamma_\alpha$ Artinian. If $f_\alpha: V_A \to \Gamma_\alpha$ is the canonical map, we obtain a map $g_\alpha: V_B \to \Gamma_\alpha$ by composition. Let $\eta^A_\alpha\coloneqq S_A\lot_{f_\alpha}\Gamma_\alpha \in D(A \otimes \Gamma_\alpha)$ and $\eta^B_\alpha\coloneqq S_B\lot_{g_\alpha}\Gamma_\alpha \in D(B \otimes \Gamma_\alpha)$ be the corresponding deformations. We have a commutative diagram $$\begin{tikzcd} D(A \otimes \Gamma_\alpha)\ar[r,"\id\otimes R_A"] \ar[d,"F\otimes\id"]& D(A\otimes U_A) \ar[d,"F\otimes G"]\\ D(B \otimes \Gamma_\alpha)\ar[r,"\id\otimes R_B"] & D(B\otimes U_B)
			\end{tikzcd}$$with vertical maps equivalences, where $R_A$ assigns a $\Gamma_\alpha$-module its underlying $U_A$-module, $R_B$ its underlying $U_B$-module, and $G$ assigns a $U_A$-module its underlying $U_B$-module.
		
		 Put $\zeta^A_\alpha\coloneqq R_A\eta^A_\alpha$. Exactly as in \ref{prodeflim}, as $\alpha$ varies the $\zeta^A_\alpha$ assemble into a pro-object $\zeta^A$ in $D(A\otimes U_A)$. By \ref{uprodefholimthm}, the homotopy limit of $\zeta^A$ is precisely the $A$-$U_A$-bimodule $U_A$. Similarly we get a pro-object $\zeta^B$ in $D(B\otimes U_B)$ whose homotopy limit is $U_B$. If $\Gamma$ is an Artinian dga and $\pi: \Gamma^* \to E_A$ is a twisting cochain, then we obtain by postcomposition a twisting cochain $\rho:\Gamma^* \to E_B$. Because the action of $E_B$ on $S_B$ is the same as the action of $E_A$ on $S_A$ via $F$ and $E_A \to E_B$, we see that $F(S_A\otimes_\pi \Gamma) \simeq S_B \otimes_\rho \Gamma$ as $B$-$\Gamma$-bimodules. In particular it follows that $F(\eta^A_\alpha)\simeq \eta^B_\alpha$ for all $\alpha$. It follows that $(F\otimes G)\zeta_\alpha^A \simeq \zeta_\alpha^B$ for all $\alpha$. Taking homotopy limits we hence have $U_B\simeq \holim_\alpha(F\otimes G)\zeta_\alpha^A$. But because $F\otimes G$ is an equivalence, it commutes with homotopy limits, and we have $$\holim_\alpha(F\otimes G)\zeta_\alpha^A\simeq (F\otimes G)\holim_\alpha\zeta_\alpha^A\simeq (F\otimes G)U_A.$$So $U_B\simeq (F\otimes G)U_A$ as $B$-$U_B$-bimodules. But forgetting the $U_B$-module structure, we have that $(F\otimes G)U_A \simeq F(U_A)$ as just $B$-modules.
		\end{proof}

		\section{Multi-pointed deformations}
		In this section, we adapt our arguments to work in the setting of multi-pointed deformation theory \cite{laudalpt, eriksen, kawamatapointed, contsdefs}; for brevity we will tend to drop the modifier `multi-'. Here, our base ring $k$ is no longer a field, but a semisimple ring of the form $\mathbb{F}^{\oplus n}$ with $\mathbb{F}$ a field and $n>1$. Note that a $k$-module has finite rank if and only if it is finite-dimensional over $\mathbb{F}$. Following Laudal \cite{laudalpt}, we refer to an augmented $k$-algebra as an $n$\textbf{-pointed} $\mathbb{F}$-algebra; note that the finite-dimensional $n$-pointed $\mathbb{F}$-algebras are precisely the (non-local) Artinian $\mathbb{F}$-algebras with precisely $n$ isomorphism classes of simple modules. 
		
		\p Let $e_i:k \to k$ be the projection onto the $i^\text{th}$ coordinate. Then $\{e_1,\ldots, e_n\}$ is a set of primitive orthogonal idempotents for $k$. If $A$ is a $k$-dga and $X$ is a right $A$-module then so is $X_i\coloneqq e_iX$, and we have $X \cong \oplus_{i=1}^nX_i$. Note also that $A_i\coloneqq e_iAe_i$ is a $\mathbb{F}$-dga, and that $X_i$ is a right $A_i$-module. An $n$-\textbf{pointed deformation of the collection} $\{X_1,\ldots,X_n\}$ is precisely a deformation of $X$ as a module over the $k$-dga $A$. Similar terms are defined analogously. The loose idea is that a pointed deformation of $\{X_1,\ldots,X_n\}$ is a compatible collection of deformations of the $A_i$-modules $X_i$, together with some extra data taking into account the $A$-linear morphisms between the $X_i$. In the derived setting, this corresponds to taking into account the higher Exts between the $X_i$.
		
		\p Our motivation for considering multi-pointed deformations is the following, after \cite{DWncdf}. Assume that $\mathbb{F}$ is an algebraically closed field of characteristic zero. Given a complete local model of a threefold flopping contraction $\pi: X \to \spec R$ over $\mathbb{F}$, one can construct a noncommutative algebra $A$ together with a derived equivalence $D^b(A)\simeq D^b(X)$ \cite{vdb}. The construction equips $A$ with an idempotent $e$, and one has $eAe\cong R$. If $E$ denotes the exceptional locus of $\pi$, then topologically $E$ is a tree of $n$ rational curves (as schemes, the curves themselves may be non-reduced). Donovan and Wemyss showed in \cite{contsdefs} that the finite-dimensional $k$-algebra $A_\con\coloneqq A/AeA$, known as the \textbf{contraction algebra} of $\pi$, controls the $n$-pointed deformations of the irreducible components of $E$. More precisely, let $C_i\cong \P^1$ be the $i^\text{th}$ irreducible component of $E$, equipped with the reduced scheme structure. Then across the equivalence $D^b(A)\simeq D^b(X)$ , the vertex simple modules $S_i\coloneqq (A_\con/\mathrm{rad}(A_\con))_i$ of the contraction algebra correspond to the structure sheaves $\mathcal{O}_{C_i}(-1)$. One says that a \textbf{noncommutative $n$-pointed deformation} of $E$ is precisely an $n$-pointed deformation of the collection $\{S_1,\ldots, S_n\}$. When $n=1$ - i.e. the exceptional locus is a single rational curve\footnote{Such flopping contractions are called \textbf{simple}.} - then \ref{dwrep} gives a precise sense in which the contraction algebra controls the noncommutative deformations of the exceptional locus.
		
		\p In \cite{me, scatsviadq} we consider the derived quotient $\dq$ as a \textbf{derived contraction algebra} of $\pi$. In the $n=1$ setting, then \ref{almostdq} gives a sense in which the derived contraction algebra controls the derived noncommutative deformations of $E$. Our motivation behind considering pointed deformation theory is to extend \ref{almostdq} to the pointed setting. In future work we will use our results to give the derived contraction algebra of a non-simple flop a derived deformation-theoretic meaning.
	
		\p In essence, any result that uses only the homological properties of $k$ remains true, as semisimple rings are precisely the rings over which every module is both projective and injective. Any result which uses the internal structure of $k$ probably becomes false as stated in the pointed setting, and requires some alterations to remain true.
		
		\subsection{Model structures and Koszul duality}
		Firstly we consider $\S\ref{modsn}$ and $\S\ref{kdsn}$, the results of which generalise straightforwardly. Everything in $\S\ref{modsn}$ goes through verbatim: the only thing that one must really check is that Pridham's proof \cite[4.3]{unifying} of \ref{proartmodel} adapts to the pointed case. In $\S\ref{kdsn}$, there are a few more things to check. One must first check that \ref{kdisrend} holds in the pointed setting: however \cite[\S19, Exercise 4]{fhtrht} holds over any base ring, so (in the notation used there) it suffices to check that the bar construction $B(A,A)$ is a cofibrant $A$-module, for which the usual proof suffices. At the level of generality in which we need it, the relevant $A_\infty$-algebra theory all works in the pointed setting; for a useful reference see \cite[\S2.4]{kalckyang3}. Everything else goes through verbatim.
		
		\subsection{Deformations and prorepresentability}

\p Now we consider $\S5$,where things start to get a little harder. We will not necessarily check that all of the statements made in $\S5$ remain true in the pointed setting: we will only use what we need to get analogues of the main results.

\p We first consider $\S5.1$. All of the basic definitions up to \ref{setdelqinvt} adapt to the pointed case. We won't check directly that \ref{setdelqinvt} or \ref{mcinvtqiso} adapt, as homotopy invariance for $\del$ and $\smc$ will follow from our later prorepresentability statements. The definition of $\pdf$ remains the same, and the Polynomial Poincar\'{e} Lemma \ref{poincare} remains true. It is easy to see that \ref{pi0del} adapts. The important thing for us to check is \ref{smcisdel}, which remains true since it is really a statement about the homotopy theory of simplicial groups \cite[2.20]{jonddefsartin}.

\p Everything in $\S5.2$ goes through in the pointed case; indeed the only $k$-linear results we use are Tabuada's \cite{tabuadamodel, tabuadadgvs}, where they are stated over arbitrary commutative base rings. As for $\S5.3$, observe that the only theorem here that needs to be checked is \ref{defisdel}; everything else follows. But it is not too hard to see that \cite[4.6]{jonddefsartin} adapts to the pointed setting; the key is that every $k$-module is flat.

\p As for $\S5.4$, it is not hard to check that the basic theory of twisting cochains adapts to the pointed case; the main thing to check is \ref{barcobaradj}, for which Loday--Vallette's proof \cite[2.2.6]{lodayvallette} adapts verbatim to the pointed setting. The proofs of \ref{mcprorep}, \ref{smcisrmap}, and \ref{proreps} all work in the pointed setting. Finally, all of the arguments of $\S5.5$ and $\S5.6$ work in the pointed setting.

\subsection{Framed deformations}
		Here is where we need to start making some changes. The results of $\S6.1$ hold in the pointed setting, as nowhere do we really use any special hypotheses on $A$ or $S$. If we keep the hypothesis that $S$ is one-dimensional, then Lemma \ref{extzlem} fails; we see that $\ext^0_A(S,S)\cong \mathbb{F}$. Hence $\R\enn_A(S)$ is $\mathbb{F}$-augmented, but not $k$-augmented (indeed if it were $k$-augmented then $\ext^0_A(S,S)$ must contain a copy of $k$). As a consequence, the rest of $\S6.2$ and $\S6.3$ both fail as written.
		
		\p The natural condition to put on $S$ in the $n$-pointed setting is that as a $k$-module, it is a copy of $k$. So in what follows, we will assume that $S\cong k$ and that the induced action of $k$ on $S$ is the identity. This forces $A$ to be $k$-augmented, and taking cohomology we see that $H^0(A)$ is $k$-augmented. Lemma \ref{extzlem} holds in this setting: the t-structure argument gives $\ext^0_A(S,S)\cong \enn_{H^0(A)}(S)$. But because $H^0(A)$ is $k$-augmented, every $H^0(A)$-linear endomorphism of $k$ necessarily factors through $H^0(A) \to k$. So one has $\enn_{H^0(A)}(S)\cong \enn_{k}(S)\cong k$, and the rest of the argument goes through. It is easy to now check that \ref{prorepfrm} and \ref{prorepfrmset} hold in the pointed setting; one has to use that $S$ is free of rank 1 over $k$ to get $\R\enn_k(S)\simeq k$.
		
		\p One must be more careful with \ref{gensegal}. The proof works up until one identifies $\ggr (k)(\Gamma)\cong 1 + \mathfrak{m}_\Gamma^0$ with the group of units $\Gamma^\times$: this no longer holds in the pointed setting. Indeed, everything in $1 + \mathfrak{m}_\Gamma^0$ is a unit, but the converse is not true: recall that $e_1,\ldots, e_n$ are the obvious set of primitive orthogonal idempotents for $k$. Then for example if $n=2$ we have $(e_1-e_2)^2=1$, and hence $e_1-e_2$ is a unit outside of $1 + \mathfrak{m}_\Gamma^0$. In general the group of units is $\Gamma^\times=(\Gamma^0)^\times=k^\times + \mathfrak{m}_\Gamma^0$, where in the first equation we use connectivity. The group of units of $k$ consists exactly of those elements $x=\sum_i\lambda_ie_i$ with all $\lambda_i \neq 0$; the inverse of $x$ is $\sum_i\lambda_i^{-1}e_i$. Say that a unit $u\in \Gamma^\times$ is \textbf{pointed} if it lies in the subgroup $1 + \mathfrak{m}_\Gamma^0$; we similarly say that an inner automorphism of $\Gamma$ is \textbf{pointed} if its corresponding unit is pointed. Then the proof of \ref{gensegal} gives in the pointed setting an isomorphism $$\defm_A(S)(\Gamma)\cong \frac{\hom_{\mathrm{Ho}(\ubproart)}(B^\sharp E, \Gamma) }{(\text{pointed inner automorphisms of }\Gamma)}.$$Theorem \ref{mysegal} holds in the pointed setting if one again replaces `inner automorphisms' with `pointed inner automorphisms'.
		
		\p We move on to $\S6.3$. Let $A$ be a $k$-algebra with an idempotent $e$. Let $S$ be the quotient of $A/AeA$ by its radical, and assume that $S\cong k$ as $k$-algebras. Braun--Chuang--Lazarev's derived quotient \cite{bcl} is defined over any base commutative ring, and so we have no problem defining $\dq$ in the pointed setting, and it has all of the properties we list. The proof of \ref{dqcon} goes through in the pointed setting: it follows easily from \cite[3.2.2]{scatsviadq}, which holds over any commutative base ring. It is easy to check that \ref{kytype}, \ref{dqcor}, and \ref{almostdq} hold in the pointed setting. Moreover, \ref{dwrep} holds in the pointed setting if, as before, one replaces `inner automorphisms' with `pointed inner automorphisms'.

		\subsection{Prodeformations}
		Finally, we consider $\S7$. The definition of prodeformations and framed deformations are the same as in the unpointed setting, and it is easy to see that all of the statements in $\S7.1$ hold in the pointed setting. The definition of the universal prodeformation in $\S7.2$ also remains the same, and the twisting result \ref{explicitwistfr} goes through. The proof of \ref{uprodefholimthm} works in the pointed setting: one only needs to check that the relevant parts of \cite{positselski} hold over semisimple rings. It is easy to check that all of the arguments of $\S7.3$ adapt to the pointed setting.

	\phantomsection
		
	\bibliographystyle{alpha}	
	\bibliography{thesisbib}

\end{document}